\documentclass[11pt,reqno]{amsart}
\usepackage[cmtip,all]{xy}
\usepackage{graphicx}
\usepackage{amssymb}
\usepackage{amsfonts}
\usepackage{amsmath}
\usepackage{mathrsfs}
\usepackage[usenames]{color}
\usepackage{hyperref}%

\newtheorem{theorem}{Theorem}[section]
\newtheorem{lemma}[theorem]{Lemma}
\newtheorem{proposition}[theorem]{Proposition}

\theoremstyle{definition}
\newtheorem{definition}{Definition}[section]

\theoremstyle{remark}
\newtheorem{remark}{Remark}[section]

\numberwithin{equation}{section}

\theoremstyle{remark}
\newtheorem{assumption}{Regularity assumption}[section]
\newtheorem{admissible}[assumption]{Admissibility condition}

\usepackage{commath}
\usepackage{xcolor}
\usepackage{graphicx}
\usepackage[toc,title,page]{appendix}
\usepackage{subcaption}
\usepackage{bbm}
\usepackage{cite}
\usepackage{color,soul}
\usepackage{bm}
\usepackage{mathrsfs}

\usepackage{tikz}
\usepackage{enumitem}
\hypersetup{colorlinks=true,linkcolor=blue,citecolor=blue,urlcolor=blue,pdfauthor={Lu Chen; Yan Jiang; Hongyu Liu; Catharine W. K. Lo; Longyue Tao},pdftitle={Determining evolutionary equations from a single passive boundary observation}}

\usepackage{needspace}

\title[Passive-observation inverse boundary problems]{Determining evolutionary equations from a single passive boundary observation}

\author{Lu Chen}
\address{School of Mathematics and Statistics, Beijing Institute of Technology, Beijing, China}
\email{chenlu5818804@163.com}

\author{Yan Jiang}
\address{Department of Mathematics, City University of Hong Kong, Kowloon, Hong Kong SAR, China}
\email{yjian24@cityu.edu.hk}

\author{Hongyu Liu}
\address{Department of Mathematics, City University of Hong Kong, Kowloon, Hong Kong SAR, China}
\email{hongyu.liuip@gmail.com; hongyliu@cityu.edu.hk}

\author{Catharine W. K. Lo}
\address{School of Mathematical Sciences, Shenzhen University, Shenzhen, Guangdong, China}
\email{cwklo@szu.edu.cn; catharinelowk@gmail.com}

\author{Longyue Tao}
\address{Department of Mathematics, City University of Hong Kong, Kowloon, Hong Kong SAR, China}
\email{longyue.tao@my.cityu.edu.hk; sdyctly@163.com}

\begin{document}

\date{November 17, 2026}


\begin{abstract}
We study inverse boundary problems (IBPs) for evolutionary partial differential equations from a single passive boundary observation generated by an unknown internal source in an unknown medium.
The goal is simultaneous recovery of the source and medium coefficients.
In stark contrast to active measurements, where designed inputs provide rich, structured data, passive-observation inverse boundary problems have minimal observable information and intrinsically coupled unknowns.
Consequently, these problems have been studied far less systematically than their active counterparts.

In this paper, we develop a {unified and novel framework} based on fundamental integral identities, integrating tools from harmonic analysis, interface asymptotics, and free-boundary regularity together with delicate low- and high-frequency asymptotics. This framework provides the first systematic resolution for a broad class of such inverse problems for {second-order hyperbolic, parabolic, and Schr\"odinger equations} within a single coherent approach. The key technical condition we impose is that the {cardinality of the measurement dataset exceeds that of the unknowns by at least one dimension}. This extra dimensional room enables us to decouple the unknowns and effectively handle inherent nonlinearity of the inverse problem. The unique identifiability results we obtain are arguably the best possible from a cardinality point of view. They not only subsume all existing results in the literature as special cases, but also cover significantly more general configurations of practical interest. Our framework complements several classical theories for IBPs and opens up a promising new direction with numerous potential developments.
\end{abstract}

\maketitle
\tableofcontents


\section{Introduction}\label{sec intr}

\subsection{Hyperbolic POIBPs and summary of major results}

We begin by introducing the inverse boundary problem for a second-order hyperbolic partial differential equation (PDE) from a single passive observation, which serves as a prototypical model for our study. 
Let \(n\geq 3\) and consider
\begin{equation}\label{eq:1}
\begin{cases}
\displaystyle \frac{1}{c(x)^2}\,\partial_t^2 u(x,t)-\nabla\!\cdot\!\bigl(\sigma(x)\nabla u(x,t)\bigr)=0, & (x,t)\in\mathbb R^n\times\mathbb R_+, \\[6pt]
u|_{t=0}=f(x), \quad \partial_tu|_{t=0}=h(x), & x\in\mathbb R^n.
\end{cases}
\end{equation}
Here \(f\) and \(h\) represent the initial displacement and initial velocity, while \(c(x)\) and \(\sigma(x)\) are positive scalar coefficients describing the medium.
We assume that there exists a bounded Lipschitz domain \(\Omega\subset\mathbb R^n\) with connected exterior such that
\begin{equation}\label{eq:general_support}
\operatorname{supp}(f)\cup\operatorname{supp}(h)\cup\operatorname{supp}(1-\sigma)\cup\operatorname{supp}(1-c)\subset\overline\Omega,
\end{equation}
and that there exist constants \(c_-,c_+,\sigma_-,\sigma_+>0\) with
\begin{equation}\label{eq:general_bounds}
0<c_-\leq c(x)\leq c_+<\infty,\qquad
0<\sigma_-\leq \sigma(x)\leq \sigma_+<\infty\quad\text{for }x\in\Omega.
\end{equation}
In this physical setup, \(\Omega\) represents a medium body embedded in a homogeneous background space, while the exterior region \(\mathbb R^n\setminus\Omega\) is uniformly homogeneous, with \(c(x)\equiv \sigma(x)\equiv 1\) for \(x\in\mathbb R^n\setminus\Omega\).
The coefficients \(c(x)\) and \(\sigma(x)\) jointly determine the local propagation speed \(c(x)\sqrt{\sigma(x)}\), with \(\sigma(x)^{-1}\) representing the density of the medium.
In the TAT/PAT case \(\sigma\equiv1\), the propagation speed reduces to \(c(x)\).
Under the assumptions of Subsection~\ref{subsec Forward Hy}, the Cauchy problem \eqref{eq:1} admits a unique global weak solution.

Starting from the unknown initial state, the resulting wave field propagates through the medium and reaches the boundary of the body, where one passively records the boundary response over time.
The inverse problem is to recover the interior physical quantities—namely the initial state and the medium parameters—from this passive boundary observation.
We refer to this inverse problem as the passive-observation inverse boundary problem (POIBP).

Mathematically, we introduce the associated boundary observation dataset by
\begin{equation}\label{eq:2}
{
\mathcal M_{f,h,c,\sigma}:=u\big|_{\partial\Omega\times\mathbb R_+}.
}
\end{equation}
\begingroup

The following two counterexamples support the use of full-time observations in \eqref{eq:2}.
Proposition~\ref{prop:observation-time-obstruction} rules out a uniform finite observation time for a family with no common positive lower speed bound.
Proposition~\ref{prop:persistent-boundary-echoes} exhibits arbitrarily late boundary echoes even under fixed positive speed bounds as in \eqref{eq:general_bounds}.
\par\endgroup
{
We write \(\nu\) for the unit outward normal to \(\partial\Omega\).
The observation \eqref{eq:2} also determines the physical conormal trace \(\sigma\partial_\nu u\) by Lemma~\ref{lem:boundary_flux_from_dirichlet}.
}
The inverse problem is to determine the initial data \((f,h)\) and the medium parameters \((c,\sigma)\) from this dataset, i.e.,
\begin{equation}\label{eq:ip1}
\mathcal{M}_{f,h,c,\sigma}(x,t),\quad (x,t)\in\partial\Omega\times\mathbb R_+
\quad \longrightarrow \quad
(f(x),h(x),c(x),\sigma(x)),\qquad x\in\Omega.
\end{equation}

Non-trapping is a sufficient condition for the low-frequency Fourier analysis in \cite{liu2015determining}, while \cite{kian2025determination} assumes exponential local energy decay.
Our uniqueness results for full-time observations require neither additional assumption.
The primary goal of this paper is to determine sufficient conditions ensuring the following one-to-one correspondence:
\begin{equation}\label{eq:unn1}
\mathcal{M}_{f,h,c,\sigma}(x,t)=\mathcal{M}_{\tilde f,\tilde h,\tilde c,\tilde\sigma}(x,t),\ \ (x,t)\in\partial\Omega\times\mathbb R_+\ \iff\ (f,h,c,\sigma)=(\tilde f,\tilde h,\tilde c,\tilde\sigma),
\end{equation}
where \((f,h,c,\sigma)\) and \((\tilde f,\tilde h,\tilde c,\tilde\sigma)\) are two admissible configurations.

Next, we state the main unique identifiability results obtained for the POIBP \eqref{eq:ip1}.
We first consider the relatively simpler case with \(\sigma\equiv 1\) and \(h\equiv 0\), which arises in thermoacoustic and photoacoustic tomography (TAT/PAT) \cite{liu2015determining}.

\begin{theorem}\label{thm:main1n}
Consider the TAT/PAT model \eqref{eq:1} with \(n=3\), \(h\equiv0\), and \(\sigma\equiv1\).
Let \((f,c)\) and \((\tilde f,\tilde c)\) belong to the same class in cases~\textup{(i)}--\textup{(iii)} below and satisfy the corresponding assumptions in Section~\ref{sec main}.
If their single passive boundary observations \eqref{eq:2} coincide, namely,
\[
\mathcal M_{f,0,c,1}(x,t)=\mathcal M_{\tilde f,0,\tilde c,1}(x,t),
\qquad (x,t)\in\partial\Omega\times\mathbb R_+,
\]
then
\[
f=\tilde f\quad\text{a.e. in }\Omega,
\qquad
c=\tilde c\quad\text{a.e. on }\operatorname{supp}(f).
\]
The cases are as follows, with \(x=(x',x_3)\):
\begin{enumerate}
\item[(i)] Separated-variable configurations with common prescribed longitudinal profiles:
\[
\left(f(x),\frac{f(x)}{c(x)^2}\right)
=\bigl(p(x')\tilde\phi(x_3),q(x')\phi(x_3)\bigr),
\]
where \(p,q\) are unknown.

\item[(ii)] Piecewise-regular configurations on unknown cylindrical components:
\[
(f(x),c(x)-1)
=\sum_{j=1}^{N}\bigl(p_j(x'),\gamma_j(x')-1\bigr)\chi_{\Omega_j}(x),
\qquad \Omega_j=\omega_j\times I_j,
\]
where the components and transverse profiles are unknown.

\item[(iii)] Piecewise-constant configurations on unknown domains:
\[
(f(x),c(x)-1)=\sum_{j=1}^{N}(f_j,c_j-1)\chi_{\Omega_j}(x),
\]
where both the domains \(\Omega_j\) and constants \(f_j,c_j\) are unknown.

\item[(iv)] Finite-time observations. If both configurations satisfy the hypotheses of Theorem~\ref{thm:finite-time-measurement}, the conclusions for cases~\textup{(i)}--\textup{(iii)} hold when the observations coincide on \(\partial\Omega\times(0,T)\) for sufficiently large \(T\).
\end{enumerate}
In cases~\textup{(ii)} and~\textup{(iii)}, we further have \(c=\tilde c\) a.e. in \(\Omega\), and the two configurations have the same number of components, whose supports coincide up to relabelling.

\end{theorem}


A crucial observation is that for cases~\textup{(i)}--\textup{(iii)}, the cardinality of the unknown configuration is \(n-1\), where by {\it cardinality} we mean the number of independent variables in a function quantity.
Here \(n=3\), and the count refers to the unknown profiles and interface parametrizations.
In cases~\textup{(i)} and~\textup{(ii)}, the unknown transverse profiles depend on \(x'\); case~\textup{(ii)} also includes unknown transverse boundaries and interval endpoints.
In case~\textup{(iii)}, the unknown coefficient values are constants, whereas each interface is locally represented by a function of \(n-1\) variables.
We would also like to note the cardinality of the dataset \(\mathcal{M}_{f,0,c,1}(x,t)\big|_{(x,t)\in\partial\Omega\times\mathbb{R}_+}\) is \(n\).
The unique identifiability results derived in Theorem~\ref{thm:main1n} are arguably the best possible one can achieve in this passive setting; see our discussion in Subsection \ref{subsec:framework}.

The detailed unique identifiability results summarized in Theorem~\ref{thm:main1n} are presented in Theorems~\ref{thm0} and \ref{thm:unique_f_c} for case (i), Theorem~\ref{thm:unique_f_c_piecewise} for case (ii), Theorem~\ref{thm:unique_f_c_piecewise_constant} for case (iii), and Theorem~\ref{thm:finite-time-measurement} for case (iv).
Case~\textup{(i)} also extends to dimensions \(n\geq3\), as noted in Remark~\ref{eq rem1}.

The counterexamples in Appendix~\ref{app:observation-time-examples} clarify the scope of the finite-time reduction in case~\textup{(iv)}.
Proposition~\ref{prop:observation-time-obstruction} rules out a uniform finite observation time for a family with no common positive lower bound on the sound speed; Proposition~\ref{prop:persistent-boundary-echoes} exhibits arbitrarily late boundary echoes even under fixed positive speed bounds.

For the POIBP \eqref{eq:ip1} with a general unknown configuration $(f,h,c,\sigma)$, the unique identifiability results obtained can be briefly summarized as follows. 
\begin{theorem}\label{thm:main2n}

Consider the general hyperbolic inverse problem \eqref{eq:1} with \(n=3\).
Suppose that the unknown configuration has the separating form
\[
(f(x),h(x),c(x),\sigma(x))
=(f_0(x_2),h_0(x_1),c_0(x_1),\sigma_0(x_1)),
\qquad x\in\Omega,
\]
where \(\Omega\subset\mathbb R^3\) is a known box and \(f_0,h_0,c_0,\sigma_0\) are unknown functions of one variable.
Let \((f,h,c,\sigma)\) and \((\tilde f,\tilde h,\tilde c,\tilde\sigma)\) be two such configurations satisfying the corresponding assumptions in Section~\ref{sec main}.
If their single passive boundary observations \eqref{eq:2} coincide, then
\[
f=\tilde f,\qquad h=\tilde h,\qquad c=\tilde c,\qquad
\sigma=\tilde\sigma\quad\text{a.e. in }\Omega.
\]

\end{theorem}

The detailed unique identifiability result summarized in Theorem~\ref{thm:main2n} is presented in Theorem~\ref{thm:general}.

\begingroup
The same unified framework also yields joint recovery results for the parabolic equation
\[
\begin{cases}
\displaystyle \frac{1}{\mu(x)}\partial_tu(x,t)-\Delta u(x,t)=0,
& (x,t)\in\mathbb R^3\times\mathbb R_+,\\[2mm]
u(x,0)=f(x),&x\in\mathbb R^3,
\end{cases}
\]
and the Schr\"odinger equation
\[
\begin{cases}
\mathrm{i}\partial_t\psi(x,t)+\Delta\psi(x,t)-V(x)\psi(x,t)=0,
& (x,t)\in\mathbb R^3\times\mathbb R_+,\\[2mm]
\psi(x,0)=f(x),&x\in\mathbb R^3.
\end{cases}
\]
Here \(\mu(x)>0\) is the diffusivity, \(V(x)\) is a real-valued potential, and \(f\) is the unknown initial state.
The passive observation is the boundary trace \(u|_{\partial\Omega\times\mathbb R_+}\) in the parabolic case and the complex boundary trace \(\psi|_{\partial\Omega\times\mathbb R_+}\) in the Schr\"odinger case.
Under the respective assumptions, equality of the single passive boundary observations gives the following joint recovery results:
\begin{enumerate}
\item[(i)] Parabolic equation (Theorem~\ref{ext:parabolic-model}):
\[
f=\tilde f\quad\text{a.e. in }\Omega,\qquad
\mu=\tilde\mu\quad\text{a.e. on }\operatorname{supp}(f).
\]
\item[(ii)] Schr\"odinger equation (Theorem~\ref{ext:schrodinger-model}):
\[
f=\tilde f\quad\text{a.e. in }\Omega,\qquad
V=\tilde V\quad\text{a.e. on }\operatorname{supp}(f).
\]
\end{enumerate}
\par
\endgroup

\subsection{Historical account and methodological background}

To appreciate the fundamental difficulty of the POIBP \eqref{eq:ip1}, it is instructive to contrast it with the corresponding inverse problem with multiple measurements, namely the active-measurement inverse boundary problem (AMIBP).
{
For fixed \((c,\sigma)\) and \(T>0\), let \(v_g\) solve the wave equation in \eqref{eq:1} on \(\Omega\times(0,T)\), with zero initial data and prescribed boundary value \(g\) vanishing near \(t=0\).
The hyperbolic Dirichlet-to-Neumann (DtN) map sends the boundary displacement to the outward conormal flux:
}
\begin{equation}\label{eq:sts}
{
\Lambda_{c,\sigma}^{T}:g\longmapsto
\sigma\partial_\nu v_g\big|_{\partial\Omega\times(0,T)}.
}
\end{equation}
{
In the active setting, the boundary input \(g\) is known and can be varied, and the map \eqref{eq:sts} is known for every admissible boundary input.
In the passive setting \eqref{eq:ip1}, there is no imposed boundary input: only the boundary displacement generated by one fixed, unknown initial pair \((f,h)\) is observed, and this pair is recovered together with the medium.
The active-measurement inverse problem is to recover the medium coefficients \((c,\sigma)\) from the full DtN map \(\Lambda_{c,\sigma}^{T}\), for a suitable observation time \(T\), i.e.,
}
\begin{equation}\label{eq:ip2}
{\Lambda_{c,\sigma}^{T}\longrightarrow (c,\sigma).}
\end{equation}

The AMIBP \eqref{eq:ip2} has been extensively and intensively studied in the literature, with a wealth of results. As we shall see, the POIBP \eqref{eq:ip1} is radically more challenging than the AMIBP, both physically and technically.

From a physical perspective, the AMIBP allows active design of inputs, producing rich, structured data whose cardinality exceeds that of the unknown.
{
Since \(c\) and \(\sigma\) are independent of time, the causal response kernel of \(\Lambda_{c,\sigma}^{T}\) depends on the source point \(y\in\partial\Omega\), the observation point \(x\in\partial\Omega\), and the time lag \(\tau=t-s\) between excitation and observation.
Thus, the active dataset has cardinality \(2(n-1)+1=2n-1\), while the unknown coefficients \((c(x),\sigma(x))\) depend on \(n\) spatial variables.
}
This higher cardinality provides an extra dimensional room that allows one to “squeeze” the inherent nonlinearity of the inverse problem—indeed, when viewed in a higher-dimensional space, nonlinear problems can be asymptotically or approximately linearized.
{For the single response in the POIBP, the dataset depends only on the boundary point and time, so its cardinality equals \(n\), matching that of the unknowns.}
The problem is therefore formally determined and inherently more nonlinear.
Moreover, it couples different types of unknowns in a nonlinear manner.

From a technical perspective, classical theories for active-data inverse boundary problems have been developed using sophisticated methodologies. Within the AMIBP framework, where multiple or infinitely many measurements are accessible, the elliptic prototype is Calderón's inverse boundary value problem, for which the Dirichlet-to-Neumann (DtN) map and the complex geometric optics (CGO) method provide the standard model \cite{calderon1980inverse,sylvester1987global,lee1989determining,astala2006calderon}. For evolutionary equations, this direction is represented by the boundary control method \cite{belishev1992reconstruction,belishev1987approach} and the boundary spectral method \cite{katchalov2004equivalence}. What these methods have in common is the richness of the available data: they generate sufficiently large probing families, and their observable datasets have higher cardinality than the unknown configuration. In the Calderón problem, for instance, the DtN map encodes data of cardinality \(2(n-1)\), which exceeds the cardinality \(n\) of the unknown coefficient when \(n\ge 3\). This extra dimensional room is what makes classical active‑measurement mechanisms effective. In the two-dimensional case \(n=2\), complex analysis and quasiconformal methods give uniqueness for uniformly positive, bounded measurable scalar conductivities. By contrast, in the passive‑observation setting considered here, one has neither adjustable inputs nor probing mechanisms built from products of two independent solutions, but only a single generated probing family associated with one passive boundary observation.

When only a single boundary measurement is available, a method based on Carleman estimates has been developed. Initiated by Bukhgeim and Klibanov, this approach has been elaborated in many works \cite{bukhgeim1981uniqueness,klibanov2004carleman,klibanov2021inverse,klibanov2006lipschitz}, and has been highly effective for source reconstruction when the medium is known. However, for the POIBP, where source terms and medium coefficients are intrinsically coupled in a single boundary observation dataset, the available information is too limited for the standard Carleman‑based framework to decouple and identify all components directly. Indeed, the Carleman approach typically assumes that either the source is known and the medium is unknown, or vice versa; it does not naturally handle the case where both are simultaneously unknown.

The first uniqueness result for simultaneously determining \(f\) and \(c\) in the TAT/PAT problem (as considered in Theorem~\ref{thm:main1n}) was established in \cite{liu2015determining}, where a Fourier‑based approach was initiated by Liu and Uhlmann.
This method relies on analyticity of the temporal Fourier transform of the wave field and low‑frequency asymptotics of the transformed field.
Moreover, because it discards high‑frequency spectral information, it can only recover \(f\) and \(c\) that satisfy stringent structural assumptions.
Subsequent works extended this direction but remained tied either to low‑frequency arguments or to additional structural or regularity assumptions.

For the wave equation and closely related TAT/PAT models, further simultaneous recovery results have been obtained for the source and the sound speed, or for closely related parameter pairs \cite{knox2020determining,kian2025determination,kian2026uniqueness}. We mention in passing some related unique recovery results by assuming one is known to recover the other one between $f$ and $c$ \cite{stefanov2013recovery}. For other passive‑observation models arising in geomagnetic anomaly detection, brain imaging, and related applications, related inverse formulations have also been studied in different physical settings \cite{deng2019identifying,deng2020identifying,deng2019inverse}. More recent works have established coupled recovery results for evolution equations under additional structural or model‑specific assumptions \cite{feizmohammadi2025reconstruction,liu2026inverse,kian2022simultaneous}.

At the same time, these studies also make clear a fundamental limitation of the passive setting: unlike the classical active‑measurement framework, one only has access to probing spaces generated by a single PDE, rather than to products of two independent solutions.
For this reason, the existing simultaneous recovery results remain tied either to low‑frequency arguments or to additional structural/regularity assumptions.
In particular, the first uniqueness result for the simultaneous recovery of \(f\) and \(c\) in TAT/PAT \cite{liu2015determining}, as well as its subsequent developments, relied on low‑frequency asymptotics and did not use the high‑frequency part of the spectral data.

In summary, classical active-measurement methods exploit higher data cardinality, while Carleman-based approaches require prior knowledge of either the source or the coefficient, and existing Fourier-based simultaneous recovery results discard high-frequency information. The present work introduces a unified framework that overcomes these limitations by combining integral identities, CGO test waves, and complementary low- and high-frequency asymptotics. Moreover, it uses interface asymptotics, unique continuation, and free-boundary regularity to tackle the piecewise-regular configurations. 
 In this way, our method complements the classical theories for inverse boundary problems and opens up a new field of research in passive-observation inverse problems for evolutionary PDEs.

\subsection{A novel unified framework}\label{subsec:framework}

\leavevmode{
Our unified framework combines integral identities with test functions adapted to different decoupling tasks.
By choosing the governing test equations, spatial profiles, and frequency regimes according to the quantities to be recovered, we extract complementary low- and high-frequency information from the transformed data of the same passive boundary observation.
The resulting relations are then combined to separate the coupled sources and medium coefficients.
For example, two configurations with the same observation \eqref{eq:2} satisfy the first integral identity
}
\begin{equation}\label{eq:td}
\begin{aligned}
&\int_\Omega\int_0^\infty
\left(\frac{1}{\tilde{c}^2}-\frac{1}{c^2}\right)\tilde{u}(x,t)\,\partial_{tt}w(x,t)\,\mathrm{d}t\mathrm{d}x
+\int_\Omega\left(\frac{\tilde{f}}{\tilde{c}^2}-\frac{f}{c^2}\right)\partial_t w(x,0)\,\mathrm{d}x\\
&\quad
-\int_\Omega\left(\frac{\tilde{h}}{\tilde{c}^2}-\frac{h}{c^2}\right)w(x,0)\,\mathrm{d}x
+\int_\Omega\int_0^\infty
(\tilde{\sigma}-\sigma)(x)\,\nabla\tilde{u}(x,t)\cdot\nabla w(x,t)\,\mathrm{d}t\mathrm{d}x
=0, 
\end{aligned}
\end{equation}
{with test waves \(w\) satisfying}
\begin{equation}\label{eq:tdtf}
\frac{1}{c(x)^2}\partial_t^2 w(x,t)-\nabla\cdot(\sigma(x)\nabla w(x,t))=0
\qquad\text{in }\Omega\times(0,\infty).
\end{equation}
{
For exponentially decaying test functions, this identity is valid in the weak formulation of Subsection~\ref{subsec Forward Hy}.
With suitable spatial tests, its low-frequency limit determines \(f/c^2\) in Theorem~\ref{thm0}.
After \(\sigma\) has been identified, the same identity supplies the complementary high- and low-frequency relations used to separate \(f\), \(h\), and \(c\) in Theorem~\ref{thm:general}.
}

{
The test equation is chosen according to the unknown to be separated, and each choice gives a corresponding integral identity.
To identify \(\sigma\), we take the coefficient of \(w_{tt}\) to be \(1\), as in \eqref{eq:free_wave_sigma2}.
The resulting identity \eqref{eq:second_identity_sigma} isolates \(\sigma-\tilde\sigma\) in the gradient pairing and puts the explicit factors involving \(c\) in the remainder terms; the same Green identity argument, applied to the high-frequency limiting problem with an adapted test family, then identifies \(\sigma\).
For joint recovery in TAT/PAT, the free spatial tests in \eqref{eq:v_helmholtz} yield the identity \eqref{eq:aux_int6}, which determines the transformed trace on an end face of the common longitudinal support interval.
The high-frequency limit of that trace supplies a weighted source relation complementary to the low-frequency quotient.
}

\Needspace{6\baselineskip}
{The four terms in the example \eqref{eq:td} correspond to mismatches in \(c\), \(f\), \(h\), and \(\sigma\).}
Their data cardinalities are:
\begin{itemize}
    \item \(I_1\) (c): \(\partial_{tt}w(x,t)\) integrated over time — cardinality \(n\) (since $w$ satisfies a PDE constraint).
    \item \(I_2\) (f): \(\partial_t w(x,0)\) — cardinality \(n-1\).
    \item \(I_3\) (h): \(w(x,0)\) — cardinality \(n-1\).
    \item \(I_4\) (\(\sigma\)): \(\nabla w(x,t)\) integrated over time — cardinality \(n\).
\end{itemize}
Without reduction, the unknowns \(f,h,c,\sigma\) each have cardinality \(n\). Terms \(I_2\) and \(I_3\) would therefore be underdetermined (data \(n-1\), unknown \(n\)). We thus impose dimensional reduction: all unknowns depend on only \(n-1\) variables. Then:

\begin{itemize}
    \item \(f\) and \(h\) (cardinality \(n-1\)) match \(I_2\) and \(I_3\).
    \item \(c\) and \(\sigma\) (cardinality \(n-1\)) are determined by \(I_1\) and \(I_4\), which retain cardinality \(n\) via time integration, providing the extra dimensional room to decouple the nonlinearity encoded in \(\tilde u\).
\end{itemize}

{
This one-dimensional excess allows us to separate the leading contributions from the asymptotically vanishing remainder terms. For sharply supported configurations, the leading interface contributions are retained and analyzed within the same integral identities.
}

These results are arguably the best one can achieve: they fully exploit the limited cardinality of the passive observation, reducing the unknowns to \(n-1\) dimensions while using the \(n\)-dimensional spacetime integral to capture the nonlinear coupling. Any further recovery would require either additional measurements or stronger a priori information.

The proofs of the main results are built upon three pillars:
\begin{enumerate}
    \item \textbf{Temporal Fourier analysis and spatial harmonic analysis.} We construct families of solutions of the form
    \[
    w(x,t) = e^{\mathrm{i}\rho_0 t}\Phi_\rho(x),\qquad \Phi_\rho(x) = e^{\mathrm{i}\rho'\cdot x}\bigl(1 + \psi_\rho(x)\bigr),
    \]
    where \(\rho=(\rho_0,\rho')\in\mathbb C\times\mathbb C^3\), \(\psi_\rho\) is the corrector, and the complex vector \(\rho'\) is chosen so that the product \(\rho'\cdot\rho' = 0\) (low-frequency regime) or satisfies a suitable dispersion relation (high-frequency regime). The corrector terms are controlled via elliptic estimates, and the temporal factor \(e^{\mathrm{i}\rho_0 t}\) and the spatial factor \(\Phi_\rho(x)\) are intricately coupled, enabling a delicate analysis that probes the medium at arbitrary frequencies while keeping the error terms asymptotically negligible.

    \item \textbf{Low- and high-frequency asymptotics.} {
The two frequency regimes provide complementary relations for the coupled unknowns.
In the TAT/PAT setting, low-frequency analysis determines \(f/c^2\), while the second integral identity and the high-frequency limit at an end face provide a different weighted source relation.
Combining these relations separates \(f\) and \(c\).
For the general hyperbolic system, high-frequency analysis first identifies \(\sigma\) and then yields two weighted source relations.
The low-frequency identity for \(h/c^2\), combined with these high-frequency relations, separates \(f\), \(h\), and \(c\).
In both settings, the low- and high-frequency information is extracted from the same passive boundary observation.
}

    \item \textbf{Interface asymptotics and free-boundary regularity for piecewise configurations.} High-frequency interface limits and unique continuation identify the support components in Section~\ref{sec:piecewise}. The free-boundary regularity argument in Subsection~\ref{subsec pice1} also uses the prescribed lack of boundary smoothness to identify a single component. These arguments recover the unknown supports and interfaces from the boundary observation.
\end{enumerate}

{
These three ingredients work together within a single framework: the test equations and test functions are chosen to target the quantities to be separated, and the resulting integral identities yield complementary low- and high-frequency relations.
Combining these relations recovers the coupled sources and medium coefficients, while interface asymptotics, unique continuation, and free-boundary regularity identify the unknown supports and interfaces in piecewise configurations.
}

This unified approach yields the unique identifiability results summarized in Theorem~\ref{thm:main1n} and its generalizations to the full system \((f,h,c,\sigma)\).
{
The parabolic and Schr\"odinger results in Theorems~\ref{ext:parabolic-model} and~\ref{ext:schrodinger-model} are further applications of this framework, with proofs given in Section~\ref{sec:parabolic}.
}


\subsection{Organization of the paper}

The remainder of the paper is structured as follows.
{
Section~\ref{sec main} states the main uniqueness results for the TAT/PAT model, piecewise-regular configurations, and the general second-order hyperbolic system.
It concludes with the joint recovery results for the parabolic and Schr\"odinger models.
}
{Section~\ref{sec:setup} develops the weak solutions of the hyperbolic Cauchy problem, their temporal Fourier transform, and the boundary Cauchy data determined by the observation.}
Section~\ref{sec:cgo} contains the proofs of the two TAT/PAT uniqueness results, where the unified framework is first implemented in a concrete model and already leads to the recovery of coupled unknowns from a single passive boundary observation.
Section~\ref{sec:piecewise} treats piecewise-regular configurations and shows that the same framework remains effective beyond the smooth setting.
Section~\ref{sec:generalhyperbolic} is devoted to the general second-order hyperbolic system, where the coupling is substantially stronger, the available information is more limited, and the same unified framework is pushed to recover four coupled parameters under suitable cardinality assumptions.
{
Finally, Section~\ref{sec:parabolic} proves the uniqueness results for the parabolic and Schr\"odinger models within the same framework.
}
\begingroup
Appendix~\ref{app:observation-time-examples} contains the complete proofs of Propositions~\ref{prop:observation-time-obstruction} and~\ref{prop:persistent-boundary-echoes} concerning the observation time.
\par\endgroup

\section{Statement of main results}\label{sec main}

\subsection{Joint recovery in TAT/PAT}

We first consider the case of thermoacoustic and photoacoustic tomography (TAT/PAT), which corresponds to taking \(n=3\), \(\sigma\equiv 1\), and \(h\equiv 0\) in \eqref{eq:1}:
\begin{equation}\label{eq:simple_hyper_eqns}
\begin{cases}
\displaystyle \frac{1}{c(x)^2}\,\partial_t^2 u(x,t)-\Delta u(x,t)=0, & (x,t)\in\mathbb R^3\times\mathbb R_+,\\[6pt]
u(x,0)=f(x),\quad \partial_t u(x,0)=0, & x\in\mathbb R^3.
\end{cases}
\end{equation}
{
In this setting, the boundary observation is \(u|_{\partial\Omega\times\mathbb R_+}\), in agreement with \eqref{eq:2}.
We abbreviate \(\mathcal M_{f,0,c,1}\) as \(\mathcal M_{f,c}\).
}
The inverse problem is to recover the source \(f\) in \(\Omega\) and the sound speed \(c\) on \(\operatorname{supp}(f)\) from this single passive boundary observation.
After translation and scaling, we work with the reference cube
\[
Q=[-\pi,\pi]^3\supset\overline{\Omega}.
\]

\subsubsection{Determining $f/c^2$ without prior knowledge of $c$}

Our first uniqueness result determines the quotient $f/c^2$ under the following separated-variable condition.

\Needspace{11\baselineskip}
\begin{admissible}\label{ass:x3_invariance_cf}
A configuration \((f,c)\) satisfies the separated-variable condition if the zero extension of \(f/c^2\) to \(Q\) has the form
\begin{equation}\label{eq:qphi1}
\frac{f(x)}{c(x)^2}=q(x')\,\phi(x_3)
\qquad\text{for a.e. }x=(x',x_3)\in Q,
\end{equation}
and the profile \(\phi\) satisfies the non-degeneracy condition
\begin{equation}\label{eq:nonde}
\int_{-\pi}^{\pi}\phi(s)e^{-\kappa_m s}\,\mathrm ds\ne0
\qquad\text{for every }m=(m_1,m_2)\in\mathbb Z^2,
\end{equation}
where \(\kappa_m=\sqrt{m_1^2+(m_2-\tfrac12)^2}\).
\end{admissible}

\begin{theorem}\label{thm0}
Let \((f,c)\) and \((\tilde f,\tilde c)\) be two TAT/PAT configurations satisfying
{
\(c,\tilde c\in L^\infty(\mathbb R^3)\), the support condition \eqref{eq:general_support}, and the uniform bounds \eqref{eq:general_bounds}.
}
Assume that the zero extensions of \(f/c^2\) and \(\tilde f/\tilde c^2\) to \(Q\) satisfy Admissibility~\ref{ass:x3_invariance_cf} with the same profile \(\phi\), so that, for almost every \(x=(x',x_3)\in Q\),
\[
\frac{f(x)}{c(x)^2}=q(x')\phi(x_3),
\qquad
\frac{\tilde f(x)}{\tilde c(x)^2}=\tilde q(x')\phi(x_3),
\]
where
\[
{
q,\tilde q\in L^2(\mathbb R^2),
\qquad
\phi|_{(a,b)}\in L^2(a,b),
}
\]
and \([a,b]\subset[-\pi,\pi]\) is the support interval of \(\phi\), with \(\phi\) extended by zero outside \([a,b]\).
If they generate the same passive boundary dataset,
\begin{equation}\label{eq:condition}
\mathcal{M}_{f,c}(x,t)=\mathcal{M}_{\tilde f,\tilde c}(x,t),
\qquad
(x,t)\in\partial\Omega\times\mathbb R_+,
\end{equation}
then the quotient \(f/c^2\) is uniquely determined:
\begin{equation}\label{eq:fc_ratio_unique}
\frac{f(x)}{c(x)^2}=\frac{\tilde f(x)}{\tilde c(x)^2}
\qquad\text{for a.e. }x\in\Omega.
\end{equation}
\end{theorem}

\begin{remark}\label{rem:chn1}
The characteristic function \(\phi=\chi_{[a,b]}\) satisfies the profile assumptions of Theorem~\ref{thm0}.
For a general domain \(\Omega\), the representation
\begin{equation}\label{eq:cc2}
\frac{f(x)}{c(x)^2}=q(x')\,\chi_\Omega(x),
\end{equation}
need not satisfy \eqref{eq:qphi1}, since the cross-sections of \(\Omega\) may vary with \(x'\).
\end{remark}

\begin{remark}
By rotational invariance of \eqref{eq:simple_hyper_eqns}, the separated-variable condition \eqref{eq:qphi1} may be imposed along any fixed direction common to the two configurations, with the corresponding non-degeneracy condition.
\end{remark}

Theorem~\ref{thm0} determines the quotient \(f/c^2\) from a single passive boundary observation.

\subsubsection{Second uniqueness result: simultaneously determining $f$ and $c$}

We next strengthen the preceding result to simultaneous recovery of both the source $f$ and the sound speed $c$.
{
The joint recovery result in Theorem~\ref{thm:unique_f_c} combines the recovered quotient with a high-frequency limit at an end face.
The source and medium perturbation share the known longitudinal support interval specified below.
}

\Needspace{5\baselineskip}
For this purpose, we impose one more structural admissibility condition:

\begin{admissible}\label{ass:x1_only_fc}
A configuration \((f,c)\) satisfies the separated-variable condition if
\begin{equation}\label{eq fphi2}
f(x)=p(x')\,\tilde\phi(x_3)
\qquad\text{for a.e. }x=(x',x_3)\in\mathbb R^3,
\end{equation}
where the profile \(\tilde\phi\) is prescribed on a known interval \([a,b]\subset[-\pi,\pi]\) and extended by zero outside it.
The profile is nonzero at at least one endpoint:
\begin{equation}\label{eq:TATPAT_endpoint_nondegeneracy}
|\tilde\phi(a)|+|\tilde\phi(b)|>0.
\end{equation}
The source and medium perturbation satisfy
\[
\operatorname{supp}(f)\cup\operatorname{supp}(1-c)
\subset\overline\Omega\cap\bigl(\mathbb R^2\times[a,b]\bigr).
\]
\end{admissible}

\begin{theorem}\label{thm:unique_f_c}
Let \((f,c)\) and \((\tilde f,\tilde c)\) be two TAT/PAT configurations satisfying the hypotheses of Theorem~\ref{thm0}.
Assume that both configurations satisfy the separated-variable condition (Admissibility~\ref{ass:x1_only_fc}) with respective transverse factors \(p,\tilde p\) and the same profile \(\tilde\phi\), where
\[
p,\tilde p\in L^2(\mathbb R^2),
\qquad
\tilde\phi\in C([a,b]),
\]
and \(c,\tilde c\) are continuous in \(\mathbb R^2\times(a,b)\), with continuous extensions from inside the strip to a common end plane at which \(\tilde\phi\) is nonzero.
If they generate the same passive boundary dataset,
\begin{equation*}
\mathcal{M}_{f,c}(x,t)=\mathcal{M}_{\tilde f,\tilde c}(x,t),
\qquad
(x,t)\in\partial\Omega\times\mathbb R_+,
\end{equation*}
then both the source and the sound speed are uniquely determined:
\begin{equation*}\label{eq:unique_f_c}
f(x)=\tilde f(x)\quad\text{for a.e. }x\in\Omega,
\qquad
c(x)=\tilde c(x)\quad\text{for a.e. }x\in\operatorname{supp}(f).
\end{equation*}
\end{theorem}

The proof is given in Section~\ref{sec:cgo}.

\begin{remark}\label{eq rem1}
The same arguments apply in dimensions \(n\geq3\) under the corresponding separated-variable assumptions.
\end{remark}

\begin{remark}\label{rem regular}
Theorems~\ref{thm0} and~\ref{thm:unique_f_c} use prescribed longitudinal profiles and allow \(L^2\) transverse source factors.
Joint recovery additionally requires a continuous source profile with a nonzero endpoint value and the continuity assumptions on the sound speeds stated in Theorem~\ref{thm:unique_f_c}.
{
These two separating conditions are compatible with a nonconstant sound speed.
To see this, choose a domain \(D\subset\mathbb R^2\) and an interval \([a,b]\) such that \(\overline D\times[a,b]\Subset\Omega\).
Let \(0\ne p\in C_c^\infty(D)\) and let \(v>0\) be smooth on \(\mathbb R^2\), with \(v-1\in C_c^\infty(D)\).
Define
\[
f(x)=p(x')\chi_{[a,b]}(x_3),\qquad
c(x)=1+\bigl(v(x')-1\bigr)\chi_{[a,b]}(x_3).
\]
Then \(f/c^2=(p/v^2)\chi_{[a,b]}\), so Admissibility~\ref{ass:x3_invariance_cf} holds with \(\phi=\chi_{[a,b]}\), while Admissibility~\ref{ass:x1_only_fc} holds with \(\tilde\phi\equiv1\) on \([a,b]\).
The Laplace moments in \eqref{eq:nonde} are positive, and the source profile is nonzero at both end faces.
The speed is continuous inside the strip and has the required one-sided extensions to its end planes, even if it jumps across them.
Pairs obtained by choosing different \(p\) and \(v\) share the prescribed longitudinal profiles.
}
When the interval \(I\) in \(f(x)=p(y)\chi_I(s)\) is unknown, its endpoints become additional geometric unknowns.
More generally, piecewise configurations involve unknown support components and interfaces:
\[
f=\sum_i f_i\chi_{\Omega_i},
\qquad
c=1+\sum_i(c_i-1)\chi_{\Omega_i}.
\]

Theorem~\ref{thm:unique_f_c_piecewise} and Section~\ref{sec:piecewise} address these geometric unknowns within the same framework.
Support identification and unique continuation enable recovery of the source and sound speed on the identified components.
The required continuity and local H\"older regularity are imposed componentwise, allowing jumps across the unknown interfaces.
\end{remark}

\subsubsection{Joint recovery for piecewise-regular configurations}\label{subsec piece}

{
The TAT/PAT results presented above use prescribed profiles in the distinguished variable.
We now consider piecewise-regular configurations with unknown support components and interfaces, as discussed in Remark~\ref{rem regular}.
}

The following admissibility condition specifies the geometry, componentwise structure, and regularity of the configurations.
The proof is given in Section~\ref{sec:piecewise}.

\begin{admissible}\label{ass:piecewise_jump}
Assume that, for some \(m\geq1\) and \(0<\mu<1\), the following piecewise conditions hold.

\smallskip
\noindent
{\rm (i) Geometric structure.}
After one fixed rigid change of coordinates, all components have the product form
\[
\Omega_i=\omega_i\times I_i,\qquad
\widetilde\Omega_j=\widetilde\omega_j\times\widetilde I_j,
\]
where \(\omega_i,\widetilde\omega_j\subset\mathbb R^2\) are bounded simply connected \(C^{1,1}\) domains and \(I_i,\widetilde I_j\subset\mathbb R\) are bounded open intervals.
{All components are compactly contained in \(\Omega\), and their closures are pairwise disjoint within each configuration.}
Moreover,
\[
\Omega\setminus\overline{\bigcup_{i=1}^N\Omega_i},
\qquad
\Omega\setminus\overline{\bigcup_{j=1}^M\widetilde\Omega_j}
\]
are connected.
The transverse boundaries \(\partial\omega_i\) and \(\partial\widetilde\omega_j\) are nowhere locally of class \(C^{m+1,\mu}\).

\smallskip
\noindent
{\rm (ii) Piecewise representation and regularity.}
The coefficients admit the representations
\begin{equation}\label{PiecewiseForm}
f=\sum_{i=1}^N f_i\chi_{\Omega_i},
\qquad
c=1+\sum_{i=1}^N (c_i-1)\chi_{\Omega_i},
\end{equation}
and
\begin{equation}\label{PiecewiseFormTilde}
\tilde f=\sum_{j=1}^M \tilde f_j\chi_{\widetilde\Omega_j},
\qquad
\tilde c=1+\sum_{j=1}^M (\tilde c_j-1)\chi_{\widetilde\Omega_j}.
\end{equation}
{In addition to the product geometry, assume the following componentwise structural condition:
\[
f_i(x',x_3)=p_i(x'),\qquad c_i(x',x_3)=\gamma_i(x')
\quad\text{in }\Omega_i=\omega_i\times I_i,
\]
and
\[
\tilde f_j(x',x_3)=\tilde p_j(x'),\qquad
\tilde c_j(x',x_3)=\tilde\gamma_j(x')
\quad\text{in }\widetilde\Omega_j=\widetilde\omega_j\times\widetilde I_j.
\]
The transverse profiles and the intervals are unknown, and the intervals may differ between components and between the two configurations.}
For every \(i=1,\ldots,N\) and \(j=1,\ldots,M\),
{
\[
f_i\in C(\overline{\Omega_i}),
\qquad
\tilde f_j\in C(\overline{\widetilde\Omega_j}),
\]
and
\[
c_i\in C(\overline{\Omega_i}),
\qquad
\tilde c_j\in C(\overline{\widetilde\Omega_j}).
\]
}
Moreover, the functions \(f_i,c_i\) admit \(C^{m,\mu}\) interior extensions up to the lateral boundary \(\partial\omega_i\times I_i\), and \(\tilde f_j,\tilde c_j\) admit \(C^{m,\mu}\) interior extensions up to the lateral boundary \(\partial\widetilde\omega_j\times\widetilde I_j\).

When extended by zero outside their components,
\[
\operatorname{supp}(f_i)=\operatorname{supp}(c_i-1)=\overline{\Omega_i},
\qquad
\operatorname{supp}(\tilde f_j)=\operatorname{supp}(\tilde c_j-1)=\overline{\widetilde\Omega_j}.
\]
Finally, the lateral jump conditions hold:
\[
c_i\neq1,\qquad f_i>0
\qquad
\text{on }\partial\omega_i\times I_i,
\]
and
\[
\tilde c_j\neq1,\qquad \tilde f_j>0
\qquad
\text{on }\partial\widetilde\omega_j\times\widetilde I_j.
\]
\end{admissible}

\begin{theorem}\label{thm:unique_f_c_piecewise}
Let \((f,c)\) and \((\tilde f,\tilde c)\) be two TAT/PAT configurations satisfying \eqref{eq:general_support} and \eqref{eq:general_bounds}{ together with the componentwise structural, regularity, and geometric conditions in Admissibility~\ref{ass:piecewise_jump}.}
If they generate the same passive boundary dataset,
\begin{equation}\label{eq measure 2}
\mathcal M_{f,c}(x,t)=\mathcal M_{\tilde f,\tilde c}(x,t)
\qquad
(x,t)\in\partial\Omega\times\mathbb R_+,
\end{equation}
then the two families of subdomains coincide up to relabelling:
\[
\{\Omega_1,\dots,\Omega_N\}=\{\widetilde\Omega_1,\dots,\widetilde\Omega_M\},
\qquad M=N.
\]
Moreover,
\[
f(x)=\tilde f(x),
\qquad
c(x)=\tilde c(x)
\quad\text{for a.e. }x\in\Omega.
\]
\end{theorem}

\begin{remark}\label{rem:piecewise-x3-profile}
\begingroup
The separated-variable conditions in Admissibilities~\ref{ass:x3_invariance_cf} and~\ref{ass:x1_only_fc} use a common prescribed one-dimensional profile for the two configurations.
In Theorem~\ref{thm:unique_f_c_piecewise}, the componentwise intervals are unknown and may differ between components and between the two configurations.
Support identification gives a common interval on each pair of matched components, after which recovery proceeds componentwise as shown in Subsection~\ref{subsec pice2}.
\endgroup
\end{remark}

\begin{remark}
An important special case is that of piecewise constant sound speeds: If \(c_i, \tilde c_j \in \mathbb R\) with \(c_i \neq 1\) and \(\tilde c_j \neq 1\) for all \(i,j\), {then Theorem~\ref{thm:unique_f_c_piecewise} applies, provided that its remaining hypotheses hold, including the componentwise source structure in Admissibility~\ref{ass:piecewise_jump}.}
\end{remark}

\subsubsection{Joint recovery for piecewise-constant configurations on general domains}\label{subsec piecewise constant}

The preceding piecewise-regular result was formulated under a product-type geometric structure for the support components.
We now treat a complementary piecewise-constant case.
Here the unknown source and sound speed are constant on each component, but the components themselves are allowed to be general bounded Lipschitz domains.
Thus the geometric setting is no longer restricted to product-type supports.
This provides a natural model for piecewise homogeneous media with jumps across general Lipschitz interfaces, and we shall show that both the support components and the corresponding constants are uniquely determined.

\begin{admissible}\label{ass:piecewise_constant_general_domains}
Assume that the following conditions hold.

\smallskip
\noindent
{\rm (i) Geometry.}
There exist bounded connected Lipschitz domains
\[
\Omega_i,\widetilde\Omega_j\Subset\Omega,
\qquad i=1,\ldots,N,\quad j=1,\ldots,M,
\]
{such that the closures in each family are pairwise disjoint, and}
\[
\Omega\setminus\overline{\bigcup_{i=1}^N\Omega_i},
\qquad
\Omega\setminus\overline{\bigcup_{j=1}^M\widetilde\Omega_j}
\]
are connected.

\smallskip
\noindent
{\rm (ii) Piecewise constant structure.}
There exist constants
\[
f_i,\ c_i,\ \tilde f_j,\ \tilde c_j,
\qquad i=1,\ldots,N,\quad j=1,\ldots,M,
\]
such that
\begin{equation}\label{PiecewiseConstantForm}
f=\sum_{i=1}^N f_i\chi_{\Omega_i},
\qquad
c=1+\sum_{i=1}^N(c_i-1)\chi_{\Omega_i},
\end{equation}
and
\begin{equation}\label{PiecewiseConstantFormTilde}
\tilde f=\sum_{j=1}^M \tilde f_j\chi_{\widetilde\Omega_j},
\qquad
\tilde c=1+\sum_{j=1}^M(\tilde c_j-1)\chi_{\widetilde\Omega_j}.
\end{equation}
The constants satisfy
\[
f_i\neq0,\qquad \tilde f_j\neq0,
\qquad
c_i>0,\quad \tilde c_j>0,
\qquad
c_i\neq1,\quad \tilde c_j\neq1.
\]
\end{admissible}

\begin{theorem}\label{thm:unique_f_c_piecewise_constant}
Let \((f,c)\) and \((\tilde f,\tilde c)\) be two TAT/PAT configurations satisfying \eqref{eq:general_support} and \eqref{eq:general_bounds}{ and the piecewise constant admissibility condition on general domains (Admissibility~\ref{ass:piecewise_constant_general_domains}).}
If they generate the same passive boundary dataset,
\begin{equation}\label{eq measure piecewise constant}
\mathcal M_{f,c}(x,t)=\mathcal M_{\tilde f,\tilde c}(x,t)
\qquad
(x,t)\in\partial\Omega\times\mathbb R_+,
\end{equation}
then the support components coincide up to relabelling:
\[
M=N,
\qquad
\{\Omega_1,\ldots,\Omega_N\}
=
\{\widetilde\Omega_1,\ldots,\widetilde\Omega_M\}.
\]
Moreover, after relabelling the components so that
\[
\Omega_i=\widetilde\Omega_i,
\qquad i=1,\ldots,N,
\]
the corresponding constants agree:
\[
f_i=\tilde f_i,
\qquad
c_i=\tilde c_i,
\qquad i=1,\ldots,N.
\]
Consequently,
\[
f(x)=\tilde f(x),
\qquad
c(x)=\tilde c(x)
\quad\text{for a.e. }x\in\Omega.
\]
\end{theorem}
\medskip

\subsection{Simultaneous recovery for the general hyperbolic system}
We now return to the general second-order hyperbolic system \eqref{eq:1}, where the unknown quantities are the initial displacement \(f\), the initial velocity \(h\), the coefficient \(c\), and the conductivity coefficient \(\sigma\). {
The passive boundary observation is the displacement trace \eqref{eq:2}.
} Under suitable admissibility conditions, these four quantities are uniquely determined by a single passive boundary observation.

{
As in the TAT/PAT case, we impose separated-variable conditions on the unknowns (see Admissibility condition~\ref{ass:general_structure}), which reduce the effective dimensionality of the unknowns and facilitate their decoupling.
}

The general four-parameter system exhibits a more intricate coupling of unknowns than the TAT/PAT case.
{
In particular, the unknown conductivity enters the principal part of the equation, while both the initial displacement and the initial velocity are coupled with the coefficient \(c\).
With the common support prescribed, we concentrate on the core result: the simultaneous recovery of the four functions $f$, $h$, $c$, and $\sigma$ within the unified framework under the stated structural assumptions.
}

With this preparation, we now present the precise formulation of our main results for the general hyperbolic system.

\medskip
\subsubsection{The general hyperbolic system}
\leavevmode{We first state the admissibility condition.}
The sources and coefficient perturbations are extended by zero outside the support sets.

\begin{admissible}\label{ass:general_structure}
{
Fix \(\Omega=(a_1,b_1)\times(a_2,b_2)\times(a_3,b_3)\subset\mathbb R^3\), where \(a_j<b_j\), \(j=1,2,3\).
A configuration \((f,h,c,\sigma)\) satisfies the admissibility condition if, for a.e. \(x=(x_1,x_2,x_3)\in\mathbb R^3\),
}
\[
\begin{alignedat}{2}
f(x)&=f_0(x_2)\chi_\Omega(x),\qquad &h(x)&=h_0(x_1)\chi_\Omega(x),\\
c(x)&=1+\bigl(c_0(x_1)-1\bigr)\chi_\Omega(x),\qquad &\sigma(x)&=1+\bigl(\sigma_0(x_1)-1\bigr)\chi_\Omega(x),
\end{alignedat}
\]
{
for some real-valued functions \(f_0,h_0,c_0,\sigma_0\), with
}
\[
\operatorname{supp}(f)=\operatorname{supp}(h)=\operatorname{supp}(1-c)=\operatorname{supp}(1-\sigma)=\overline{\Omega}.
\]
{
We further require \(f_0'(b_2)\ne0\), where the derivative is taken from inside \((a_2,b_2)\).
}
\end{admissible}


We now state our main uniqueness result for the general hyperbolic system.

\begin{theorem}\label{thm:general}
{
Let \((f,h,c,\sigma)\) and \((\tilde f,\tilde h,\tilde c,\tilde\sigma)\) be two configurations for the hyperbolic system \eqref{eq:1} satisfying the uniform bounds \eqref{eq:general_bounds} and Admissibility condition~\ref{ass:general_structure}, with the interior regularity
}
{
\[
f,\tilde f,h,\tilde h\in H^3(\Omega),\qquad
c,\tilde c\in W^{2,\infty}(\Omega),\qquad
\sigma,\tilde\sigma\in W^{4,\infty}(\Omega).
\]
}
If they generate the same passive boundary dataset,
\[
\mathcal{M}_{f,h,c,\sigma}(x,t)=\mathcal{M}_{\tilde f,\tilde h,\tilde c,\tilde\sigma}(x,t),\qquad (x,t)\in\partial\Omega\times\mathbb R_+,
\]
then
\[
f=\tilde f,\quad h=\tilde h,\quad c=\tilde c,\quad \sigma=\tilde\sigma\qquad\text{a.e. }x\in\Omega.
\]
{
{
Here the observed displacement trace is understood in the weak sense specified in Subsection~\ref{subsec:weak-boundary-data}.
}
}
\end{theorem}

\begin{remark}
{
Theorem~\ref{thm:general} is stated in dimension three.
}
{
The unified framework can be extended to study joint recovery in higher dimensions and for more complex configurations, by adapting the test families and decoupling arguments to the geometry and coefficient structure of the model.
}
\end{remark}

{
The proof of Theorem~\ref{thm:general} follows the unified framework of Subsection~\ref{subsec:framework}.
The four unknown quantities are decoupled using the two integral identities and three families of test functions.
The first family identifies \(\sigma\) by applying the Green identity argument underlying the second integral identity to the high-frequency limiting problem.
With \(\sigma\) fixed, the second family extracts weighted information on \(f\) and \(h\) from the first integral identity at high frequency, while the third yields \(h/c^2\) through the low-frequency limit and zero-frequency CGO analysis.
Combining these relations determines \(c\), \(h\), and \(f\).
}
The detailed proof is given in Section~\ref{sec:generalhyperbolic}.

\medskip
\begingroup

\subsection{The role of the observation time}
The spatial admissibility conditions alone do not justify shortening the observation interval in \eqref{eq:2}.
The next two propositions exhibit different obstructions: failure of uniform quotient recovery over a family without a common positive lower bound on the speed, and absence of a finite cutoff even under fixed positive speed bounds.

\Needspace{8\baselineskip}
\begin{proposition}\label{prop:observation-time-obstruction}
Let \(\Omega=B(0,3)\subset\mathbb R^n\), \(n\geq3\), and write \(\Gamma=\partial\Omega\).
There exist fixed, distinct, nonzero initial pressures \(f_1,f_2\in C_c^\infty(\Omega)\) with the following property.
For every \(T>0\), there are smooth sound speeds \(0<c_{j,T}\leq1\), equal to one outside \(\Omega\), such that the corresponding solutions of \eqref{eq:1} with \(\sigma=1\) and \(h=0\) satisfy
\[
 u_1=u_2=0\quad\text{on }\Gamma\times(0,T),\qquad
 \frac{f_1}{c_{1,T}^2}\not\equiv\frac{f_2}{c_{2,T}^2}.
\]
The quotients share a prescribed longitudinal profile independent of \(T\).
In dimension three, both configurations belong to the admissible class of Theorem~\ref{thm0}.
\end{proposition}

\begin{remark}\label{rem:finite-record-obstruction}
In dimension three, the full-time records distinguish these configurations by Theorem~\ref{thm0}, while every prescribed finite window fails for some pair in the family.
Thus the full-time hypothesis in that theorem cannot be replaced by one recording time fixed in advance for this family.
The sources and common profile are fixed; the sound speeds depend on the recording time and have no common positive lower bound.
The construction and the resulting nonzero terminal energy are given in Appendix~\ref{subsec:observation-time-obstruction}.
\end{remark}

\medskip
\Needspace{8\baselineskip}
\begin{proposition}\label{prop:persistent-boundary-echoes}
Let \(\Omega=B(0,3)\subset\mathbb R^3\), and take \(\sigma=1\) and \(h=0\) in \eqref{eq:1}.
There exists a fixed configuration \((f,c)\) in the admissible class of Theorem~\ref{thm0}, with nonnegative, nonzero \(f\in C_c^\infty(\Omega)\) and \(1/2\leq c\leq1\), for which
\[
 u|_{\partial\Omega\times(T,\infty)}\not\equiv0
 \qquad\text{for every }T>0.
\]
\end{proposition}

\begin{remark}\label{rem:persistent-boundary-observation}
This proposition exhibits a different limitation of finite observations: fixed positive speed bounds do not ensure that the boundary signal ends after a finite recording time.
Consequently, a finite record cannot generally be completed by zero extension to supply the full-time data used in our integral identities.
This conclusion concerns the duration of the signal and does not assert nonuniqueness from finite-time data.
The explicit medium and the echo calculation are given in Appendix~\ref{subsec:persistent-boundary-echoes}.
\end{remark}

For the finite-time reduction, we impose the following finite extinction condition near the observation boundary.

\begin{admissible}\label{ass:strong_huygens}
For every compact set \(K\subset\overline\Omega\), there exist \(T_K>0\) and an open neighbourhood \(U_K\) of \(\partial\Omega\) such that every weak solution of \eqref{eq:1} with \(L^2\) initial data supported in \(K\) satisfies
\[
u=0\qquad\text{in }U_K\times(T_K,\infty)
\]
in the distributional sense.
The time \(T_K\) and the neighbourhood \(U_K\) may depend on \(K\) and the medium \((c,\sigma)\), but not on the choice of initial data supported in \(K\).
\end{admissible}

\begin{theorem}\label{thm:finite-time-measurement}
Let two configurations satisfy the hypotheses of one of Theorems~\ref{thm0}--\ref{thm:general}, and suppose that both media satisfy Admissibility condition~\ref{ass:strong_huygens}.
Then, for all sufficiently large \(T>0\), their boundary observations on \(\partial\Omega\times\mathbb R_+\) are determined by their restrictions to \(\partial\Omega\times(0,T)\).
Consequently, equality of these finite-time observations implies the corresponding uniqueness conclusion.
\end{theorem}
Choose \(T\) larger than the two cutoff times supplied by Admissibility condition~\ref{ass:strong_huygens} for the respective source supports.
The weak boundary observations then vanish for \(t>T\), so the full observations are recovered by extending the finite records by zero.
Lemma~\ref{lem:boundary_flux_from_dirichlet} also applies on that interval, so only the displacement trace is observed.
For smooth coefficients, a sufficient setting for Admissibility condition~\ref{ass:strong_huygens} is the global strong Huygens principle together with a uniform time after which all bicharacteristics issued from the source support remain outside a fixed neighbourhood of \(\overline\Omega\); cf.\ \cite{friedlander1975wave,sonego1992huygens,belger1997survey}.

\begin{remark}
Theorem~\ref{thm:finite-time-measurement} gives a sufficient condition for a finite-time reduction.
The time \(T\) may depend on the two configurations and their source supports.
For this reduction to use a recording interval prescribed uniformly over a coefficient class, it suffices to assume a common a priori bound on the corresponding cutoff times.
Conditions ensuring the strong Huygens principle for variable coefficients are geometry-dependent; see \cite{belger1997survey}.
\end{remark}
\par\endgroup

\medskip
\subsection{\texorpdfstring{{Joint recovery for parabolic and Schr\"odinger equations}}{Joint recovery for parabolic and Schr\"odinger equations}}
The unified framework establishes joint recovery for parabolic and Schr\"odinger equations from a single passive boundary observation.
Theorems~\ref{ext:parabolic-model} and~\ref{ext:schrodinger-model} show that this observation uniquely determines the initial state throughout \(\Omega\) and the diffusivity or potential on the support of that state.

\medskip
\subsubsection{Parabolic model.}
Consider the parabolic equation
\begin{equation}\label{eq:parabolic-isotropic}
\begin{cases}
\displaystyle \frac{1}{\mu(x)}\,\partial_t u(x,t)-\Delta u(x,t)=0, & (x,t)\in\mathbb R^3\times\mathbb R_+,\\[2mm]
u(x,0)=f(x), & x\in\mathbb R^3.
\end{cases}
\end{equation}
\begingroup
Let \(\Omega\) be a bounded Lipschitz observation domain, and assume \(f\in L^2(\mathbb R^3)\) and
\endgroup
\[
\operatorname{supp}(f)\cup\operatorname{supp}(1-\mu)\Subset\Omega,\qquad
0<\mu_-\le\mu(x)\le\mu_+<\infty\quad\text{for }x\in\Omega.
\]
\begingroup
The passive boundary observation is
\begin{equation}\label{eq:parabolic-observation}
\mathcal P_{f,\mu}:=u|_{\partial\Omega\times\mathbb R_+}.
\end{equation}
\endgroup

\begingroup
\begin{theorem}\label{ext:parabolic-model}
\begingroup
Let \((f,\mu)\) and \((\tilde f,\tilde\mu)\) satisfy the above assumptions, and suppose that the zero extensions of \(f/\mu\) and \(\tilde f/\tilde\mu\) satisfy Admissibility~\ref{ass:x3_invariance_cf} with a common prescribed profile.
If the observations \eqref{eq:parabolic-observation} coincide, then \(f/\mu=\tilde f/\tilde\mu\) almost everywhere in \(\Omega\).

For joint recovery, assume in addition that
\[
f=p(x')\tilde\phi(x_3),\qquad \tilde f=\tilde p(x')\tilde\phi(x_3),
\]
where {\(p,\tilde p\in L^2(\mathbb R^2)\)} and \(\tilde\phi\in C([a,b])\) is prescribed on a known interval \(I=(a,b)\), extended by zero outside \(I\), and satisfies \(\tilde\phi(b)\ne0\).
Assume also that both diffusivities are continuous on \(\mathbb R^2\times(a,b)\) up to \(x_3=b\) from below and equal one for \(x_3>b\).
Then the same observation gives
\[
f=\tilde f\quad\text{a.e. in }\Omega,\qquad
\mu=\tilde\mu\quad\text{a.e. on }\operatorname{supp}(f).
\]
\endgroup
\end{theorem}
\endgroup

\medskip
\subsubsection{Schr\"odinger model.}
Consider the Schr\"odinger equation with an external potential:
\begin{equation}\label{eq:schrodinger-potential}
\begin{cases}
\displaystyle \mathrm{i}\partial_t\psi(x,t)+\Delta\psi(x,t)-V(x)\psi(x,t)=0,
& (x,t)\in\mathbb R^3\times\mathbb R_+,\\[2mm]
\psi(x,0)=f(x),
& x\in\mathbb R^3.
\end{cases}
\end{equation}
\begingroup
We assume that \(f\) and the real-valued potential \(V\) are compactly supported in \(\Omega\), and observe the complex boundary trace, including its phase:
\endgroup
\begingroup
\begin{equation}\label{eq:schrodinger-observation}
\mathcal S_{f,V}:=\psi|_{\partial\Omega\times\mathbb R_+}.
\end{equation}
\endgroup

\begingroup
\begin{theorem}\label{ext:schrodinger-model}
\begingroup
Let \((f,V)\) and \((\tilde f,\tilde V)\) satisfy the support assumption above and have zero extensions of the form
\[
 f(x)=p(x')\tilde\phi(x_3),\qquad V(x)=B(x')\theta(x_3),
\]
and likewise for \((\tilde f,\tilde V)\), with the same prescribed profiles on a known interval \(I=(a,b)\), extended by zero outside \(I\).
Here \(p,\tilde p\in C_c^2(\mathbb R^2;\mathbb C)\), \(B,\tilde B\in C_c(\mathbb R^2;\mathbb R)\), \(\tilde\phi\in C^2([a,b];\mathbb C)\), and \(\theta\in L^\infty(I;\mathbb R)\), with
\[
\tilde\phi(b)\ne0,\qquad \int_a^b\theta(r)\,\mathrm dr\ne0.
\]
If the observations \eqref{eq:schrodinger-observation} coincide, then
\[
f=\tilde f\quad\text{a.e. in }\Omega,\qquad
V=\tilde V\quad\text{a.e. on }\operatorname{supp}(f).
\]
\endgroup
\end{theorem}
\endgroup

\begingroup
The proofs and the weak interpretation of the observations are given in Section~\ref{sec:parabolic}.
\endgroup

\section{Mathematical setting and hyperbolic forward problems}\label{sec:setup}

\begingroup
We establish the weak solution, temporal transform, and boundary trace conventions used in the hyperbolic uniqueness arguments.

\subsection{Weak solutions of the Cauchy problem}\label{subsec Forward Hy}

Let \(f,h\in L^2(\mathbb R^n)\) and let \(c,\sigma\in L^\infty(\mathbb R^n)\) satisfy \eqref{eq:general_support} and \eqref{eq:general_bounds}.
These assumptions allow jumps in both the initial data and the coefficients.
We define the weak solution of \eqref{eq:1} through the energy problem
\begin{equation}\label{eq:general_integrated_wave}
\begin{cases}
\displaystyle \frac{1}{c^2}\partial_t^2 U-\nabla\cdot(\sigma\nabla U)=\frac{h}{c^2}
&\text{in }\mathbb R^n\times(0,\infty),\\[4pt]
U(\cdot,0)=0,\qquad \partial_tU(\cdot,0)=f
&\text{in }\mathbb R^n,
\end{cases}
\end{equation}
and set \(u=\partial_tU\), so that \(U(\cdot,t)=\int_0^t u(\cdot,s)\,\mathrm ds\).

\begin{lemma}\label{lem fhy}
Under the assumptions above, \eqref{eq:general_integrated_wave} has a unique energy solution with
\[
U\in C([0,\infty);H^1(\mathbb R^n)),
\qquad
u=\partial_tU\in C([0,\infty);L^2(\mathbb R^n)).
\]
The field \(u\) is the unique weak solution of \eqref{eq:1} in this integrated sense, and, for \(t\ge0\),
\begin{equation}\label{eq:general_weak_energy}
\|u(\cdot,t)\|_{L^2}
+\|\nabla U(\cdot,t)\|_{L^2}
\le C\bigl(\|f\|_{L^2}+t\|h\|_{L^2}\bigr).
\end{equation}
Here the norms are over \(\mathbb R^n\), and \(C\) depends only on the uniform coefficient bounds, including the exterior value \(1\).
\end{lemma}

\begin{proof}
In \(L^2(\mathbb R^n;c^{-2}\,\mathrm dx)\), the quadratic form \(\int_{\mathbb R^n}\sigma|\nabla v|^2\,\mathrm dx\) with domain \(H^1(\mathbb R^n)\) is closed and nonnegative.
The spectral calculus for the associated nonnegative self-adjoint operator gives the unique energy solution of \eqref{eq:general_integrated_wave}.
The energy estimate for this equation and the uniform coefficient bounds give \eqref{eq:general_weak_energy}.
Since \(U(\cdot,0)=0\), integration also bounds \(\|U(\cdot,t)\|_{L^2}\), so \(U\) grows at most polynomially in \(H^1(\mathbb R^n)\).

Differentiating \eqref{eq:general_integrated_wave} in time gives \eqref{eq:1} in the integrated weak sense.
The initial displacement is \(u(\cdot,0)=f\) in \(L^2\).
The initial velocity is determined by
\[
\partial_t(c^{-2}u)
=\nabla\cdot(\sigma\nabla U)+c^{-2}h
\in C([0,\infty);H^{-1}(\mathbb R^n)),
\]
whose value at \(t=0\) is \(c^{-2}h\).
Conversely, the time primitive of a weak solution in this class solves \eqref{eq:general_integrated_wave}, so uniqueness follows from energy uniqueness.
\end{proof}

For globally smooth coefficients and initial data, this weak solution agrees with the regular Cauchy solution \cite{kachalov2001inverse,evans2010partial}.
The regularity assumptions for smooth configurations apply to the zero extensions of the initial data and the background extensions of the coefficients to \(\mathbb R^n\).
The surface \(\partial\Omega\) is the observation boundary, while the wave evolves in \(\mathbb R^n\).

\subsection{Temporal Fourier transform}\label{subsec:temporal_FT}

We use the half-line temporal Fourier transform with an exponentially decaying weight:
\begin{equation}\label{eq FT}
\hat u(x,k)=\int_0^\infty u(x,t)\exp(-\mathrm i kt)\,\mathrm dt,
\qquad x\in\mathbb R^n,\quad\operatorname{Im}k<0.
\end{equation}
The estimate \eqref{eq:general_weak_energy} makes this integral convergent in \(L^2(\mathbb R^n)\).
The polynomial growth of \(U\) in \(H^1\) and \(U(\cdot,0)=0\) justify integration by parts and give
\begin{equation}\label{eq:general_weak_transform}
\hat u(\cdot,k)
=\mathrm i k\int_0^\infty \exp(-\mathrm i kt)U(\cdot,t)\,\mathrm dt
\in H^1(\mathbb R^n).
\end{equation}
Transforming \eqref{eq:general_integrated_wave} therefore yields
\begin{equation}\label{FourierTransform}
-\Delta_\sigma\hat u(x,k)-\frac{k^2}{c(x)^2}\hat u(x,k)
=\frac{\mathrm i kf(x)+h(x)}{c(x)^2},
\qquad x\in\mathbb R^n,
\end{equation}
where \(\Delta_\sigma v=\nabla\cdot(\sigma\nabla v)\).
The right-hand side and the coefficient perturbations are supported in \(\overline\Omega\).
In particular, \(\nabla\cdot(\sigma\nabla\hat u)\in L^2(\mathbb R^n)\), so the transformed field has Dirichlet and weak conormal traces.

The hyperbolic uniqueness arguments take \(k\in-\mathrm i\mathbb R_+\), for which \(\mathrm i k>0\) and the mass term \(-k^2/c^2\) is positive.
In the TAT/PAT case, \(\sigma=1\), and \eqref{FourierTransform} gives \(\Delta\hat u\in L^2(\mathbb R^3)\).
Together with the \(L^2\) bound, this implies \(\hat u(\cdot,k)\in H^2(\mathbb R^3)\) at each fixed such frequency.

\subsection{Boundary traces and Cauchy data}\label{subsec:weak-boundary-data}

For the weak solution in Subsection~\ref{subsec Forward Hy}, the spatial trace \(\operatorname{Tr}U(\cdot,t)\) belongs to \(H^{1/2}(\partial\Omega)\).
We define the observation \eqref{eq:2} by
\[
u|_{\partial\Omega}:=\partial_t\operatorname{Tr}U.
\]
To define the conormal trace, integrate \eqref{eq:general_integrated_wave} once in time.
The \(H^1\) function \(\int_0^t(t-s)u(\cdot,s)\,\mathrm ds=\int_0^tU(\cdot,s)\,\mathrm ds\) satisfies
\[
-\nabla\cdot\left(\sigma\nabla\int_0^t(t-s)u(\cdot,s)\,\mathrm ds\right)
=\frac{f+t h-u(\cdot,t)}{c^2}\in L^2(\mathbb R^n).
\]
It therefore has a weak conormal trace in \(H^{-1/2}(\partial\Omega)\), and we set
\[
(\sigma\partial_\nu u)|_{\partial\Omega}
:=\partial_t^2\left[\sigma\partial_\nu
\int_0^t(t-s)u(\cdot,s)\,\mathrm ds\right]_{\partial\Omega}.
\]
The time derivatives are taken distributionally after extending the primitives by zero to \(t<0\).
Their restrictions to \(t>0\) define the weak boundary traces, while the full causal distributions are retained for the transforms below.
The Dirichlet and conormal traces have values in \(H^{1/2}(\partial\Omega)\) and \(H^{-1/2}(\partial\Omega)\), respectively, in this distributional sense.
These definitions agree with the usual traces whenever the latter exist.

The transform \eqref{eq FT} is compatible with these trace definitions:
\[
\widehat{u|_{\partial\Omega}}(\cdot,k)=\operatorname{Tr}\hat u(\cdot,k),
\qquad
\widehat{(\sigma\partial_\nu u)|_{\partial\Omega}}(\cdot,k)
=(\sigma\partial_\nu\hat u)|_{\partial\Omega},
\]
where the transforms on the left are taken for the causal distributions just defined.
Indeed, their time derivatives transform into multiplication by \(\mathrm i k\) and \((\mathrm i k)^2\), respectively.
The polynomial \(H^1\) bound permits the transform of \(U\) to commute with the Dirichlet trace.
For the conormal trace, the divergence identity above and \eqref{eq:general_weak_energy} show that \(\sigma\nabla\int_0^tU(\cdot,s)\,\mathrm ds\) is continuous and grows at most polynomially in \(H(\operatorname{div};\mathbb R^n)\), with both the flux and its divergence controlled in \(L^2\).
The weak normal trace is continuous in this norm, so it also commutes with the transform; \eqref{eq:general_weak_transform} then gives the two identities.
All Green boundary terms below use bilinear duality pairings of \(H^{-1/2}(\partial\Omega)\) with \(H^{1/2}(\partial\Omega)\), without complex conjugation.

We now show that equality of the observed Dirichlet traces gives equality of the physical conormal traces, also on finite time intervals.

\begin{lemma}\label{lem:boundary_flux_from_dirichlet}
Let \((f,h,c,\sigma)\) and \((\tilde f,\tilde h,\tilde c,\tilde\sigma)\) satisfy \eqref{eq:general_support} and \eqref{eq:general_bounds}, with \(L^2(\mathbb R^n)\) initial data and bounded coefficients.
Let \(u\) and \(\tilde u\) be the corresponding weak solutions, and let \(0<T\le\infty\).
If
\[
u=\tilde u\qquad\text{on }\partial\Omega\times(0,T),
\]
then the solutions agree in \((\mathbb R^n\setminus\overline\Omega)\times(0,T)\), and
\begin{equation}\label{eq:boundary_flux_from_dirichlet}
\sigma\partial_\nu u=\tilde\sigma\partial_\nu\tilde u
\qquad\text{on }\partial\Omega\times(0,T)
\end{equation}
in the weak trace sense defined above.
\end{lemma}

\begin{proof}
Let \(U\) and \(\tilde U\) be the energy solutions in \eqref{eq:general_integrated_wave} for the two configurations.
Equality of the observed traces gives \(\partial_t\operatorname{Tr}(U-\tilde U)=0\) on \((0,T)\).
Since both time integrals vanish at zero, \(\operatorname{Tr}(U-\tilde U)=0\) on \((0,T)\).
By \eqref{eq:general_support}, their difference satisfies
\[
(\partial_t^2-\Delta)(U-\tilde U)=0
\qquad\text{in }(\mathbb R^n\setminus\overline\Omega)\times(0,T),
\]
with zero initial data and zero Dirichlet trace.
The exterior energy identity therefore gives \(U=\tilde U\), and hence \(u=\tilde u\), in this region.

The functions \(\int_0^t U(\cdot,s)\,\mathrm{d}s\) and \(\int_0^t\tilde U(\cdot,s)\,\mathrm{d}s\) also agree in the exterior.
Integrating \eqref{eq:general_integrated_wave} shows that their respective flux divergences belong to \(L^2(\mathbb R^n)\).
The full-space weak equations therefore give continuity of their physical conormal fluxes across \(\partial\Omega\), with both traces taken in the direction \(\nu\) outward from \(\Omega\).
Their equal exterior traces consequently give equal interior conormal traces.
Taking two distributional time derivatives proves \eqref{eq:boundary_flux_from_dirichlet}.

The proof is complete.
\end{proof}

For observations over all positive times, this lemma and the trace compatibility above give equality of the Dirichlet and conormal traces of the transformed fields used in the uniqueness proofs.

\begin{remark}\label{rem:other_boundary_flux}
For the parabolic and Schr\"odinger models, equal Dirichlet observations also give equal normal derivatives under the stated assumptions on support and forward solutions.
The differences of the time primitives have zero initial values and zero Dirichlet traces, and satisfy the free heat and Schr\"odinger equations, respectively, in the exterior.
The heat energy identity and conservation of the exterior \(L^2\) norm for the Schr\"odinger equation show that these differences vanish there.
Differentiation in time gives equality of the exterior fields.
The full-space \(H^2\) regularity of the primitives gives matching interior and exterior normal derivatives, with the same normal direction \(\nu\).
Taking one distributional time derivative yields equality of the weak normal traces.
Thus the observations \eqref{eq:parabolic-observation} and \eqref{eq:schrodinger-observation} determine the normal derivatives used in the corresponding Green identities.
\end{remark}
\endgroup

\section{Joint recovery in TAT/PAT}\label{sec:cgo}

This section proves Theorems~\ref{thm0} and~\ref{thm:unique_f_c} for the TAT/PAT problem.
The first result determines the quotient \(f/c^2\), while the second recovers \(f\) and \(c\) from a single passive boundary observation.
As noted in Remark~\ref{eq rem1}, we present the arguments in the physically relevant dimension \(n=3\); the same framework extends to \(n\geq3\) under the corresponding one-directional structural assumptions.

Both proofs follow the unified framework of Subsection~\ref{subsec:framework}.
The first integral identity, combined with low-frequency CGO tests, determines the quotient.
For joint recovery, we use free background tests in the second identity to obtain equality of the transformed traces on a common end face.
The high-frequency limit of these traces then supplies a second weighted source relation, which is combined with the recovered quotient to separate the source and sound speed.

We use the transform defined in Subsection~\ref{subsec:temporal_FT} at frequencies \(k\in-\mathrm i\mathbb R_+\).
Along this ray, \(e^{-\mathrm ikt}=e^{-|k|t}\), \(\mathrm ik=|k|>0\), and \(-k^2=|k|^2\).
The transformed boundary data are obtained directly from the full-time observation.
The exponential weight controls the growth allowed by the weak energy estimate and justifies the temporal integrations by parts described there.
The transformed equation also contains the positive term \(|k|^2/c^2\), which is used in the energy estimates and interface analysis below.
The low- and high-frequency limits correspond to \(|k|\downarrow0\) and \(|k|\to\infty\), respectively, along this same ray.

The separated representations and regularity hypotheses are those stated in Theorems~\ref{thm0} and~\ref{thm:unique_f_c}.
In these representations, the unknown transverse factors \(q,\tilde q,p,\tilde p\) depend on \(x'=(x_1,x_2)\) and belong to \(L^2\), while the longitudinal dependence is prescribed by the known profiles \(\phi\) and \(\tilde\phi\).
Joint recovery additionally uses a nonzero endpoint value of the continuous source profile \(\tilde\phi\) and continuous interior extensions of both sound speeds to that end plane.
These are regularity requirements within the common support interval and at its one-sided boundary traces; the extensions across the end planes are allowed to have jumps.

\subsection{Step 1: Time-domain identity}

We start the proof of Theorem~\ref{thm0} with a time-domain integral identity.
Let \(\mathscr W\) consist of waves \(w(x,t)=\exp(-\tau t)v(x)\), where \(\tau>0\), \(v\in H^2(Q)\), and \(w\) solves
\begin{equation}\label{eq:wave_w}
\frac{1}{c(x)^2}\partial_t^2 w-\Delta w=0
\quad\text{in }\Omega\times(0,\infty),
\end{equation}
and note that the Sobolev embedding \(H^2(Q)\hookrightarrow L^\infty(Q)\) gives
\begin{equation}\label{eq:w_bounded_simple}
\sup_{t\ge 0}\Bigl(\|w(\cdot,t)\|_{L^\infty(\Omega)}
+\|\partial_t w(\cdot,t)\|_{L^\infty(\Omega)}\Bigr)<\infty.
\end{equation}
We use the CGO waves constructed below in this class to test the quotient \(f/c^2\).

The initial data in Theorem~\ref{thm0} belong to \(L^2(\mathbb R^3)\).
We use the weak solutions of Subsection~\ref{subsec Forward Hy} and the boundary traces of Subsection~\ref{subsec:weak-boundary-data}, with \(h=0\) and \(\sigma=1\).
The exponential temporal factor and the weak energy estimate make the time pairings convergent.
For joint recovery, the free spatial test functions enter the Green identity in Lemma~\ref{lem:second_identity}.

\begin{lemma}\label{lem:integral_identity}
{
Let \((f,c)\) and \((\tilde f,\tilde c)\) be two TAT/PAT configurations satisfying \eqref{eq:general_support} and \eqref{eq:general_bounds}, with \(f,\tilde f\in L^2(\mathbb R^3)\).
Let \(u\) and \(\tilde u\) be the corresponding weak solutions, and assume that their observations coincide in the sense of \eqref{eq:condition}.
Let \(w(x,t)=\exp(-\tau t)v(x)\in\mathscr W\), with \(\tau>0\) and \(v\in H^2(Q)\).
Then the following integral identity holds:
}
\begin{multline}\label{eq:iid1}
\int_0^\infty\!\!\int_\Omega
\left(\frac{1}{\tilde c(x)^2}-\frac{1}{c(x)^2}\right)\tilde u(x,t)\,\partial_t^2 w(x,t)\,\mathrm{d}x\,\mathrm{d}t\\
-
\int_\Omega
\left(\frac{f(x)}{c(x)^2}-\frac{\tilde f(x)}{\tilde c(x)^2}\right)\partial_t w(x,0)\,\mathrm{d}x
=0.
\end{multline}
\end{lemma}

\begin{proof}
We keep the coefficients fixed and approximate \(f,\tilde f\) in \(L^2(\mathbb R^3)\) by data in the domains of the corresponding wave operators, as in Lemma~\ref{lem:first_identity}.
The calculations below are carried out for these data on finite time intervals, retaining all boundary terms.
At each fixed \(\tau>0\), the transformed equations give convergence of the fields in \(H^2(\mathbb R^3)\), and hence convergence of their Dirichlet and weak normal traces.
We pass to the original data before using the common Cauchy data to cancel the boundary terms.
The energy bounds and the exponential factor make the terminal pairings tend to zero as \(T\to\infty\).
We multiply the equation for $\tilde u$ by $w$ and integrate over $\Omega\times(0,\infty)$.
An integration by parts in the spatial variable yields
\begin{multline}\label{eq:stepA_rewrite}
\int_0^\infty\!\!\int_\Omega \frac{1}{\tilde c(x)^2}\,\partial_t^2\tilde u(x,t)\,w(x,t)\,\mathrm{d}x\,\mathrm{d}t\\
-
\int_0^\infty\!\!\int_\Omega \tilde u(x,t)\,\Delta w(x,t)\,\mathrm{d}x\,\mathrm{d}t
+
\int_0^\infty\!\!\int_{\partial\Omega}\Bigl(\tilde u\,\partial_\nu w-(\partial_\nu\tilde u)\,w\Bigr)\,\mathrm{d}S\,\mathrm{d}t
=0.
\end{multline}
Since $w$ satisfies \eqref{eq:wave_w}, we have $\Delta w=\frac{1}{c(x)^2}\partial_t^2 w$ for a.e.\ $t>0$.
Substituting this into \eqref{eq:stepA_rewrite} gives
\begin{multline}\label{eq:stepB_rewrite}
\int_0^\infty\!\!\int_\Omega \frac{1}{\tilde c(x)^2}\,\partial_t^2\tilde u\,w\,\mathrm{d}x\,\mathrm{d}t\\
-
\int_0^\infty\!\!\int_\Omega \frac{1}{c(x)^2}\,\tilde u\,\partial_t^2 w\,\mathrm{d}x\,\mathrm{d}t
+
\int_0^\infty\!\!\int_{\partial\Omega}\Bigl(\tilde u\,\partial_\nu w-(\partial_\nu\tilde u)\,w\Bigr)\,\mathrm{d}S\,\mathrm{d}t
=0.
\end{multline}

{
Next, we integrate by parts twice in time in the spatial \(L^2\) pairing in the first term of \eqref{eq:stepB_rewrite}.
This gives
\[
\begin{aligned}
\int_0^T\!\!\int_\Omega \frac{\partial_t^2\tilde u\,w}{\tilde c^2}\,\mathrm{d}x\,\mathrm{d}t
=
\left[\int_\Omega\frac{\partial_t\tilde u\,w-\tilde u\,\partial_t w}{\tilde c^2}\,\mathrm{d}x\right]_{t=0}^{t=T}
\quad+\int_0^T\!\!\int_\Omega \frac{\tilde u\,\partial_t^2 w}{\tilde c^2}\,\mathrm{d}x\,\mathrm{d}t.
\end{aligned}
\]
{
The terminal pairing tends to zero as \(T\to\infty\).
}
}
Using the initial conditions $\tilde u(\cdot,0)=\tilde f$ and $\partial_t\tilde u(\cdot,0)=0$, we obtain
\begin{equation}\label{eq:time_twice_rewrite}
\int_0^\infty\!\!\int_\Omega \frac{1}{\tilde c^2}\,\partial_t^2\tilde u\,w\,\mathrm{d}x\,\mathrm{d}t
=
\int_0^\infty\!\!\int_\Omega \frac{1}{\tilde c^2}\,\tilde u\,\partial_t^2 w\,\mathrm{d}x\,\mathrm{d}t
+
\int_\Omega \frac{\tilde f(x)}{\tilde c(x)^2}\,\partial_t w(x,0)\,\mathrm{d}x.
\end{equation}
Substituting \eqref{eq:time_twice_rewrite} into \eqref{eq:stepB_rewrite} yields
\begin{multline}\label{eq:stepC_rewrite}
\int_0^\infty\!\!\int_\Omega
\left(\frac{1}{\tilde c(x)^2}-\frac{1}{c(x)^2}\right)\tilde u(x,t)\,\partial_t^2 w(x,t)\,\mathrm{d}x\,\mathrm{d}t\\
+
\int_\Omega \frac{\tilde f(x)}{\tilde c(x)^2}\,\partial_t w(x,0)\,\mathrm{d}x
+
\int_0^\infty\!\!\int_{\partial\Omega}\Bigl(\tilde u\,\partial_\nu w-(\partial_\nu\tilde u)\,w\Bigr)\,\mathrm{d}S\,\mathrm{d}t
=0.
\end{multline}

It remains to rewrite the boundary flux term.
Since $u$ and $w$ satisfy the same wave equation with coefficient $c$, we have in $\Omega\times(0,\infty)$
\[
\frac{1}{c^2}\partial_t^2 u-\Delta u=0,
\qquad
\frac{1}{c^2}\partial_t^2 w-\Delta w=0.
\]
We multiply the first equation by $w$, the second by $u$, subtract, and integrate over $\Omega$.
After integrating by parts in $x$, we obtain that
\begin{multline*}
    \frac{\mathrm{d}}{\mathrm{d}t}\int_\Omega \frac{1}{c(x)^2}\Bigl(\partial_t u(x,t)\,w(x,t)-u(x,t)\,\partial_t w(x,t)\Bigr)\,\mathrm{d}x
=\\
\int_{\partial\Omega}\Bigl(\partial_\nu u(x,t)\,w(x,t)-u(x,t)\,\partial_\nu w(x,t)\Bigr)\,\mathrm{d}S\qquad \text{for a.e.\ $t>0$}.
\end{multline*}

{
Integrating in time and using the vanishing terminal pairing gives the boundary contribution below.
}
Using $u(\cdot,0)=f$ and $\partial_t u(\cdot,0)=0$, we obtain
\begin{equation}\label{eq:boundary_time_integral_rewrite}
\int_0^\infty\!\!\int_{\partial\Omega}\Bigl(u\,\partial_\nu w-(\partial_\nu u)\,w\Bigr)\,\mathrm{d}S\,\mathrm{d}t
=
-\int_\Omega \frac{f(x)}{c(x)^2}\,\partial_t w(x,0)\,\mathrm{d}x.
\end{equation}

{
Passing to the original data, \eqref{eq:condition} and Lemma~\ref{lem:boundary_flux_from_dirichlet} give $u=\tilde u$ and $\partial_\nu u=\partial_\nu\tilde u$ on $\partial\Omega\times(0,\infty)$ in the sense of traces.
}
Therefore,
\[
\int_0^\infty\!\!\int_{\partial\Omega}\Bigl(\tilde u\,\partial_\nu w-(\partial_\nu\tilde u)\,w\Bigr)\,\mathrm{d}S\,\mathrm{d}t
=
\int_0^\infty\!\!\int_{\partial\Omega}\Bigl(u\,\partial_\nu w-(\partial_\nu u)\,w\Bigr)\,\mathrm{d}S\,\mathrm{d}t.
\]
Substituting \eqref{eq:boundary_time_integral_rewrite} into \eqref{eq:stepC_rewrite} yields \eqref{eq:iid1}.
\end{proof}
\subsection{Step 2: First CGO construction}

We now take the test wave $w\in\mathscr W$ in the complex geometric optics (CGO) form
\begin{equation}\label{w_ansatz}
w(x,t)=e^{\mathrm{i}(\rho_0 t+\rho'\cdot x)}\bigl(1+\psi_{\rho}(x)\bigr),
\end{equation}
where $\rho=(\rho_0,\rho')\in\mathbb C\times\mathbb C^3$, and the corrector $\psi_{\rho}$ depends only on $x$.
A direct computation gives
\[
\frac{1}{c(x)^2}\partial_t^2 w
=
-\frac{\rho_0^2}{c(x)^2}\,w.
\]
Moreover, using $\nabla e^{\mathrm{i}\rho'\cdot x}=\mathrm{i}\rho'e^{\mathrm{i}\rho'\cdot x}$, we obtain
\[
\Delta w
=
e^{\mathrm{i}(\rho_0 t+\rho'\cdot x)}
\left(
\Delta\psi_{\rho}
+2\mathrm{i}\rho'\cdot\nabla\psi_{\rho}
-(\rho'\cdot\rho')(1+\psi_{\rho})
\right).
\]
Substituting these expressions into \eqref{eq:wave_w} and dividing by
$e^{\mathrm{i}(\rho_0 t+\rho'\cdot x)}$ leads to the equation
\begin{equation}\label{eq:psi_rho}
\Delta\psi_{\rho}+2\mathrm{i}\rho'\cdot\nabla\psi_{\rho}
=
\left(\rho'\cdot\rho'-\frac{\rho_0^2}{c(x)^2}\right)(1+\psi_{\rho}).
\end{equation}

To solve \eqref{eq:psi_rho}, we introduce a Fourier-multiplier right inverse $G_\zeta$ for the conjugated operator $\Delta+2\mathrm{i}\zeta\cdot\nabla$ on the reference {domain} $Q$.
Its $L^2(Q)\to H^2(Q)$ mapping property provides the required control of the remainder term in \eqref{w_ansatz}; cf.\ \cite{colton2019inverse}.

\begin{lemma}\label{lem G zeta}
Let
\[\begin{gathered}
\zeta=(\zeta_{1}, \zeta_{2}, 0) + \mathrm{i} \left(0, 0, \left(\sum_{i=1}^{2}\zeta_{i}^{2}\right)^{\frac{1}{2}}\right)\in\mathbb{C}^{3},
\\
(\zeta_{1}, \zeta_{2}, 0)\in \left\{\beta-\left(0, \frac{1}{2}, 0\right): \beta \in \mathbb{Z}^3\right\}.
\end{gathered}\]
Define
\begin{equation}\label{op_G_zeta}
G_\zeta f:=-\sum_{\alpha \in \widetilde{\mathbb{Z}}^3} \frac{\hat{f}_\alpha}{\alpha \cdot \alpha+2 \zeta \cdot \alpha} e_\alpha,
\end{equation}
where
\[
\widetilde{\mathbb{Z}}^3:=\left\{\alpha=\beta-\left(0, 0, \frac{1}{2}\right): \beta \in \mathbb{Z}^3\right\},
\qquad
Q=[-\pi,\pi]^3,
\]
and
\[
e_\alpha(x)=\frac{1}{(2\pi)^{3/2}}e^{\mathrm{i}\alpha\cdot x},
\qquad
\hat{f}_\alpha = \int_{Q} f(x)\,\overline{e_\alpha(x)}\,\mathrm{d}x
=
\frac{1}{(2 \pi)^{3/2}}\int_{Q} f(x)e^{-\mathrm{i} \alpha \cdot x}\,\mathrm{d}x.
\]
Then $\{e_\alpha\}_{\alpha\in\widetilde{\mathbb Z}^3}$ is an orthonormal basis of $L^2(Q)$, and \eqref{op_G_zeta} defines a linear operator $G_\zeta:L^2(Q)\to H^2(Q)$ satisfying
\begin{equation}\label{bound_G_zeta}
\left\|G_\zeta f\right\|_{L^2(Q)} \leq 2\|f\|_{L^2(Q)}
\end{equation}
and
\begin{equation}\label{inverse_G_zeta}
\Delta G_\zeta f+2 \mathrm{i} \zeta \cdot \nabla G_\zeta f=f
\end{equation}
in the weak sense for all $f\in L^2(Q)$.
\end{lemma}

\begin{proof}
Let $t=\left(\zeta_1^2+\zeta_2^2\right)^{\frac{1}{2}}>0$ denote the positive square root.
Since $\zeta_2\in\mathbb Z+\frac{1}{2}$, we have $t\geq |\zeta_2|\geq \frac{1}{2}$.
Since $|\alpha_{3}|\geq\frac{1}{2}$ for all $\alpha = (\alpha_{1}, \alpha_{2}, \alpha_{3}) \in \widetilde{\mathbb{Z}}^3$, we obtain
\begin{equation*}
|\alpha \cdot \alpha+2 \zeta \cdot \alpha|
\geq |\Im(\alpha \cdot \alpha+2 \zeta \cdot \alpha)|=2t|\alpha_{3}|\geq t\geq \frac{1}{2}.
\end{equation*}
Hence $G_{\zeta}:L^2(Q)\to L^2(Q)$ is well defined.
Moreover, by Parseval's identity,
\begin{equation*}
\left\|G_\zeta f\right\|_{L^2(Q)}^2
=
\sum_{\alpha \in \widetilde{\mathbb{Z}}^3}
\left|\frac{\hat{f}_\alpha}{\alpha \cdot \alpha+2 \zeta \cdot \alpha}\right|^2
\leq
4\|f\|_{L^2(Q)}^2,
\end{equation*}
which yields \eqref{bound_G_zeta}.

We next verify the $H^2$ mapping property.
There exists a constant $C_\zeta>0$ such that
\begin{equation}\label{eq:H2_multiplier_bound}
\left|\frac{1+\alpha \cdot \alpha}{\alpha \cdot \alpha+2 \zeta \cdot \alpha}\right| \leq C_\zeta,
\qquad
\alpha \in \widetilde{\mathbb{Z}}^3.
\end{equation}
{
Indeed, for \(|\alpha|\geq4|\zeta|\), the triangle inequality gives
\[
\left|\alpha\cdot\alpha+2\zeta\cdot\alpha\right|
\geq|\alpha|^2-2|\zeta|\,|\alpha|
\geq\frac12|\alpha|^2.
\]
}
Consequently,
\[
\left|\frac{1+|\alpha|^{2}}{\alpha \cdot \alpha+2 \zeta \cdot \alpha}\right|
\leq 2\left(1+\frac{1}{|\alpha|^2}\right)
\leq 4
\qquad\text{for all }|\alpha|\geq 4|\zeta|.
\]
For the remaining set $\{\alpha\in\widetilde{\mathbb Z}^3:|\alpha|<4|\zeta|\}$, which is finite, the supremum of the left-hand side is finite as well.
Taking $C_\zeta$ as the maximum of these bounds yields \eqref{eq:H2_multiplier_bound}.
Therefore $G_\zeta f\in H^{2}(Q)$ and $\|G_\zeta f\|_{H^{2}(Q)}\leq C_\zeta\|f\|_{L^{2}(Q)}$.

Convergence in \(H^2(Q)\) permits termwise differentiation of the Fourier series, giving
\begin{equation*}
\Delta G_\zeta f+2 \mathrm{i} \zeta \cdot \nabla G_\zeta f
=
-\sum_{\alpha \in \widetilde{\mathbb{Z}}^3}(\alpha \cdot \alpha+2 \zeta \cdot \alpha)\,\widehat{G_\zeta f}_\alpha\, e_\alpha
=
\sum_{\alpha \in \widetilde{\mathbb{Z}}^3} \hat{f}_\alpha e_\alpha
=
f,
\end{equation*}
which proves \eqref{inverse_G_zeta} in the weak sense.
\end{proof}

With the Fourier-multiplier right inverse $G_\zeta$ from Lemma~\ref{lem G zeta} at hand, we can solve \eqref{eq:psi_rho} for the corrector $\psi_\rho$ by rewriting it as a fixed-point problem on $L^2(Q)$.

\begin{lemma}\label{lem rho_0}
Let \(c\in L^\infty(Q)\) be the background extension of the sound speed, with \(c=1\) in \(Q\setminus\overline\Omega\) and the bounds \eqref{eq:general_bounds} in \(\Omega\).
Let $\rho'\in\mathbb{C}^{3}$ be such that $\rho'=\zeta$ for some $\zeta$ of the form in Lemma~\ref{lem G zeta}, and in particular $\rho'\cdot\rho'=0$.
Assume that
\begin{equation}\label{eq:rho0_smallness_new}
\left\|\frac{\rho_{0}^{2}}{c^{2}}\right\|_{L^\infty(Q)}\leq\frac{1}{4}.
\end{equation}
Then there exists $\psi_{\rho}\in H^{2}(Q)$ solving \eqref{eq:psi_rho} in the weak sense.
Moreover,
\begin{equation}\label{estimate_remainder_term_new}
\|\psi_{\rho}\|_{L^{2}(Q)}
\leq
4|\rho_{0}|^{2}\left\|\frac{1}{c^{2}}\right\|_{L^{2}(Q)}.
\end{equation}
\end{lemma}

\begin{proof}
Set $\zeta=\rho'$ and define $T:L^{2}(Q)\to L^{2}(Q)$ by
\begin{equation*}\label{eq:T_mapping_new}
T(\phi):=-G_\zeta\left(\frac{\rho_{0}^{2}}{c^{2}}(\phi+1)\right).
\end{equation*}
For any $\phi,\varphi\in L^{2}(Q)$, Lemma~\ref{lem G zeta} and \eqref{bound_G_zeta} yield
\[
\|T(\phi)-T(\varphi)\|_{L^{2}(Q)}
=
\left\|G_\zeta\left(\frac{\rho_{0}^{2}}{c^{2}}(\phi-\varphi)\right)\right\|_{L^{2}(Q)}
\leq
2\left\|\frac{\rho_{0}^{2}}{c^{2}}\right\|_{L^\infty(Q)}\|\phi-\varphi\|_{L^{2}(Q)}.
\]
Thus $T$ is a contraction on $L^{2}(Q)$ by \eqref{eq:rho0_smallness_new}.
By the contraction mapping theorem, there exists a unique $\phi\in L^{2}(Q)$ such that
\begin{equation}\label{eq:fixed_point_new}
\phi=T(\phi),
\qquad\text{equivalently}\qquad
\phi+G_\zeta\left(\frac{\rho_{0}^{2}}{c^{2}}(\phi+1)\right)=0.
\end{equation}

Taking $L^{2}(Q)$ norms in \eqref{eq:fixed_point_new} and using \eqref{bound_G_zeta}, we obtain
\[
\|\phi\|_{L^{2}(Q)}
=
\left\|G_\zeta\left(\frac{\rho_{0}^{2}}{c^{2}}(\phi+1)\right)\right\|_{L^{2}(Q)}
\leq
2\left\|\frac{\rho_{0}^{2}}{c^{2}}\right\|_{L^\infty(Q)}\|\phi\|_{L^{2}(Q)}
+
2|\rho_{0}|^{2}\left\|\frac{1}{c^{2}}\right\|_{L^{2}(Q)}.
\]
By \eqref{eq:rho0_smallness_new}, we can absorb the first term on the right-hand side into the left-hand side, which gives
\[
\|\phi\|_{L^{2}(Q)}
\leq
4|\rho_{0}|^{2}\left\|\frac{1}{c^{2}}\right\|_{L^{2}(Q)}.
\]
Since $\phi\in L^{2}(Q)$ implies $\frac{\rho_0^2}{c^2}(\phi+1)\in L^{2}(Q)$ and $G_\zeta:L^{2}(Q)\to H^{2}(Q)$ by Lemma~\ref{lem G zeta}, it follows from \eqref{eq:fixed_point_new} that $\phi\in H^{2}(Q)$.

Applying $\Delta+2\mathrm{i}\zeta\cdot\nabla$ to \eqref{eq:fixed_point_new} and using \eqref{inverse_G_zeta}, we obtain in the weak sense
\[
(\Delta+2\mathrm{i}\zeta\cdot\nabla)\phi+\frac{\rho_{0}^{2}}{c^{2}}(\phi+1)=0.
\]
Finally, set $\psi_{\rho}=\phi$ and $\rho'=\zeta$.
Since $\rho'\cdot\rho'=0$, the above equation is equivalent to \eqref{eq:psi_rho}, and \eqref{estimate_remainder_term_new} follows from the $L^{2}(Q)$ estimate proved above.
\end{proof}

\subsection{Step 3: Proof of Theorem~\ref{thm0}}
\leavevmode{
By Admissibility~\ref{ass:x3_invariance_cf}, the zero-extended quotients satisfy
\[
\frac{f(x)}{c(x)^2}=q(x')\phi(x_3),
\qquad
\frac{\tilde f(x)}{\tilde c(x)^2}=\tilde q(x')\phi(x_3)
\]
for a.e.\ \(x=(x',x_3)\in Q\).
Set
\[
g:=\frac{f}{c^2}-\frac{\tilde f}{\tilde c^2},
\qquad
g_0:=q-\tilde q.
\]
}
The zero extension of \(g\) then satisfies
\[
g(x)=g_0(x')\phi(x_3)
\qquad\text{for a.e. }x\in Q.
\]

In this setting, the identity \eqref{eq:iid1} takes the form
\begin{equation*}
\int_0^{\infty}\!\!\int_{\Omega}
\left(\frac{1}{\tilde{c}^2}-\frac{1}{c^2}\right)\tilde{u}(x,t)\,\partial_t^2 w(x,t)\,\mathrm{d}x\,\mathrm{d}t
=
\int_{\Omega}\left(\frac{f}{c^2}-\frac{\tilde{f}}{\tilde{c}^2}\right)\partial_t w(x,0)\,\mathrm{d}x.
\end{equation*}
We take the test wave $w\in\mathscr W$ in the form \eqref{w_ansatz} with a purely imaginary temporal frequency
\[
\rho_0=\mathrm{i}\tau,
\qquad
\tau>0,
\]
so that $e^{\mathrm{i}\rho_0 t}=e^{-\tau t}$ and hence $w$ decays as $t\to\infty$.
By Lemma~\ref{lem G zeta} and Lemma~\ref{lem rho_0}, we may choose a time-independent corrector $\psi_\rho\in H^2(Q)$ solving \eqref{eq:psi_rho} and satisfying \eqref{estimate_remainder_term_new}.
In particular,
\[
\partial_t^2 w(x,t)=-\rho_0^{2}w(x,t),
\qquad
\partial_t w(x,0)
=
\mathrm{i}\rho_{0}e^{\mathrm{i}\rho'\cdot x}\left(1+\psi_{\rho}(x)\right).
\]
Substituting these identities into \eqref{eq:iid1} and dividing by $\rho_0$, we obtain
\begin{multline*}
-\rho_0\int_0^\infty\!\!\int_\Omega
\left(\frac{1}{\tilde c^2}-\frac{1}{c^2}\right)\tilde u(x,t)w(x,t)\,\mathrm{d}x\,\mathrm{d}t\\
=
\mathrm{i}\int_\Omega g(x)e^{\mathrm{i}\rho'\cdot x}\,\mathrm{d}x
+
\mathrm{i}\int_\Omega g(x)e^{\mathrm{i}\rho'\cdot x}\psi_\rho(x)\,\mathrm{d}x.
\end{multline*}
{
We estimate the spacetime term by the transformed equation \eqref{FourierTransform} at \(k=-\mathrm i\tau\).
Here \(\rho_0=-k\), so the conventions \(e^{\mathrm i\rho_0t}\) and \(e^{-\mathrm ikt}\) give the same decaying factor \(e^{-\tau t}\).
Since \(h=0\) and \(\sigma=1\), testing this equation with \(\overline{\hat u}\) and using the Sobolev inequality in \(\mathbb R^3\) gives
\[
\|\nabla\hat u(\cdot,-\mathrm i\tau)\|_{L^2(\mathbb R^3)}^2
\le \tau\|c^{-2}f\|_{L^{6/5}(\mathbb R^3)}
\|\hat u(\cdot,-\mathrm i\tau)\|_{L^6(\mathbb R^3)}.
\]
The source is supported in \(\overline\Omega\), so \eqref{eq:general_bounds} and H\"older's inequality yield
\begin{equation}\label{eq:quotient_low_frequency_bound}
\|\nabla\hat u(\cdot,-\mathrm i\tau)\|_{L^2(\mathbb R^3)}
+\|\hat u(\cdot,-\mathrm i\tau)\|_{L^2(\Omega)}
\le C\tau\|f\|_{L^2(\Omega)},\qquad 0<\tau\le1.
\end{equation}
Here \(C\) depends only on \(\Omega\) and the uniform coefficient bounds.
The same estimate holds for the second configuration.
For fixed \(\rho'\), Lemma~\ref{lem rho_0} bounds \(\exp(\mathrm i\rho'\cdot x)(1+\psi_\rho)\) uniformly in \(L^2(\Omega)\) as \(\tau\to0^+\).
Consequently,
\[
\left|\rho_0\int_0^\infty\!\!\int_\Omega
\left(\frac1{\tilde c^2}-\frac1{c^2}\right)\tilde u\,w\,\mathrm{d}x\,\mathrm{d}t\right|
\le C_{\rho'}\tau\|\widehat{\tilde u}(\cdot,-\mathrm i\tau)\|_{L^2(\Omega)}
\le C_{\rho'}\tau^2\|\tilde f\|_{L^2(\Omega)}\longrightarrow0.
\]
}
For fixed \(\rho'\), the second term on the right tends to \(0\) by Lemma~\ref{lem rho_0} and the Cauchy--Schwarz inequality:
\[
\left|\int_\Omega g(x)e^{\mathrm{i}\rho'\cdot x}\psi_\rho(x)\,\mathrm{d}x\right|
\le \|e^{\mathrm{i}\rho'\cdot x}\|_{L^\infty(\Omega)}\|g\|_{L^2(\Omega)}\|\psi_\rho\|_{L^2(\Omega)}
\le C_{\rho'}\|g\|_{L^2(\Omega)}\,|\rho_0|^2.
\]
Letting \(\tau\to0^+\), we obtain
\begin{equation}\label{eq:Fourier_g}
\int_{\Omega} g(x)e^{\mathrm{i}\rho'\cdot x}\,\mathrm{d}x=0.
\end{equation}

Here $\rho'=(\rho_{1}, \rho_{2}, 0) + \mathrm{i}\left(0, 0, \left(\rho_{1}^{2}+\rho_{2}^{2}\right)^{1/2}\right)\in\mathbb{C}^{3}$ and
\[
(\rho_1,\rho_2,0)\in \left\{\beta-\left(0,\frac{1}{2},0\right):\beta\in\mathbb Z^3\right\}.
\]
In particular, $|\rho_2|\ge \frac12$, and hence $\kappa=(\rho_1^2+\rho_2^2)^{1/2}>0$.
Therefore \eqref{eq:Fourier_g} is equivalent to
\begin{equation}\label{eq:weighted_2D}
\int_{\Omega} g(x)\,e^{-\kappa x_3}e^{\mathrm{i}(\rho_1 x_1+\rho_2 x_2)}\,\mathrm{d}x=0.
\end{equation}

Since both initial sources vanish in $\mathbb R^3\setminus\Omega$, the quotient difference \(g\) vanishes a.e.\ in \(Q\setminus\Omega\), and hence the integral over \(\Omega\) may be replaced by the integral over \(Q=[-\pi,\pi]^3\).
We use the separated representation
\[
g(x_1,x_2,x_3)=g_0(x_1,x_2)\phi(x_3)
\qquad\text{for a.e. }(x_1,x_2,x_3)\in Q.
\]
Applying Fubini's theorem to \eqref{eq:weighted_2D}, we obtain
\[
0
=
\left(\int_{-\pi}^{\pi}\phi(x_3)e^{-\kappa x_3}\,\mathrm{d}x_3\right)
\int_{[-\pi,\pi]^2} g_0(x_1,x_2)e^{\mathrm{i}(\rho_1x_1+\rho_2x_2)}\,\mathrm{d}x_1\,\mathrm{d}x_2.
\]
{Writing \((\rho_1,\rho_2)=(m_1,m_2-\tfrac12)\), we have \(\kappa=\kappa_m\), so \eqref{eq:nonde} makes the first factor nonzero.}
Therefore,
\begin{equation}\label{eq:2D_modes}
\int_{[-\pi,\pi]^2} g_0(x_1,x_2)e^{\mathrm{i}(\rho_1x_1+\rho_2x_2)}\,\mathrm{d}x_1\,\mathrm{d}x_2=0
\end{equation}
for all admissible \((\rho_1,\rho_2)\).
{
These modes are obtained from the standard Fourier basis on \([-\pi,\pi]^2\) by multiplication with \(e^{-\mathrm i x_2/2}\), and are therefore complete in \(L^2([-\pi,\pi]^2)\).
It follows from \eqref{eq:2D_modes} that
}
\[
g_0=0
\qquad\text{a.e.\ in }[-\pi,\pi]^2,
\]
and hence
\[
g=0
\qquad\text{a.e.\ in }Q.
\]
Therefore,
\[
\frac{f}{c^2}=\frac{\tilde f}{\tilde c^2}
\qquad\text{a.e.\ in }\Omega.
\]

The proof is complete.
\begingroup
\subsection{Step 4: High-frequency limits at an end face}\label{sec:prooff}

\leavevmode\begingroup
Having determined the quotient \(f/c^2\) in Theorem~\ref{thm0}, we now derive an additional relation for \(f/(1+c)\) at a common end face.
Combining this relation with the recovered quotient will separate the source and sound speed in Theorem~\ref{thm:unique_f_c}.

To obtain this relation, choose the common end plane in Theorem~\ref{thm:unique_f_c} at which both sound speeds have continuous interior extensions and the profile satisfies the nonzero endpoint condition \eqref{eq:TATPAT_endpoint_nondegeneracy}.
Since the quotient identity established in Step~3 is invariant under reflection of \(x_3\), we may reflect the coordinate and relabel the interval, if necessary, so that the selected plane is \(x_3=b\) and \(\tilde\phi(b)\ne0\).
By the common support condition, both configurations satisfy \(f=\tilde f=0\) and \(c=\tilde c=1\) in the upper half-space \(x_3>b\).

For \(k\in-\mathrm i\mathbb R_+\), the transformed TAT/PAT equation from Subsection~\ref{subsec:temporal_FT} is
\begin{equation}\label{eq:PDE_hat_u_three}
-\Delta\hat u(x,k)-\frac{k^2}{c(x)^2}\hat u(x,k)
=\frac{\mathrm i k}{c(x)^2}f(x)
\quad\text{in }\mathbb R^3.
\end{equation}
\Needspace{10\baselineskip}
The following local limit identifies the endpoint weight \(f/(1+c)\) needed for the simultaneous recovery.
\endgroup

\begingroup
\begin{lemma}\label{lem:piecewise_interface_trace}
Consider a TAT/PAT configuration satisfying \eqref{eq:general_support} and \eqref{eq:general_bounds}, with \(f\in L^\infty(\mathbb R^3)\).
Let \(\hat u\) be the transformed field in \eqref{eq:PDE_hat_u_three}.
Suppose that, near \(x^*\in\partial D\), \(D\) is a Lipschitz domain and \(f=0,\ c=1\) outside \(D\).
Assume that the boundary graph is differentiable at \(x^*\) and that \(f,c\) have limits as \(x\to x^*\) from \(D\), denoted by \(f(x^*)\) and \(c(x^*)\), respectively.
Then, along \(k\in-\mathrm i\mathbb R_+\),
\[
\lim_{\mathrm{i}k\to+\infty}(\mathrm{i}k)\hat u(x^*,k)=\frac{f(x^*)}{1+c(x^*)}.
\]
If instead \(f=0\) and \(c=1\) in a neighbourhood of \(x^*\), then the limit is zero.
\end{lemma}

\begin{proof}
We rescale around \(x^*\) at length \(1/(\mathrm ik)\) to balance the Laplacian and the positive mass term.
By Subsection~\ref{subsec:temporal_FT}, \(\hat u(\cdot,k)\in H^2(\mathbb R^3)\) at each fixed frequency, and we use its continuous representative.
Multiplying \eqref{eq:PDE_hat_u_three} by \(\mathrm ik\) gives an equation for \((\mathrm ik)\hat u\) with a positive zeroth-order coefficient.
The constant \(\|f\|_\infty\) is a supersolution for the real part after any constant phase rotation.
The weak maximum principle \cite[Theorem~8.1]{gilbarg2001elliptic} therefore gives \(\|(\mathrm{i}k)\hat u\|_\infty\leq\|f\|_\infty\).
Testing with the positive part above \(\|f\|_\infty\) gives this comparison on \(\mathbb R^3\).
First rotate the coordinate axes so that the inward unit normal to \(\partial D\) at \(x^*\) points in the positive \(x_3\)-direction.
Then translate and rescale by setting \(y=(\mathrm ik)(x-x^*)\), so that \(x=x^*+y/(\mathrm ik)\) and \(x^*\) corresponds to \(y=0\).
The rescaled field satisfies
\begin{equation}\label{eq:interface_rescaled_equation}
-\Delta_y\bigl[(\mathrm{i}k)\hat u(x,k)\bigr]+c(x)^{-2}(\mathrm{i}k)\hat u(x,k)=c(x)^{-2}f(x).
\end{equation}
The preceding bound on \((\mathrm ik)\hat u\), together with the boundedness of \(f\) and \(c^{-2}\), shows that the Laplacian of the rescaled field is bounded independently of \(k\).
Local regularity and estimates for the Poisson equation give uniform \(W^{2,p}\) bounds on each fixed ball, for every \(1<p<\infty\); see \cite[Theorem~9.11 and Lemma~9.16]{gilbarg2001elliptic}.
Choosing \(p>3\), the compact Sobolev embedding \cite[Theorem~7.26]{gilbarg2001elliptic} and a diagonal argument give a subsequence converging in \(C^1\) on compact subsets of \(\mathbb R^3\).

In a Lipschitz graph chart \(z=(z',z_3)\) centred at \(x^*\), the domain \(D\) is locally given by \(z_3>g(z')\), with \(g(0)=0\).
Differentiability of the boundary graph at the origin gives
\[
g(z')=\nabla g(0)\cdot z'+o(|z'|),
\qquad
(\mathrm ik)g\!\left(\frac{z'}{\mathrm ik}\right)
\longrightarrow\nabla g(0)\cdot z'.
\]
The second convergence is uniform on bounded sets of \(z'\).
For a fixed point off the tangent plane, the sign of \(z_3-(\mathrm ik)g(z'/(\mathrm ik))\) therefore eventually agrees with that of \(z_3-\nabla g(0)\cdot z'\).
In the \(y\)-coordinates chosen above, the tangent plane is \(y_3=0\) and its inward side is \(y_3>0\).
Consequently,
\[
\mathbf{1}_D\!\left(x^*+\frac{y}{\mathrm ik}\right)
\longrightarrow\mathbf{1}_{\{y_3>0\}}(y)
\qquad\text{for }y_3\ne0.
\]
Since the plane \(y_3=0\) has Lebesgue measure zero, this proves the asserted almost everywhere convergence of the characteristic functions.

For every fixed \(y\), the point \(x=x^*+y/(\mathrm ik)\) tends to the same point \(x^*\) as \(\mathrm ik\to+\infty\).
If \(y_3>0\), the characteristic-function convergence shows that this point lies in \(D\) for all sufficiently large \(\mathrm ik\).
The assumed interior limits therefore give
\[
c\!\left(x^*+\frac{y}{\mathrm ik}\right)\longrightarrow c(x^*),
\qquad
f\!\left(x^*+\frac{y}{\mathrm ik}\right)\longrightarrow f(x^*).
\]
These limiting values are independent of \(y\), since \(x^*\) is fixed.
For \(y_3<0\), the point eventually lies outside \(D\), where \(c=1\) and \(f=0\).
Thus the limiting coefficients and forcing depend only on which side of the plane contains \(y\).
Let \(U_*\) denote any subsequential limit of the rescaled fields; along the corresponding subsequence,
\[
(\mathrm ik)\hat u\!\left(x^*+\frac{y}{\mathrm ik},k\right)
\longrightarrow U_*(y)
\quad\text{in }C^1_{\mathrm{loc}}(\mathbb R^3).
\]
The uniform bounds permit dominated convergence in the weak form of \eqref{eq:interface_rescaled_equation}, so \(U_*\) is a bounded solution of
\begin{equation}\label{eq:interface_halfspace_problem}
\begin{cases}
-\Delta U_*+c(x^*)^{-2}U_*=c(x^*)^{-2}f(x^*),&y_3>0,\\
-\Delta U_*+U_*=0,&y_3<0.
\end{cases}
\end{equation}
The \(C^1\) convergence holds on balls crossing the interface, so both \(U_*\) and \(\partial_{y_3}U_*\) are continuous at \(y_3=0\).

The problem \eqref{eq:interface_halfspace_problem}, with these continuity conditions, has at most one bounded solution.
Indeed, let \(w\) be the difference of two such solutions and choose \(0<\delta<\min\{1,c(x^*)^{-1}\}\).
On the cube \(|y_j|<R\), the function \(\sum_{j=1}^3\cosh(\delta y_j)/\cosh(\delta R)\) is at least one on the boundary.
It is a smooth supersolution of the homogeneous problem, since both \(1-\delta^2\) and \(c(x^*)^{-2}-\delta^2\) are positive.
Applying the same weak comparison principle to the rotated real parts of \(w\) gives
\begin{equation}\label{eq:interface_uniqueness_bound}
|w(y)|\leq\|w\|_\infty
\sum_{j=1}^3\frac{\cosh(\delta y_j)}{\cosh(\delta R)}.
\end{equation}
For each fixed \(y\), letting \(R\to\infty\) gives \(w(y)=0\), proving uniqueness.

The coefficients and forcing in \eqref{eq:interface_halfspace_problem} are constant in each half-space, and the interface is the plane \(y_3=0\).
For any \(a'\in\mathbb R^2\), the tangential translate \(U_*(y'+a',y_3)\) therefore satisfies the same equations and the same continuity conditions at \(y_3=0\), and remains bounded.
By the uniqueness just proved,
\[
U_*(y'+a',y_3)=U_*(y',y_3)
\qquad\text{for every }a'\in\mathbb R^2.
\]
Since \(a'\) is arbitrary, \(U_*\) is independent of \(y'\), and we write \(U_*=U_*(y_3)\).
Boundedness on the two half-lines and continuity at zero give
\[
U_*(y_3)=
\begin{cases}
U_*(0)\exp(y_3),&y_3<0,\\[2mm]
f(x^*)+\bigl(U_*(0)-f(x^*)\bigr)\exp\!\bigl(-y_3/c(x^*)\bigr),&y_3>0.
\end{cases}
\]
The continuity of the first derivative established above yields
\[
U_*(0)=\frac{f(x^*)-U_*(0)}{c(x^*)},
\qquad
U_*(0)=\frac{f(x^*)}{1+c(x^*)}.
\]
The compactness obtained from \eqref{eq:interface_rescaled_equation} applies along every sequence with \(\mathrm ik\to+\infty\).
All subsequential limits have this same value at \(y=0\), corresponding to \(x=x^*\), so the stated limit holds for the whole family.

Finally, if \(f=0\) and \(c=1\) near \(x^*\), the same rescaling and compactness give \(-\Delta U_*+U_*=0\) on \(\mathbb R^3\).
The comparison argument in \eqref{eq:interface_uniqueness_bound}, now with \(0<\delta<1\), shows that its only bounded solution is zero.
This proves the background limit.

The proof is complete.
\end{proof}
\endgroup

\begingroup
The following proposition gives the required end-face relation for the sources in Theorem~\ref{thm:unique_f_c}.

\begin{proposition}\label{prop:TATPAT_endface_limit}
Let \((f,c)\) be a TAT/PAT configuration satisfying \eqref{eq:general_support} and \eqref{eq:general_bounds}, with \(f=0\) and \(c=1\) for \(x_3\notin[a,b]\).
Assume that
\[
f\in C([a,b];L^2(\mathbb R^2)),
\qquad
c\in L^\infty(\mathbb R^3)\cap C(\mathbb R^2\times(a,b)),
\]
where the continuity of \(f\) is understood from inside the strip.
Assume also that \(c\) has a continuous interior extension to \(x_3=b\).
Let \(\hat u\) be the transformed field in \eqref{eq:PDE_hat_u_three}.
Then, along \(k\in-\mathrm i\mathbb R_+\),
\begin{equation}\label{eq:TATPAT_endface_limit}
(\mathrm i k)\hat u(\cdot,b,k)
\longrightarrow\frac{f(\cdot,b-)}{1+c(\cdot,b-)}
\quad\text{in }L^2_{\mathrm{loc}}(\mathbb R^2)
\end{equation}
as \(\mathrm i k\to+\infty\).
Here \(f(\cdot,b-)\) is the interior \(L^2\) trace of \(f\), and \(c(\cdot,b-)\) is the continuous interior trace of \(c\).
\end{proposition}

\begin{proof}
We first obtain a bound for the end-face trace that is uniform in \(k\), so that the limit can be passed from bounded sources to \(f\).
By Subsection~\ref{subsec:temporal_FT}, the transformed field belongs to \(H^2(\mathbb R^3)\) at each fixed frequency, and we use its continuous representative.
Multiplying \eqref{eq:PDE_hat_u_three} by \(\mathrm i k\) gives
\[
\left(-\Delta+\frac{|k|^2}{c^2}\right)
\bigl((\mathrm i k)\hat u\bigr)=\frac{|k|^2}{c^2}f.
\]
Apply the weak comparison principle \cite[Theorem~8.1]{gilbarg2001elliptic} after any constant phase rotation, and compare with the positive free resolvent of \(-\Delta+|k|^2/\max\{1,c_+\}^2\).
Its solution with forcing \(|k|^2c^{-2}|f|\) is a supersolution for each rotated real part of \((\mathrm i k)\hat u\), since \(c^{-2}\ge\max\{1,c_+\}^{-2}\).
The coefficient bounds give
\[
|(\mathrm i k)\hat u(x,k)|
\le C|k|^2\int_{\mathbb R^3}
\frac{\exp\!\left(-|k||x-y|/\max\{1,c_+\}\right)}
{4\pi|x-y|}\,|f(y)|\,\mathrm dy.
\]
Both resolvent fields are in \(H^2(\mathbb R^3)\), so their continuous representatives satisfy the comparison on the end plane.
At each fixed \(y_3<b\), integration in polar coordinates shows that the tangential \(L^1\) norm of the free kernel at \(x_3=b\) is
\[
\frac{\max\{1,c_+\}}{2|k|}
\exp\!\left(-\frac{|k|(b-y_3)}{\max\{1,c_+\}}\right).
\]
Young's and Minkowski's inequalities therefore give
\[
\begin{aligned}
&\|(\mathrm i k)\hat u(\cdot,b,k)\|_{L^2(\mathbb R^2)}\\
&\quad\le C|k|\int_a^b
\exp\!\left(-\frac{|k|(b-y_3)}{\max\{1,c_+\}}\right)
\|f(\cdot,y_3)\|_{L^2(\mathbb R^2)}\,\mathrm dy_3.
\end{aligned}
\]
The normal integral satisfies
\[
|k|\int_a^b
\exp\!\left(-\frac{|k|(b-y_3)}{\max\{1,c_+\}}\right)\,\mathrm dy_3
\le\max\{1,c_+\},
\]
so we obtain
\begin{equation}\label{eq:TATPAT_uniform_face_L2}
\|(\mathrm i k)\hat u(\cdot,b,k)\|_{L^2(\mathbb R^2)}
\le C\|f\|_{C([a,b];L^2(\mathbb R^2))}.
\end{equation}
Here \(C\) depends only on the uniform coefficient bounds and is independent of \(k\) and \(f\).
By linearity, the same estimate applies to differences of sources with \(c\) fixed.

Let \(f_j\) be mollifications of \(f\) in the tangential variables, using compactly supported smooth mollifiers and extending each source by zero outside \([a,b]\).
The slices \(\{f(\cdot,x_3):a\le x_3\le b\}\) form a compact subset of \(L^2(\mathbb R^2)\), so mollification converges uniformly over these slices:
\[
\|f_j-f\|_{C([a,b];L^2(\mathbb R^2))}\longrightarrow0.
\]
For each fixed \(j\), \(f_j\) is bounded and continuous up to both end planes from inside the strip, with compact support in the tangential variables.
Let \(\hat u_j\) solve \eqref{eq:PDE_hat_u_three} with source \(f_j\) and the same sound speed \(c\).
The compact supports of \(f_j\) and \(1-c\) can be enclosed in a bounded domain, so Lemma~\ref{lem:piecewise_interface_trace} applies locally at the flat face and gives
\[
(\mathrm i k)\hat u_j(x',b,k)
\longrightarrow\frac{f_j(x',b-)}{1+c(x',b-)}
\quad\text{for every }x'\in\mathbb R^2.
\]
The bound \(\|(\mathrm i k)\hat u_j\|_\infty\le\|f_j\|_\infty\) from the weak maximum principle allows dominated convergence on each compact \(K\subset\mathbb R^2\).
Thus, for each fixed \(j\), this convergence also holds in \(L^2(K)\).

We now use \eqref{eq:TATPAT_uniform_face_L2} to control \(\hat u-\hat u_j\).
Since \(1/(1+c(\cdot,b-))\) is bounded, the triangle inequality and the convergence for each fixed \(j\) yield
\[
\begin{aligned}
&\limsup_{\mathrm i k\to+\infty}
\left\|(\mathrm i k)\hat u(\cdot,b,k)
-\frac{f(\cdot,b-)}{1+c(\cdot,b-)}\right\|_{L^2(K)}\\
&\qquad\le C\|f-f_j\|_{C([a,b];L^2(\mathbb R^2))}.
\end{aligned}
\]
Letting \(j\to\infty\) proves \eqref{eq:TATPAT_endface_limit}.

The proof is complete.
\end{proof}

The source assumptions of Theorem~\ref{thm:unique_f_c} give \(f,\tilde f\in C([a,b];L^2(\mathbb R^2))\), so Proposition~\ref{prop:TATPAT_endface_limit} applies to both configurations.
We next use the common measurement to equate their weighted source traces at the end face and combine this relation with the recovered quotient to separate \(f\) and \(c\).
\endgroup

\subsection{Step 5: Frequency-domain identity}\label{subsec:TATPAT_frequency_identity}

\leavevmode\begingroup
Fix \(k\in-\mathrm i\mathbb R_+\).
To derive the second integral identity, we use free test waves of the form
\begin{equation}\label{eq SS}
w_k(x,t)=e^{-\mathrm i kt}v(x),
\end{equation}
where \(v\in H^1(\Omega)\) is a weak solution of
\begin{equation}\label{eq:v_helmholtz}
-\Delta v-k^2v=0\quad\text{in }\Omega.
\end{equation}
Thus \(w_k\) solves the free wave equation in \(\Omega\times\mathbb R_+\) in the weak sense, and its temporal factor decays exponentially because \(k\in-\mathrm i\mathbb R_+\).

By Subsection~\ref{subsec:temporal_FT}, the transformed fields \(\hat u\) and \(\widehat{\tilde u}\) belong to \(H^2(\mathbb R^3)\) at each fixed frequency.
Since the observations agree, Lemma~\ref{lem:boundary_flux_from_dirichlet} gives equal Dirichlet and normal traces of these fields on \(\partial\Omega\).
These common Cauchy data eliminate the boundary terms when Green's identity is applied to their difference, yielding the following integral identity.
\endgroup

\begingroup
\begin{lemma}\label{lem:second_identity}
Let \((f,c)\) and \((\tilde f,\tilde c)\) be two TAT/PAT configurations satisfying \eqref{eq:general_support} and \eqref{eq:general_bounds}, with
\[
f,\tilde f\in L^2(\mathbb R^3),
\qquad
c,\tilde c\in L^\infty(\mathbb R^3).
\]
Let \(\hat u\) and \(\widehat{\tilde u}\) be the corresponding transformed fields, and suppose that the observations agree as in \eqref{eq:condition}.
For every \(k\in-\mathrm i\mathbb R_+\) and every weak solution \(v\in H^1(\Omega)\) of \eqref{eq:v_helmholtz}, one has
\begin{equation}\label{eq:aux_int6}
\int_\Omega\left[\frac{\mathrm i k f}{c^2}
+k^2\left(\frac1{c^2}-1\right)\hat u\right]v\,\mathrm dx=\int_\Omega\left[\frac{\mathrm i k\tilde f}{\tilde c^2}+k^2\left(\frac1{\tilde c^2}-1\right)\widehat{\tilde u}\right]v\,\mathrm dx.
\end{equation}
\end{lemma}

\begin{proof}
We write both transformed equations with the same free operator \(-\Delta-k^2\), which also annihilates the test function \(v\).
Rewriting \eqref{eq:PDE_hat_u_three} gives
\begin{equation}\label{eq:modified}
(-\Delta-k^2)\hat u
=\frac{\mathrm i k f}{c^2}
+k^2\left(\frac1{c^2}-1\right)\hat u
\quad\text{in }\mathbb R^3,
\end{equation}
and the analogous equation holds for the second configuration.

By Subsection~\ref{subsec:temporal_FT}, \(\hat u,\widehat{\tilde u}\in H^2(\mathbb R^3)\).
Also, \(v\in H^1(\Omega)\) and \(\Delta v=-k^2v\in L^2(\Omega)\), so its weak normal trace is well defined.
We may therefore use the weak Green identity with the bilinear boundary pairings of Subsection~\ref{subsec:weak-boundary-data}, without complex conjugation.
By Lemma~\ref{lem:boundary_flux_from_dirichlet} and the trace compatibility in that subsection, the common observations give
\[
\hat u-\widehat{\tilde u}=0,
\qquad
\partial_\nu(\hat u-\widehat{\tilde u})=0
\quad\text{on }\partial\Omega.
\]
The common Cauchy data make both boundary pairings in the weak Green identity vanish. Hence
\[
\begin{aligned}
&\int_\Omega v(-\Delta-k^2)(\hat u-\widehat{\tilde u})\,\mathrm dx\\
&\qquad=\int_\Omega(\hat u-\widehat{\tilde u})(-\Delta-k^2)v\,\mathrm dx=0,
\end{aligned}
\]
where the last equality follows from \eqref{eq:v_helmholtz}.
Substituting \eqref{eq:modified} and its tilded counterpart proves \eqref{eq:aux_int6}.
\end{proof}

In Subsection~\ref{subsec:TATPAT_final_recovery}, we use \eqref{eq:aux_int6} to obtain equality of the transformed traces on the common end face.
\endgroup

\subsection{Step 6: Proof of Theorem~\ref{thm:unique_f_c}}\label{subsec:TATPAT_final_recovery}

\leavevmode\begingroup
\begin{proof}[Proof of Theorem~\ref{thm:unique_f_c}]
To separate \(f\) and \(c\), we combine the recovered quotient \eqref{eq:fc_ratio_unique} with a second weighted source relation at the common end face \(x_3=b\) chosen in Subsection~\ref{sec:prooff}.
We first show that the common measurement gives equal transformed traces there; their high-frequency limits will provide the relation involving \(f/(1+c)\).

Fix \(k\in-\mathrm i\mathbb R_+\) and \(m\in\mathbb R^2\), and choose
\begin{equation}\label{eq:TATPAT_endface_test}
v(x)=\exp\!\left(-\mathrm i m\cdot x'
+(x_3-b)\sqrt{|m|^2-k^2}\right).
\end{equation}
The square root is positive, and \(v\in H^1(\Omega)\) satisfies \eqref{eq:v_helmholtz}.
By the common support interval in Admissibility~\ref{ass:x1_only_fc}, the right-hand side of \eqref{eq:modified} is compactly supported in \(\overline\Omega\cap\{x_3\le b\}\).
It belongs to \(L^2(\mathbb R^3)\), since \(f,\hat u\in L^2(\mathbb R^3)\) and the coefficients are bounded, and hence also to \(L^1(\mathbb R^3)\).
Taking the Fourier transform in \(x'\) reduces its left-hand operator to \(-\partial_3^2+|m|^2-k^2\), whose decaying Green kernel is
\[
\frac{\exp\!\left(-|x_3-y_3|\sqrt{|m|^2-k^2}\right)}
{2\sqrt{|m|^2-k^2}}.
\]
The transformed field is the \(L^2\) solution of the free equation \eqref{eq:modified}, so this kernel gives its representation.
At \(x_3=b\), every point in the support of the right-hand side has \(y_3\le b\).
The normal factor in the kernel therefore becomes \(\exp\!\left((y_3-b)\sqrt{|m|^2-k^2}\right)\).
The test function \(v(y)\) in \eqref{eq:TATPAT_endface_test} is the product of this normal factor and the tangential Fourier factor \(e^{-\mathrm i m\cdot y'}\).
The kernel representation gives
\begin{equation}\label{eq:TATPAT_face_representation}
\begin{aligned}
&2\sqrt{|m|^2-k^2}
\int_{\mathbb R^2}\hat u(x',b,k)e^{-\mathrm i m\cdot x'}\,\mathrm dx'\\
&\quad=\int_\Omega\left[\frac{\mathrm i k f}{c^2}
+k^2\left(\frac1{c^2}-1\right)\hat u\right]v\,\mathrm dx,
\end{aligned}
\end{equation}
with \(v\) as in \eqref{eq:TATPAT_endface_test}.
The corresponding formula holds for the second configuration.
By \eqref{eq:aux_int6} in Lemma~\ref{lem:second_identity}, the right-hand sides of the two formulas \eqref{eq:TATPAT_face_representation} agree.

The end-face traces are continuous by the global \(H^2\) regularity of the transformed fields in Subsection~\ref{subsec:temporal_FT}.
They belong to \(L^1(\mathbb R^2)\), since the full-space free kernel decays exponentially away from the compact support in \eqref{eq:modified}.
Their Fourier transforms agree for every \(m\in\mathbb R^2\), so Fourier uniqueness and continuity give
\begin{equation}\label{eq:TATPAT_face_trace}
\hat u(x',b,k)=\widehat{\tilde u}(x',b,k),
\qquad x'\in\mathbb R^2,\quad k\in-\mathrm i\mathbb R_+.
\end{equation}
Proposition~\ref{prop:TATPAT_endface_limit} applies to both configurations.
Taking the \(L^2_{\mathrm{loc}}\) limits in \eqref{eq:TATPAT_face_trace} using \eqref{eq:TATPAT_endface_limit} gives the first weighted source relation at the end face, involving \(f/(1+c)\).

We next pass the recovered quotient to the same end face.
The equality \eqref{eq:fc_ratio_unique} also holds outside \(\Omega\), where both sources vanish.
By Fubini's theorem, for almost every \(x_3\in(a,b)\) the quotient equality holds in \(L^2(\mathbb R^2)\).
Such slices can be chosen arbitrarily close to \(b\), since the exceptional set has measure zero.
The continuity of \(\tilde\phi\), the interior continuity of the sound speeds at the chosen end plane, and their uniform positive bounds imply, by dominated convergence, that
\begin{equation}\label{eq:TATPAT_quotient_face_limit}
\frac{p(\cdot)\tilde\phi(x_3)}{c(\cdot,x_3)^2}
\longrightarrow
\frac{p(\cdot)\tilde\phi(b)}{c(\cdot,b-)^2}
\quad\text{in }L^2(\mathbb R^2)
\end{equation}
as \(x_3\uparrow b\), and the same holds for the second configuration.
Indeed, the absolute values are bounded by a constant times
\(\|\tilde\phi\|_\infty|p|\), an \(L^2(\mathbb R^2)\) function.
Taking the limit along slices on which the quotient equality holds, and combining it with the high-frequency relation, we obtain, almost everywhere in \(\mathbb R^2\),
\[
\begin{aligned}
\frac{f(\cdot,b-)}{1+c(\cdot,b-)}
&=\frac{\tilde f(\cdot,b-)}{1+\tilde c(\cdot,b-)},\\
\frac{f(\cdot,b-)}{c(\cdot,b-)^2}
&=\frac{\tilde f(\cdot,b-)}{\tilde c(\cdot,b-)^2}.
\end{aligned}
\]
We now use the common source profile in \eqref{eq fphi2}.
The end face was chosen so that \(\tilde\phi(b)\ne0\), and cancelling this common factor in both relations gives, for almost every \(x'\in\mathbb R^2\),
\begin{equation}\label{eq:TATPAT_face_relations}
\frac{p(x')}{1+c(x',b-)}
=\frac{\tilde p(x')}{1+\tilde c(x',b-)},
\qquad
\frac{p(x')}{c(x',b-)^2}
=\frac{\tilde p(x')}{\tilde c(x',b-)^2}.
\end{equation}
If the common value in the second relation is zero, positivity of the sound speeds gives \(p(x')=\tilde p(x')=0\).
Otherwise, division of the first relation by the second gives
\[
\frac{c(x',b-)^2}{1+c(x',b-)}
=\frac{\tilde c(x',b-)^2}{1+\tilde c(x',b-)}.
\]
The map \(c\mapsto c^2/(1+c)\) has derivative \(c(c+2)/(1+c)^2>0\) for \(c>0\).
It follows that \(c(x',b-)=\tilde c(x',b-)\), and substitution into \eqref{eq:TATPAT_face_relations} gives \(p(x')=\tilde p(x')\).
Both cases therefore give \(p=\tilde p\) almost everywhere in \(\mathbb R^2\).
The same prescribed profile determines the dependence of both sources on \(x_3\) in \eqref{eq fphi2}, so this equality gives \(f=\tilde f\) almost everywhere in \(\Omega\).

Substituting this equality into \eqref{eq:fc_ratio_unique} gives \((c^{-2}-\tilde c^{-2})f=0\) almost everywhere.
If the continuous coefficient difference were nonzero at a point of \(\operatorname{supp}(f)\) inside the strip, it would be bounded away from zero in a neighbourhood of that point.
The product equality would force \(f=0\) almost everywhere in that neighbourhood, contradicting the definition of the support of an \(L^2\) function.
Thus \(c=\tilde c\) on \(\operatorname{supp}(f)\cap(\mathbb R^2\times(a,b))\).
The two end planes have measure zero.
Consequently,
\[
f=\tilde f\quad\text{a.e. in }\Omega,\qquad
c=\tilde c\quad\text{a.e. on }\operatorname{supp}(f).\qedhere
\]
\end{proof}
\endgroup
\endgroup

\section{Joint recovery of piecewise configurations}\label{sec:piecewise}

{
Theorems~\ref{thm0} and~\ref{thm:unique_f_c} use prescribed profiles in the distinguished variable.
We now treat piecewise configurations whose support components and interfaces are also unknown.
}

We first consider piecewise-regular configurations with product-type support components.
{Subsection~\ref{subsec pice1} gives a free-boundary proof of support identification for a single component.
Subsection~\ref{subsec pice2} uses the interface trace to identify finitely many components and recover the source and sound speed.}
This part uses Admissibility~\ref{ass:piecewise_jump}, which specifies the product geometry, local coefficient regularity, and jump conditions near exposed lateral interfaces.

We then treat the piecewise constant case in Subsection~\ref{subsec:piecewise-constant}.
In that setting the coefficients are constant on each component, while the components are allowed to be general Lipschitz domains.
This part uses Admissibility~\ref{ass:piecewise_constant_general_domains}, which specifies the general Lipschitz component geometry and the piecewise constant coefficient structure.

\subsection{Step 1: The single-component case}\label{subsec pice1}

\begin{proof}[Proof of Theorem~\ref{thm:unique_f_c_piecewise}, Step 1]
We first treat the single-component case \(N=M=1\) and prove that
\[
\Omega_1=\widetilde\Omega_1.
\]
{{We use the transform in Subsection~\ref{subsec:temporal_FT} with \(k\in-\mathrm{i}\mathbb R_+\), so that \(\mathrm{i}k>0\) and \eqref{FourierTransform} holds for the weak solutions.}}
In the single-component case \(N=M=1\), the configurations reduce to
\[
f=f_1\chi_{\Omega_1},
\qquad
c=1+(c_1-1)\chi_{\Omega_1},
\]
and
\[
\tilde f=\tilde f_1\chi_{\widetilde\Omega_1},
\qquad
\tilde c=1+(\tilde c_1-1)\chi_{\widetilde\Omega_1}.
\]
In the three-dimensional case with $\sigma\equiv1$ and $h\equiv0$, the frequency-domain equations become
\begin{equation}\label{FourierTransformPiecewise}
\begin{cases}
-\Delta \hat u(x,k)-k^2\hat u(x,k)=0,
& x\in \mathbb R^3\setminus \Omega_1,\\
-\Delta \hat u(x,k)-\dfrac{k^2}{c(x)^2}\hat u(x,k)
=\mathrm{i}k\dfrac{f(x)}{c(x)^2},
& x\in \Omega_1,
\end{cases}
\end{equation}
and, similarly,
\begin{equation}\label{FourierTransformPiecewiseTilde}
\begin{cases}
-\Delta \hat{\tilde u}(x,k)-k^2\hat{\tilde u}(x,k)=0,
& x\in \mathbb R^3\setminus \widetilde\Omega_1,\\
-\Delta \hat{\tilde u}(x,k)-\dfrac{k^2}{\tilde c(x)^2}\hat{\tilde u}(x,k)
=\mathrm{i}k\dfrac{\tilde f(x)}{\tilde c(x)^2},
& x\in \widetilde\Omega_1.
\end{cases}
\end{equation}

{
By \eqref{eq measure 2} and Lemma~\ref{lem:boundary_flux_from_dirichlet}, the two frequency-domain solutions have the same Cauchy data on $\partial\Omega$.
}
We define
\begin{equation}\label{eq ww}
w:=\hat{\tilde u}-\hat u.
\end{equation}
Then
\[
w=0,
\qquad
\partial_\nu w=0
\qquad
\text{on }\partial\Omega,\quad {k\in-\mathrm{i}\mathbb R_+.}
\]
Moreover, by \eqref{FourierTransformPiecewise} and \eqref{FourierTransformPiecewiseTilde}, the function $w$ satisfies
\[
-\Delta w-k^2w=0
\qquad
\text{in } \mathbb R^3\setminus \overline{\Omega_1\cup\widetilde\Omega_1}.
\]
By the unique continuation principle for the Helmholtz equation and the exterior connectedness assumption in Admissibility~\ref{ass:piecewise_jump}, we obtain
\[
w=0
\]
in the connected exterior component of
\[
\mathbb R^3\setminus \overline{\Omega_1\cup\widetilde\Omega_1}
\]
that is connected to \(\partial\Omega\).

We first record a simple geometric observation used to choose the local comparison point.

\begin{lemma}\label{lem:lateral-exposed-point}
Let
\[
D=\omega\times I,
\qquad
\widetilde D=\widetilde\omega\times\widetilde I,
\]
where \(\omega,\widetilde\omega\subset\mathbb R^2\) are bounded simply connected \(C^{1,1}\) domains and \(I,\widetilde I\subset\mathbb R\) are bounded open intervals.
If
\[
D\neq \widetilde D,
\]
then at least one of the two lateral exposed sets is nonempty:
\[
(\partial\omega\times I)\setminus\overline{\widetilde D}\neq\emptyset
\]
or
\[
(\partial\widetilde\omega\times \widetilde I)\setminus\overline D\neq\emptyset.
\]
\end{lemma}

\begin{proof}
If \(I\neq\widetilde I\), then either
\[
I\setminus\overline{\widetilde I}\neq\emptyset
\]
or
\[
\widetilde I\setminus\overline I\neq\emptyset.
\]
In the first case, choose \(x_3^*\in I\setminus\overline{\widetilde I}\) and \(x'^*\in\partial\omega\).
Then
\[
(x'^*,x_3^*)\in(\partial\omega\times I)\setminus\overline{\widetilde D}.
\]
Hence
\[
(\partial\omega\times I)\setminus\overline{\widetilde D}\neq\emptyset.
\]
The second case gives
\[
(\partial\widetilde\omega\times\widetilde I)\setminus\overline D\neq\emptyset
\]
in the same way.

It remains to consider the case \(I=\widetilde I\).
Since \(D\neq\widetilde D\), we have \(\omega\neq\widetilde\omega\).
We claim that at least one of
\[
\partial\omega\setminus\overline{\widetilde\omega},
\qquad
\partial\widetilde\omega\setminus\overline\omega
\]
is nonempty.
Otherwise,
\[
\partial\omega\subset\overline{\widetilde\omega},
\qquad
\partial\widetilde\omega\subset\overline\omega.
\]
{The complement of \(\overline{\widetilde\omega}\) is connected and unbounded.
If it met \(\omega\), a path in this complement to infinity would cross \(\partial\omega\), contradicting the first inclusion.
Thus \(\omega\subset\overline{\widetilde\omega}\), and the reverse inclusion of the closures follows by interchanging the two domains.
Consequently}
\[
\partial\omega=\partial\widetilde\omega.
\]
As \(C^{1,1}\) domains are regular open sets, this implies
\[
\omega=\widetilde\omega,
\]
a contradiction.

If
\[
\partial\omega\setminus\overline{\widetilde\omega}\neq\emptyset,
\]
choose \(x'^*\in\partial\omega\setminus\overline{\widetilde\omega}\) and any \(x_3^*\in I\).
Then
\[
(x'^*,x_3^*)\in(\partial\omega\times I)\setminus\overline{\widetilde D}.
\]
The case
\[
\partial\widetilde\omega\setminus\overline\omega\neq\emptyset
\]
is identical.
The proof is complete.
\end{proof}

We claim that
\begin{equation}\label{eq single domain}
\Omega_1=\widetilde\Omega_1.
\end{equation}
Assume, to the contrary, that
\[
\Omega_1\neq\widetilde\Omega_1.
\]
By Admissibility condition~\ref{ass:piecewise_jump}, after a rigid change of coordinates, we may write
\[
\Omega_1=\omega_1\times I_1,
\qquad
\widetilde\Omega_1=\widetilde\omega_1\times\widetilde I_1,
\]
where \(\omega_1,\widetilde\omega_1\subset\mathbb R^2\) are bounded simply connected \(C^{1,1}\) domains and \(I_1,\widetilde I_1\subset\mathbb R\) are bounded open intervals.
By Lemma~\ref{lem:lateral-exposed-point}, at least one of the two components has an exposed lateral boundary point relative to the other.
Interchanging the two configurations if necessary, we may assume that
\[
(\partial\widetilde\omega_1\times \widetilde I_1)\setminus\overline{\Omega_1}\neq\emptyset.
\]
Hence we may choose an exposed lateral boundary point
\[
x^*=(x'^*,x_3^*)\in
(\partial\widetilde\omega_1\times \widetilde I_1)\setminus\overline{\Omega_1}
\]
such that the exterior side of the local boundary patch belongs to the connected component of
\[
\mathbb R^3\setminus\overline{\Omega_1\cup\widetilde\Omega_1}
\]
which is connected to \(\partial\Omega\).
{Indeed, if the intervals differ, choose the height in one interval outside the closure of the other, as in the proof of Lemma~\ref{lem:lateral-exposed-point}.
At that height the exterior of the transverse domain is connected to infinity.
If the intervals coincide, points just outside the exposed lateral patch have transverse coordinates outside both closed bases and can be joined to infinity by a vertical ray.}
For \(\varepsilon>0\) sufficiently small, we have
\[
B_\varepsilon(x^*)\cap \overline{\Omega_1}=\emptyset.
\]
In particular, \(\hat u\) satisfies the homogeneous Helmholtz equation
\[
-\Delta\hat u-k^2\hat u=0
\]
in \(B_\varepsilon(x^*)\).

We now describe the local geometry near \(x^*\).
By the choice of \(x^*\), after shrinking \(\varepsilon\) if necessary, the relevant boundary patch is a lateral cylindrical patch:
\[
\widetilde\Omega_1\cap B_\varepsilon(x^*)
=
(\widetilde\omega_1\times \widetilde I_1)\cap B_\varepsilon(x^*),
\]
and
\[
\partial\widetilde\Omega_1\cap B_\varepsilon(x^*)
=
(\partial\widetilde\omega_1\times \widetilde I_1)\cap B_\varepsilon(x^*).
\]
In particular, the top and bottom faces of the cylindrical component do not occur in this local patch.

We next pass the exterior vanishing of \(w\) to the local lateral boundary patch.
Let \(U\) be a neighbourhood of \(x^*\) such that
\[
U\cap\overline{\Omega_1}=\emptyset.
\]
{The continuity on the component closures makes \(f,\tilde f\) bounded.
Since the principal part is the Laplacian, the equations for \(\hat u\) and \(\hat{\tilde u}\) have bounded lower-order coefficients and right-hand sides locally in \(U\).}
By standard local elliptic regularity and the usual bootstrap argument, see for instance
\cite[Chapters~1--2 and Chapter~4]{grisvard1985elliptic},
\cite[Theorem~9.11]{gilbarg2001elliptic}, we obtain, for every \(U'\Subset U\),
\[
\hat u,\hat{\tilde u}\in W^{2,p}(U')
\qquad
\text{for every }1<p<\infty.
\]
Choosing \(p>3\) and applying the Sobolev embedding theorem gives
\[
w=\hat{\tilde u}-\hat u\in C^{1,\alpha}(\overline{U'})
\qquad
\text{for some }\alpha\in(0,1).
\]
By the choice of \(x^*\), the exterior side of
\(\partial\widetilde\Omega_1\cap U'\) lies in the common exterior component where \(w=0\).
The continuity of \(w\) and \(\nabla w\) therefore implies, after shrinking \(U'\) if necessary, that
\begin{equation}\label{eq local cauchy}
w(x,k)=0,
\qquad
\partial_\nu w(x,k)=0,
\qquad
x\in \partial\widetilde\Omega_1\cap U',\ {k\in-\mathrm{i}\mathbb R_+,}
\end{equation}
on the lateral boundary patch, in the trace sense.
Shrinking \(\varepsilon>0\) if necessary, we may assume that
\[
B_\varepsilon(x^*)\Subset U'.
\]
Hence \eqref{eq local cauchy} holds on
\[
\partial\widetilde\Omega_1\cap B_\varepsilon(x^*).
\]

Taking the difference between \eqref{FourierTransformPiecewise} and
\eqref{FourierTransformPiecewiseTilde} in \(\widetilde\Omega_1\cap B_\varepsilon(x^*)\), we obtain
\begin{equation}\label{eq F-Ftilde}
\Delta w+k^2 w
=
\left(1-\frac{1}{\tilde c^2}\right)k^2\hat{\tilde u}
-\mathrm{i}k\frac{\tilde f}{\tilde c^2}
\qquad
\text{in }\widetilde\Omega_1\cap B_\varepsilon(x^*).
\end{equation}
{For \(W=\operatorname{Re}w\), taking real parts in \eqref{eq F-Ftilde} gives}
\begin{equation}\label{eq Im F-Ftilde}
\Delta W+k^2 W
=
H(x,k)
\qquad
\text{in }\widetilde\Omega_1\cap B_\varepsilon(x^*),
\end{equation}
where
\begin{equation}\label{eq:H-local}
H(x,k):=
{\left(1-\frac{1}{\tilde c(x)^2}\right)k^2\operatorname{Re}\hat{\tilde u}(x,k)
-\mathrm{i}k\frac{\operatorname{Re}\tilde f(x)}{\tilde c(x)^2}.}
\end{equation}
Moreover, \eqref{eq local cauchy} gives
\begin{equation}\label{eq Im local zero cauchy}
W=0,
\qquad
\partial_\nu W=0
\qquad
\text{on }\partial\widetilde\Omega_1\cap B_\varepsilon(x^*).
\end{equation}

We next record the local free-boundary inputs used below.
The first lemma gives the local \(C^{1,1}\) regularity and the one-sidedness.
The second lemma is the Caffarelli-type step, which upgrades the local free boundary to a \(C^1\) boundary and gives \(C^2\) regularity of the solution.
The third lemma is the Kinderlehrer--Nirenberg higher-regularity step.

\begin{lemma}\label{lem local one sign}
{Fix \(k\in-\mathrm{i}\mathbb R_+\) and let \(x^*\) be a point of a Lipschitz boundary patch of \(D\subset\mathbb R^3\).
Assume that \(H\) has a H\"older continuous extension to a neighbourhood of \(x^*\).}
Let \(W\) be a real-valued function satisfying
\begin{equation}\label{eq:onesign-eq}
\Delta W+k^2W=H
\qquad
\text{in }D\cap B_\varepsilon(x^*),
\end{equation}
and
\begin{equation}\label{eq:onesign-cauchy}
W=\partial_\nu W=0
\qquad
\text{on }\partial D\cap B_\varepsilon(x^*).
\end{equation}
Assume that the zero extension of \(W\) across the exterior side of the local boundary patch belongs to the corresponding local \(H^2\) class.
Assume further that
\begin{equation}\label{eq:onesign-negative-H}
H\leq-\gamma<0
\qquad
\text{in }D\cap B_\varepsilon(x^*)
\end{equation}
for some \(\gamma>0\).
Then, after possibly shrinking \(\varepsilon\),
\[
W\in C^{1,1}(\overline{D\cap B_\varepsilon(x^*)})
\]
and
\[
{W\leq0}
\qquad
\text{in }D\cap B_\varepsilon(x^*).
\]
\end{lemma}

\begin{proof}
{In dimension three the \(H^2\) zero extension of \(W\) is locally H\"older continuous and vanishes at \(x^*\).
Thus \(k^2W-H\) has a H\"older continuous extension which is strictly positive in a smaller ball.
There \(-W\) satisfies \(\Delta(-W)=(k^2W-H)\chi_{\{W\ne0\}}\) in the whole ball.
Indeed, the zero Cauchy data remove the interface terms, and the second weak derivatives of \(W\) vanish almost everywhere on its zero set, which therefore has measure zero inside \(D\).
The local \(C^{1,1}\) estimate follows from the no-sign obstacle regularity result in \cite[Section~5]{andersson2013optimal}, since the Newton potential of a H\"older continuous density is locally \(C^{1,1}\).
The exterior cone of the Lipschitz patch gives a uniform lower bound for the thickness of the zero set on small balls centred at the boundary.
The non-negativity argument in \cite[Section~6]{andersson2013optimal} then gives \(-W\geq0\) near \(x^*\).}
\end{proof}

\begin{lemma}\label{lem CPAM 33}
Let \(D\subset\mathbb R^3\) be a domain and let \(x^*\in\partial D\).
Assume that
\begin{equation}\label{eq:caff-boundary-lip}
\partial D\cap B_\varepsilon(x^*)
\quad\text{is Lipschitz}.
\end{equation}
Let \(W\) be a real-valued function satisfying
\begin{equation}\label{eq:caff-eq}
\Delta W=G
\qquad
\text{in }D\cap B_\varepsilon(x^*),
\end{equation}
and
\begin{equation}\label{eq:caff-cauchy}
W=\partial_\nu W=0
\qquad
\text{on }\partial D\cap B_\varepsilon(x^*).
\end{equation}
Assume that
\begin{equation}\label{eq:caff-C11}
W\in C^{1,1}(\overline{D\cap B_\varepsilon(x^*)}),
\end{equation}
that
\begin{equation}\label{eq:caff-W-negative}
W\leq0
\qquad
\text{in }D\cap B_\varepsilon(x^*),
\end{equation}
and that \(G\) admits a \(C^1\) extension \(G^*\) to a neighbourhood of
\[
\overline{D\cap B_\varepsilon(x^*)}
\]
such that
\begin{equation}\label{eq:caff-G-negative}
G^*\leq-\alpha<0
\end{equation}
for some constant \(\alpha>0\).
Then there exists \(0<\varepsilon'<\varepsilon\) such that
\[
\partial D\cap B_{\varepsilon'}(x^*)
\]
is of class \(C^1\), and
\[
W\in C^2(\overline{D\cap B_{\varepsilon'}(x^*)}).
\]
\end{lemma}

\begin{proof}
{This is the local Caffarelli free-boundary regularity result stated in \cite[Section~3]{cakoni2023singularities}.
It applies after translating \(x^*\) to the origin and restricting to a sufficiently small neighbourhood.}
\end{proof}

\begin{lemma}\label{lem CPAM}
Let \(D\subset\mathbb R^3\) be a domain and let \(x^*\in\partial D\).
Let \(W,a,B\) be real-valued functions.
Assume that
\begin{equation}\label{eq:KN-boundary-C1}
\partial D\cap B_\varepsilon(x^*)
\quad\text{is of class }C^1,
\end{equation}
and that
\begin{equation}\label{eq:KN-W-C2}
W\in C^2(\overline{D\cap B_\varepsilon(x^*)}).
\end{equation}
Suppose that
\begin{equation}\label{eq:KN-eq}
\Delta W+a(x)W=B(x)
\qquad
\text{in }D\cap B_\varepsilon(x^*),
\end{equation}
and
\begin{equation}\label{eq:KN-cauchy}
W=\partial_\nu W=0
\qquad
\text{on }\partial D\cap B_\varepsilon(x^*).
\end{equation}
Assume that
\begin{equation}\label{eq:KN-coeff-reg}
a,B\in C^{m,\mu}(\overline{D\cap B_\varepsilon(x^*)})
\end{equation}
for some \(m\geq1\) and \(0<\mu<1\), and that the interior trace satisfies
\begin{equation}\label{eq:KN-nondeg}
{a(x^*)B(x^*)\neq0.}
\end{equation}
Then there exists \(0<\varepsilon'<\varepsilon\) such that
\[
\partial D\cap B_{\varepsilon'}(x^*)
\]
is of class \(C^{m+1,\mu}\).
\end{lemma}

\begin{proof}
{This is the local form of the Kinderlehrer--Nirenberg theorem \cite[Theorem~\(1'\), p.~377]{kinderlehrer1977regularity}, as applied to \(\Delta W+aW=B\) in \cite[Theorem~3.1]{cakoni2023singularities}.}
\end{proof}

We now verify the hypotheses needed to apply Lemmas~\ref{lem local one sign}, \ref{lem CPAM 33}, and~\ref{lem CPAM} in sequence.
Recall that \eqref{eq Im F-Ftilde} has the form
\begin{equation}\label{eq Im F-Ftilde CPAM}
\Delta W+k^2 W=H(x,k)
\qquad
\text{in }\widetilde\Omega_1\cap B_\varepsilon(x^*),
\end{equation}
{where \(H\) is defined in \eqref{eq:H-local}.}
Moreover, by \eqref{eq Im local zero cauchy}, we have
\begin{equation}\label{eq cpam_cauchy_zero}
W=\partial_\nu W=0
\qquad
\text{on }\partial\widetilde\Omega_1\cap B_\varepsilon(x^*).
\end{equation}

{To apply Lemma~\ref{lem local one sign}, we first note that the lateral patch is \(C^{1,1}\).
The local \(W^{2,p}\) regularity of \(\hat{\tilde u}\) and the interior extensions of \(\tilde f,\tilde c\) in Admissibility~\ref{ass:piecewise_jump} give the required H\"older extension of \(H\).
The local \(W^{2,p}\) regularity of \(w\) across the patch and its exterior vanishing give the required \(H^2\) zero extension of \(W\).}

It remains to verify the sign condition \eqref{eq:onesign-negative-H}.
{By \eqref{FourierTransformPiecewiseTilde}, the field \(\hat{\tilde u}/(\mathrm{i}k)\) solves
\[
\left(-\Delta+(\mathrm{i}k)^2\tilde c^{-2}\right)\frac{\hat{\tilde u}}{\mathrm{i}k}
=\frac{\tilde f}{\tilde c^2}\quad\text{in }\mathbb R^3.
\]
The Newton potential of \(|\tilde f|/\tilde c^2\) is bounded since the source is bounded and compactly supported.
Comparison with this potential for both signs of the real part gives \(\|\operatorname{Re}\hat{\tilde u}(\cdot,k)\|_\infty=O(|k|)\) as \(\mathrm{i}k\downarrow0\).
Consequently \(H(x,k)=-(\mathrm{i}k)\operatorname{Re}\tilde f(x)/\tilde c(x)^2+O(|k|^3)\), uniformly up to the interface from the interior.
The lateral jump assumption gives \((\operatorname{Re}\tilde f)/\tilde c^2\geq c_f>0\) in a smaller patch.
For all sufficiently small \(\mathrm{i}k>0\) we therefore have
\begin{equation}\label{eq H negative}
H(x,k)\leq-\frac{\mathrm{i}k\,c_f}{2}<0
\qquad\text{in }\widetilde\Omega_1\cap B_\varepsilon(x^*).
\end{equation}
This is \eqref{eq:onesign-negative-H} with \(\gamma=(\mathrm{i}k)c_f/2\).}
In particular, the interior trace satisfies
\begin{equation}\label{eq k^2kf}
H(x^*,k)\neq0.
\end{equation}

Hence Lemma~\ref{lem local one sign}, applied with \(D=\widetilde\Omega_1\), yields, after possibly shrinking \(\varepsilon\),
\begin{equation}\label{eq W C11}
W\in C^{1,1}(\overline{\widetilde\Omega_1\cap B_\varepsilon(x^*)})
\end{equation}
and
\begin{equation}\label{eq W negative}
{W(x,k)\leq0}
\qquad
\text{in }\widetilde\Omega_1\cap B_\varepsilon(x^*).
\end{equation}

{To apply Lemma~\ref{lem CPAM 33}, set \(G:=H-k^2W\) in \eqref{eq Im F-Ftilde CPAM}.
The lateral patch is Lipschitz, and \eqref{eq cpam_cauchy_zero}, \eqref{eq W C11}, and \eqref{eq W negative} give the required Cauchy data, regularity, and sign.
It remains to verify \eqref{eq:caff-G-negative}.}
By \eqref{eq W C11} and \eqref{eq cpam_cauchy_zero}, there exists \(C>0\) such that
\[
|W(x,k)|
\leq
C\,\operatorname{dist}(x,\partial\widetilde\Omega_1)^2
\qquad
\text{for }x\in \widetilde\Omega_1\cap B_\varepsilon(x^*).
\]
{Fix \(k\) as in \eqref{eq H negative}.
After shrinking the patch, the preceding bound gives \(|k|^2|W|\leq(\mathrm{i}k)c_f/4\), so \eqref{eq H negative} yields \(G=H-k^2W\leq-(\mathrm{i}k)c_f/4\).
The zero extension of \(W\) is \(C^1\) across the patch by the local \(W^{2,p}\) estimates.
Together with the interior extensions of \(\tilde c,\tilde f\) and the smoothness of \(\hat u\), the relation \(\operatorname{Re}\hat{\tilde u}=W+\operatorname{Re}\hat u\) thus gives a \(C^1\) extension of \(G\).}
Therefore \(G\) admits, after possibly shrinking the neighbourhood, a \(C^1\) extension \(G^*\) satisfying
\[
G^*\leq-\alpha<0,
\qquad
{\alpha:=\frac{\mathrm{i}k\,c_f}{4}.}
\]
Thus \eqref{eq:caff-G-negative} is verified.
Lemma~\ref{lem CPAM 33} now gives, after possibly shrinking \(\varepsilon\),
\begin{equation}\label{eq boundary C1 after caff}
\partial\widetilde\Omega_1\cap B_\varepsilon(x^*)\in C^1
\end{equation}
and
\begin{equation}\label{eq W C2 after caff}
W\in C^2(\overline{\widetilde\Omega_1\cap B_\varepsilon(x^*)}).
\end{equation}

{For Lemma~\ref{lem CPAM}, substitute \(\operatorname{Re}\hat{\tilde u}=W+\operatorname{Re}\hat u\) in \eqref{eq Im F-Ftilde CPAM} to obtain}
\begin{equation}\label{eq KN rewritten}
\Delta W+\frac{k^2}{\tilde c(x)^2}W=B(x,k)
\qquad
\text{in }\widetilde\Omega_1\cap B_\varepsilon(x^*),
\end{equation}
where
\[
B(x,k):=
{\left(1-\frac{1}{\tilde c(x)^2}\right)k^2\operatorname{Re}\hat u(x,k)
-\mathrm{i}k\frac{\operatorname{Re}\tilde f(x)}{\tilde c(x)^2}.}
\]
{The required boundary and solution regularity follow from \eqref{eq boundary C1 after caff}--\eqref{eq W C2 after caff}, and the zero Cauchy data are given by \eqref{eq cpam_cauchy_zero}.
Admissibility~\ref{ass:piecewise_jump} and the smoothness of \(\hat u\) in \(B_\varepsilon(x^*)\) give \(a(x)=k^2/\tilde c(x)^2\) and \(B(\cdot,k)\) in \(C^{m,\mu}(\overline{\widetilde\Omega_1\cap B_\varepsilon(x^*)})\).
Since \(W(x^*,k)=0\), we have \(B(x^*,k)=H(x^*,k)\ne0\) by \eqref{eq k^2kf}, while \(a(x^*)\ne0\).
Thus \eqref{eq:KN-coeff-reg} and \eqref{eq:KN-nondeg} hold.}
Therefore Lemma~\ref{lem CPAM} gives, after possibly shrinking \(\varepsilon\), that
\[
\partial\widetilde\Omega_1\cap B_\varepsilon(x^*)\in C^{m+1,\mu}.
\]
Since \(x^*\) is an exposed lateral boundary point, this means that the corresponding transverse boundary \(\partial\widetilde\omega_1\) is \(C^{m+1,\mu}\) near \(x'^*\).
This contradicts the transverse nonsmoothness condition in Admissibility~\ref{ass:piecewise_jump}.
Therefore the assumption \(\Omega_1\neq\widetilde\Omega_1\) is false.
Hence
\[
\Omega_1=\widetilde\Omega_1.
\]
This proves \eqref{eq single domain}.

\end{proof}

\subsection{Step 2: Multiple components}\label{subsec pice2}

\begin{proof}[Proof of Theorem~\ref{thm:unique_f_c_piecewise}, Step 2]
\leavevmode{We now use the high-frequency interface trace in Lemma~\ref{lem:piecewise_interface_trace} to identify finitely many components and recover the coefficients.}
Set
\[
D:=\bigcup_{i=1}^N \Omega_i,
\qquad
\widetilde D:=\bigcup_{j=1}^M \widetilde\Omega_j.
\]
By the piecewise representations \eqref{PiecewiseForm}--\eqref{PiecewiseFormTilde}, we have
\[
f=0,\qquad c=1
\qquad \text{in } \Omega\setminus \overline D,
\]
and
\[
\tilde f=0,\qquad \tilde c=1
\qquad \text{in } \Omega\setminus \overline{\widetilde D}.
\]

{For \(k\in-\mathrm{i}\mathbb R_+\), the observation identity \eqref{eq measure 2} and Lemma~\ref{lem:boundary_flux_from_dirichlet} give equality of the transformed fields outside \(\overline\Omega\).
The difference \(w\) defined in \eqref{eq ww} satisfies the homogeneous Helmholtz equation in \(\mathbb R^3\setminus\overline{D\cup\widetilde D}\).
Unique continuation therefore gives \(w=0\) in its unbounded component.}

We first identify the total supports.
We claim that
\[
D=\widetilde D.
\]
Suppose otherwise.
{Each open union is the interior of its closure, since its Lipschitz components have disjoint closures.
After interchanging the configurations if necessary, choose a point of \(\widetilde D\setminus\overline D\).
In the connected exterior of \(\overline D\), join a point outside \(\overline\Omega\) to this point by a path.
Its first point of contact with \(\overline{\widetilde D}\) lies outside \(\overline D\) and on the boundary of exactly one component \(\widetilde\Omega_j\).
In a sufficiently small Lipschitz chart the exterior side is connected and meets the path before contact.
It therefore belongs to the common exterior component where \(w=0\).

On this relatively open patch choose a point \(x^*\) of differentiability where the interior trace of \(\tilde f_j\) is nonzero.
Such a point exists on a lateral face by the lateral jump condition.
On an end face it exists because the continuous transverse source has full support in the base.
The edges of the cylinder have empty relative interior in its boundary.
Both transformed fields are continuous, so their values agree at \(x^*\) for each \(k\in-\mathrm{i}\mathbb R_+\).
By Lemma~\ref{lem:piecewise_interface_trace}, their limits after multiplication by \(\mathrm{i}k\) are respectively zero and \(\tilde f_j(x^*)/(1+\tilde c_j(x^*))\ne0\).
This is a contradiction.}
Therefore
\[
D=\widetilde D.
\]

Since the components in each family are pairwise disjoint connected domains, equality of the total supports gives equality of the connected components after relabelling.
Thus
\[
N=M,
\qquad
\Omega_i=\widetilde\Omega_i,\quad i=1,\ldots,N.
\]

Having identified the support componentwise, we next show that the Cauchy data coincide on each component boundary.
Since \(\overline D\subset\Omega\), define
\[
G:=\Omega\setminus \overline D.
\]
{By Admissibility~\ref{ass:piecewise_jump}, \(G\) is connected and lies in the common exterior component, so \(w=0\) there.}
{Consequently, for \(k\in-\mathrm{i}\mathbb R_+\) and \(i=1,\ldots,N\),}
\[
\hat u(x,k)=\hat{\tilde u}(x,k),
\qquad
\partial_\nu \hat u(x,k)=\partial_\nu \hat{\tilde u}(x,k)
\qquad
\text{on } \partial\Omega_i,
\]
where \(\nu\) denotes the unit outward normal to \(\Omega_i\), pointing into \(G\).
{By uniqueness of the Laplace transform, this also gives}
\[
u(x,t)=\tilde u(x,t),
\qquad
\partial_\nu u(x,t)=\partial_\nu \tilde u(x,t)
\qquad
\text{on } \partial\Omega_i\times\mathbb R_+,
\quad i=1,\ldots,N.
\]

{Fix an identified component
\[
\Omega_i=\widetilde\Omega_i=\omega_i\times I_i,
\qquad I_i=(a_i,b_i).
\]
The componentwise structural assumption gives
\[
f=p_i(x'),\quad \tilde f=\tilde p_i(x'),\quad
c=\gamma_i(x'),\quad \tilde c=\tilde\gamma_i(x')
\qquad\text{in }\Omega_i.
\]
\begingroup
The zero-extended source contribution and quotient on this component are
\[
f_i\chi_{\Omega_i}
=p_i(x')\chi_{\omega_i}(x')\chi_{I_i}(x_3),
\qquad
\frac{f_i}{c_i^2}\chi_{\Omega_i}
=\frac{p_i(x')}{\gamma_i(x')^2}
\chi_{\omega_i}(x')\chi_{I_i}(x_3),
\]
with the transverse quotient defined on \(\omega_i\) and extended by zero outside it.
The corresponding representations hold for the tilde configuration, with the same interval \(I_i\).
\endgroup
{To recover the quotient, divide \eqref{FourierTransform} by \(\mathrm{i}k\), subtract the two equations and test on \(\Omega_i\) with a function \(\varphi\) harmonic in a neighbourhood of \(\overline{\Omega_i}\).
The common Cauchy data cancel both boundary terms.
By \eqref{eq:general_weak_transform}, the divided fields belong to \(H^1(\mathbb R^3)\).
Testing the full-space equation for \(\hat u/(\mathrm{i}k)\) against its complex conjugate and applying the Sobolev inequality gives
\[
\|\nabla(\hat u/(\mathrm{i}k))\|_{L^2(\mathbb R^3)}\leq C\|f/c^2\|_{L^{6/5}(\mathbb R^3)}.
\]
The same bound holds for the tilded field, so both divided fields are bounded in \(L^2(\Omega_i)\).
Letting \(\mathrm{i}k\downarrow0\) removes the mass terms and gives
\[
\int_{\Omega_i}(f/c^2-\tilde f/\tilde c^2)\varphi\,\mathrm{d}x=0.
\]
Taking \(\varphi(x)=\exp(\mathrm{i}\xi'\cdot x'-|\xi'|x_3)\), the factor \(\int_{a_i}^{b_i}\exp(-|\xi'|x_3)\,\mathrm{d}x_3\) is strictly positive.
Fourier uniqueness in the transverse variables therefore gives}
\begin{equation}\label{eq:component_quotient_identity}
\frac{p_i(x')}{\gamma_i(x')^2}
=\frac{\tilde p_i(x')}{\tilde\gamma_i(x')^2}
\qquad\text{for a.e. }x'\in\omega_i.
\end{equation}

{For source recovery, take \(x'\in\omega_i\) on the open end face at \(x_3=b_i\).
The disjoint component closures in Admissibility~\ref{ass:piecewise_jump} ensure that the exterior side of this face is locally in the background.
The separated form and continuity of \(p_i,\gamma_i\) give the required interior limits; the same holds for the second configuration.
The equality of the transformed traces and Lemma~\ref{lem:piecewise_interface_trace} give
\begin{equation}\label{eq:component_end_face_identity}
\frac{p_i(x')}{1+\gamma_i(x')}=\frac{\tilde p_i(x')}{1+\tilde\gamma_i(x')}.
\end{equation}
By continuity of both sides, \eqref{eq:component_quotient_identity} holds for every \(x'\in\omega_i\).
Where \(p_i\ne0\), divide \eqref{eq:component_end_face_identity} by \eqref{eq:component_quotient_identity}.
Since \(c\mapsto c^2/(1+c)\) is strictly increasing for \(c>0\), the sound speeds agree there, and then so do the sources.
Where \(p_i=0\), \eqref{eq:component_quotient_identity} gives \(\tilde p_i=0\).
Thus \(f=\tilde f\) almost everywhere in \(\Omega_i\).}
}

{Since \(\operatorname{supp}(f_i)=\overline{\Omega_i}\), the set where \(f_i\ne0\) is dense in \(\Omega_i\).
{The assumed continuity of \(c_i,\tilde c_i\) therefore extends their equality from this set to the whole component.}}

Together with \(f=\tilde f\) in \(\Omega_i\), this gives
\[
f=\tilde f,\qquad c=\tilde c
\qquad\text{for a.e. }x\in\Omega_i.\qedhere
\]
\end{proof}

\subsection{The piecewise constant case}\label{subsec:piecewise-constant}

We now prove Theorem~\ref{thm:unique_f_c_piecewise_constant}, which gives the corresponding result for piecewise constant configurations on general Lipschitz domains.

\subsubsection{Step 1: Support identification in the piecewise constant case}

\begin{proof}[Proof of Theorem~\ref{thm:unique_f_c_piecewise_constant}, Step 1]
{We identify the supports using the exposed-patch argument in Subsection~\ref{subsec pice2} and the interface trace in Lemma~\ref{lem:piecewise_interface_trace}.}

Suppose, to the contrary, that
\[
\bigcup_{i=1}^N\Omega_i\neq
\bigcup_{j=1}^M\widetilde\Omega_j.
\]
{After interchanging the configurations if necessary, the same path construction gives an exposed patch of some \(\partial\widetilde\Omega_j\) accessible from the common exterior.
On this patch choose a point of differentiability}
\[
x^*\in\partial\widetilde\Omega_j
\setminus
\overline{\bigcup_{i=1}^N\Omega_i}.
\]
{The full transformed fields agree at \(x^*\) by unique continuation in the common exterior and continuity up to the patch.
Lemma~\ref{lem:piecewise_interface_trace} gives zero for the first configuration and \(\tilde f_j/(1+\tilde c_j)\) for the second after multiplication by \(\mathrm{i}k\) and passage to the limit.
The latter is nonzero by Admissibility~\ref{ass:piecewise_constant_general_domains}, which is a contradiction.}
Therefore
\[
\bigcup_{i=1}^N\Omega_i=
\bigcup_{j=1}^M\widetilde\Omega_j.
\]
After relabelling the connected components, we obtain
\[
N=M,\qquad
\Omega_i=\widetilde\Omega_i,\quad i=1,\ldots,N.
\]
\end{proof}

\subsubsection{Step 2: Determination of the piecewise constants}

\begin{proof}[Proof of Theorem~\ref{thm:unique_f_c_piecewise_constant}, Step 2]
{We next determine the constants using the low-frequency identity from Subsection~\ref{subsec pice2} and the interface trace.}
Since the supports have been identified, after relabelling we have
\[
N=M,\qquad \Omega_i=\widetilde\Omega_i,\quad i=1,\ldots,N.
\]
By the unique-continuation argument used above, the Cauchy data of \(u\) and \(\tilde u\) agree on the boundary of each identified component, in the trace sense from the exterior region. 
{Fix \(D=\Omega_i=\widetilde\Omega_i\).
For \(k\in-\mathrm{i}\mathbb R_+\), divide \eqref{FourierTransform} and its tilded counterpart by \(\mathrm{i}k\), subtract and test on \(D\) with the harmonic function \(1\).
The common normal derivatives cancel the boundary term, giving
\[
|D|\left(\frac{f_i}{c_i^2}-\frac{\tilde f_i}{\tilde c_i^2}\right)
=(\mathrm{i}k)^2\int_D
\left(\frac{\hat u}{\mathrm{i}k\,c_i^2}
-\frac{\hat{\tilde u}}{\mathrm{i}k\,\tilde c_i^2}\right)\,\mathrm{d}x.
\]
By the energy estimate and Sobolev inequality used in Subsection~\ref{subsec pice2}, both \(\hat u/(\mathrm{i}k)\) and \(\hat{\tilde u}/(\mathrm{i}k)\) are uniformly bounded in \(L^2(D)\) as \(\mathrm{i}k\downarrow0\).
The right-hand side is therefore \(O(|k|^2)\), and letting \(\mathrm{i}k\downarrow0\) gives \(f_i/c_i^2=\tilde f_i/\tilde c_i^2\), since \(|D|>0\).

At a point of differentiability of \(\partial D\), the common trace and Lemma~\ref{lem:piecewise_interface_trace} give \(f_i/(1+c_i)=\tilde f_i/(1+\tilde c_i)\).
The sources are nonzero, so division gives \(c_i^2/(1+c_i)=\tilde c_i^2/(1+\tilde c_i)\).
Strict monotonicity on \((0,\infty)\) yields \(c_i=\tilde c_i\), and the quotient identity then gives \(f_i=\tilde f_i\).
This proves the assertion on every component.}
\end{proof}

\section{Simultaneous recovery for general second-order hyperbolic equations}\label{sec:generalhyperbolic}

{
In this section, we prove Theorem~\ref{thm:general} for the general second-order hyperbolic system \eqref{eq:1}, with the {geometry and separated form} prescribed in Admissibility~\ref{ass:general_structure}.
The unknown quantities are the initial displacement \(f\), the initial velocity \(h\), the coefficient \(c\), and the conductivity coefficient \(\sigma\).
The measurement is given by the displacement trace \eqref{eq:2}.
Under the assumptions of the theorem, these four quantities are uniquely determined by a single passive measurement.
}
{
The proof extends the unified time-domain and frequency-domain framework of Theorems~\ref{thm0} and~\ref{thm:unique_f_c} to a systematic decoupling of the four unknown functions.
By combining complementary high- and low-frequency information through three families of test functions, we separate the initial sources from the medium coefficients.
}

\begin{proof}[Proof of Theorem~\ref{thm:general}]

\leavevmode{
We first apply the Green identity argument underlying the second integral identity to the high-frequency limiting problem, using the first test family to identify \(\sigma\).
With this coefficient fixed, we use the first integral identity for the remaining steps: the second test family gives two weighted source terms at high frequency, while the third gives the low-frequency information needed to separate \(f\), \(h\), and \(c\).
}

\subsection{Step 1: Time-domain identity}\label{subsec:general-time-identity}

{We use the weak solutions of Subsection~\ref{subsec Forward Hy} and the boundary traces of Subsection~\ref{subsec:weak-boundary-data}.}
Throughout this section we take \(k\in-\mathrm i\mathbb R_+\) in \eqref{eq FT}.
{
{As explained in Subsection~\ref{subsec:temporal_FT}, the exponentially decaying temporal factor controls the polynomial growth allowed by the weak energy estimate, so no additional local decay or non-trapping condition is needed.}
}
By \eqref{eq:general_weak_transform}, the transformed fields belong to \(H^1(\mathbb R^3)\), and \eqref{FourierTransform} gives their weak conormal traces.
Equality of the observations \eqref{eq:2}, together with Lemma~\ref{lem:boundary_flux_from_dirichlet}, therefore gives
\begin{equation}\label{eq:general_transformed_cauchy}
\hat u=\widehat{\tilde u},\qquad
\sigma\partial_\nu\hat u
=\tilde\sigma\partial_\nu\widehat{\tilde u}
\quad\text{on }\partial\Omega.
\end{equation}
These equalities justify the boundary-term comparison in Lemma~\ref{lem:first_identity}.
Its time-integrated gradient pairings are understood through the transformed fields, as specified in the statement.
\begin{lemma}\label{lem:first_identity}
Let \(\Omega\subset\mathbb R^3\) be a bounded Lipschitz domain and let \((f,h,c,\sigma)\) and \((\tilde f,\tilde h,\tilde c,\tilde\sigma)\) satisfy \eqref{eq:general_support} and \eqref{eq:general_bounds}.
Assume that \(f,h,\tilde f,\tilde h\in L^2(\mathbb R^3)\) and \(c,\sigma,\tilde c,\tilde\sigma\in L^\infty(\mathbb R^3)\).
{The solutions are understood as in Subsection~\ref{subsec Forward Hy}, and their traces as in Subsection~\ref{subsec:weak-boundary-data}.}
Let $u$ and $\tilde u$ be the corresponding solutions of \eqref{eq:1}.
Assume that the measurements coincide in the sense of \eqref{eq:2}.
Fix \(k\in-\mathrm{i}\mathbb R_+\), and let \(v\in H^1(\Omega)\) be a weak solution of
\[
-\nabla\cdot(\sigma\nabla v)-\frac{k^2}{c^2}v=0
\quad\text{in }\Omega.
\]
Since \(\nabla\cdot(\sigma\nabla v)=-k^2c^{-2}v\in L^2(\Omega)\), the weak conormal trace \(\sigma\partial_\nu v\in H^{-1/2}(\partial\Omega)\) is well defined.
Then \(w(x,t)=\exp(-\mathrm{i}kt)v(x)\) satisfies
\[
\frac{1}{c(x)^2}\partial_t^2 w(x,t)-\nabla\cdot(\sigma(x)\nabla w(x,t))=0
\quad\text{in }\Omega\times(0,\infty).
\]
The gradient term in the identity below is defined by
\[
\begin{aligned}
&\int_\Omega\int_0^\infty
(\tilde\sigma-\sigma)(x)\,\nabla\tilde u(x,t)\cdot\nabla w(x,t)\,\mathrm{d}t\,\mathrm{d}x\\
&\qquad:=\int_\Omega
(\tilde\sigma-\sigma)(x)\,\nabla\widehat{\tilde u}(x,k)\cdot\nabla v(x)\,\mathrm{d}x.
\end{aligned}
\]
Here \(\widehat{\tilde u}(\cdot,k)\in H^1(\mathbb R^3)\) is given by \eqref{eq:general_weak_transform} for the second configuration.
The remaining integrals converge in the usual sense by \eqref{eq:general_weak_energy} and the exponential decay of \(w\) and its time derivatives.
Then
\begin{equation}\label{eq:first_identity}
\begin{aligned}
&\int_\Omega\int_0^\infty
\left(\frac{1}{\tilde{c}^2}-\frac{1}{c^2}\right)\tilde{u}(x,t)\,\partial_{tt}w(x,t)\,\mathrm{d}t\mathrm{d}x
+\int_\Omega\left(\frac{\tilde{f}}{\tilde{c}^2}-\frac{f}{c^2}\right)\partial_t w(x,0)\,\mathrm{d}x\\
&\quad
-\int_\Omega\left(\frac{\tilde{h}}{\tilde{c}^2}-\frac{h}{c^2}\right)w(x,0)\,\mathrm{d}x
+\int_\Omega\int_0^\infty
(\tilde{\sigma}-\sigma)(x)\,\nabla\tilde{u}(x,t)\cdot\nabla w(x,t)\,\mathrm{d}t\mathrm{d}x
=0.
\end{aligned}
\end{equation}
\end{lemma}

\begin{proof}
We approximate the initial data in \(L^2(\mathbb R^3)\) by elements of the domains of the corresponding wave operators, keeping the coefficients fixed.
For the resulting solutions, the integrations by parts below are justified and the terminal terms vanish by the energy bounds and the exponential decay of the test function.
At each fixed \(k\in-\mathrm{i}\mathbb R_+\), the coercive estimate for \eqref{FourierTransform} gives convergence of the transformed fields in \(H^1(\mathbb R^3)\).
The same equation gives convergence of their respective flux divergences in \(L^2(\mathbb R^3)\), so their Dirichlet and weak conormal traces converge in \(H^{1/2}(\partial\Omega)\) and \(H^{-1/2}(\partial\Omega)\), respectively.
The approximating solutions need not have equal boundary observations.
We therefore retain all boundary terms when passing to the original data and use \eqref{eq:general_transformed_cauchy} only after taking this limit.
The identities below are understood in this sense.

The function $\tilde u$ satisfies
\[
\frac{1}{\tilde c(x)^2}\partial_t^2\tilde u(x,t)-\nabla\cdot(\tilde\sigma(x)\nabla\tilde u(x,t))=0
\quad\text{in }\Omega\times(0,\infty),
\]
together with the initial conditions
\[
\tilde u(\cdot,0)=\tilde f,
\qquad
\partial_t\tilde u(\cdot,0)=\tilde h.
\]
Multiplying the equation for $\tilde u$ by $w$ and integrating over $\Omega\times(0,\infty)$, we obtain
\begin{equation}\label{eq:first_id_step1}
\int_0^\infty\int_\Omega \frac{1}{\tilde c^2}\tilde u_{tt}\,w\,\mathrm{d}x\,\mathrm{d}t
+\int_0^\infty\int_\Omega \tilde\sigma\nabla\tilde u\cdot\nabla w\,\mathrm{d}x\,\mathrm{d}t
-\int_0^\infty\int_{\partial\Omega}(\nu\cdot\tilde\sigma\nabla\tilde u)\,w\,\mathrm{d}S\,\mathrm{d}t
=0.
\end{equation}

We decompose the second volume integral as
\[
\int_0^\infty\int_\Omega \tilde\sigma\nabla\tilde u\cdot\nabla w\,\mathrm{d}x\,\mathrm{d}t
=
\int_0^\infty\int_\Omega \sigma\nabla\tilde u\cdot\nabla w\,\mathrm{d}x\,\mathrm{d}t
+\int_0^\infty\int_\Omega (\tilde\sigma-\sigma)\nabla\tilde u\cdot\nabla w\,\mathrm{d}x\,\mathrm{d}t.
\]
For the first term on the right-hand side, integration by parts in $x$ gives
\[
\int_0^\infty\int_\Omega \sigma\nabla\tilde u\cdot\nabla w\,\mathrm{d}x\,\mathrm{d}t
=
-\int_0^\infty\int_\Omega \tilde u\,\nabla\cdot(\sigma\nabla w)\,\mathrm{d}x\,\mathrm{d}t
+\int_0^\infty\int_{\partial\Omega}\tilde u\,(\nu\cdot\sigma\nabla w)\,\mathrm{d}S\,\mathrm{d}t.
\]
Since $w$ solves
\[
\frac{1}{c(x)^2}w_{tt}-\nabla\cdot(\sigma(x)\nabla w)=0,
\]
we have
\[
\nabla\cdot(\sigma\nabla w)=\frac{1}{c^2}w_{tt}.
\]
Hence
\begin{equation}\label{eq:first_id_step2}
\int_0^\infty\int_\Omega \sigma\nabla\tilde u\cdot\nabla w\,\mathrm{d}x\,\mathrm{d}t
=
-\int_0^\infty\int_\Omega \frac{1}{c^2}\tilde u\,w_{tt}\,\mathrm{d}x\,\mathrm{d}t
+\int_0^\infty\int_{\partial\Omega}\tilde u\,(\nu\cdot\sigma\nabla w)\,\mathrm{d}S\,\mathrm{d}t.
\end{equation}
Substituting \eqref{eq:first_id_step2} into \eqref{eq:first_id_step1}, we arrive at
\begin{equation}\label{eq:first_id_step3}
\int_0^\infty\int_\Omega \frac{1}{\tilde c^2}\tilde u_{tt}\,w\,\mathrm{d}x\,\mathrm{d}t
-\int_0^\infty\int_\Omega \frac{1}{c^2}\tilde u\,w_{tt}\,\mathrm{d}x\,\mathrm{d}t
+\int_0^\infty\int_\Omega (\tilde\sigma-\sigma)\nabla\tilde u\cdot\nabla w\,\mathrm{d}x\,\mathrm{d}t
+\mathcal B=0,
\end{equation}
where
\[
\mathcal B
:=
\int_0^\infty\int_{\partial\Omega}
\Big(\tilde u\,(\nu\cdot\sigma\nabla w)-(\nu\cdot\tilde\sigma\nabla\tilde u)\,w\Big)\,\mathrm{d}S\,\mathrm{d}t.
\]

For the regularized solutions, we integrate by parts twice in time in the spatial \(L^2\) pairing in the first term of \eqref{eq:first_id_step3}, obtaining
\[
\begin{aligned}
\int_0^\infty\int_\Omega
\frac{\tilde u_{tt}\,w}{\tilde c^2}\,\mathrm dx\,\mathrm dt
=\left[\int_\Omega
\frac{\tilde u_t\,w-\tilde u\,w_t}{\tilde c^2}\,\mathrm dx
\right]_{t=0}^{t=\infty}
\quad+\int_0^\infty\int_\Omega
\frac{\tilde u\,w_{tt}}{\tilde c^2}\,\mathrm dx\,\mathrm dt.
\end{aligned}
\]
The energy bounds for the regularized solutions give at most polynomial growth of \(\tilde u\) and \(\tilde u_t\) in \(L^2(\Omega)\).
Together with the exponential decay of \(w\) and \(w_t\), they imply that the terminal pairing vanishes.
The identity then follows for the original initial data by the convergence established above.
Using also $\tilde u(\cdot,0)=\tilde f$ and $\tilde u_t(\cdot,0)=\tilde h$, we obtain
\begin{equation}\label{eq:first_id_timeparts}
\int_0^\infty\int_\Omega \frac{1}{\tilde c^2}\tilde u_{tt}\,w\,\mathrm{d}x\,\mathrm{d}t
=
\int_0^\infty\int_\Omega \frac{1}{\tilde c^2}\tilde u\,w_{tt}\,\mathrm{d}x\,\mathrm{d}t
+\int_\Omega \frac{\tilde f}{\tilde c^2}\,w_t(x,0)\,\mathrm{d}x
-\int_\Omega \frac{\tilde h}{\tilde c^2}\,w(x,0)\,\mathrm{d}x.
\end{equation}
Substituting \eqref{eq:first_id_timeparts} into \eqref{eq:first_id_step3}, we get
\begin{multline}\label{eq:first_id_step4}
\int_0^\infty\int_\Omega \left(\frac{1}{\tilde c^2}-\frac{1}{c^2}\right)\tilde u\,w_{tt}\,\mathrm{d}x\,\mathrm{d}t
+\int_\Omega \frac{\tilde f}{\tilde c^2}\,w_t(x,0)\,\mathrm{d}x
-\int_\Omega \frac{\tilde h}{\tilde c^2}\,w(x,0)\,\mathrm{d}x\\
+\int_0^\infty\int_\Omega (\tilde\sigma-\sigma)\nabla\tilde u\cdot\nabla w\,\mathrm{d}x\,\mathrm{d}t
+\mathcal B=0.
\end{multline}

It remains to compute the boundary term $\mathcal B$.
Using \(w=e^{-\mathrm i kt}v\) and the equality of the transformed Cauchy data in \eqref{eq:general_transformed_cauchy}, we obtain
\[
\begin{aligned}
\mathcal B
=\int_{\partial\Omega}
\Big(\widehat{\tilde u}\,\sigma\partial_\nu v
-\tilde\sigma\partial_\nu\widehat{\tilde u}\,v\Big)\,\mathrm dS
=\int_{\partial\Omega}
\Big(\hat u\,\sigma\partial_\nu v
-\sigma\partial_\nu\hat u\,v\Big)\,\mathrm dS,
\end{aligned}
\]
where the boundary integrals denote the corresponding trace dualities.
This is the boundary pairing for \((u,w)\) in the weak interpretation specified above.

Since \(u\) and \(w\) satisfy the same wave equation, repeating the above integration-by-parts argument for the pair \((u,w)\) gives
\[
\mathcal B
=
-\int_\Omega \frac{f}{c^2}\,w_t(x,0)\,\mathrm{d}x
+\int_\Omega \frac{h}{c^2}\,w(x,0)\,\mathrm{d}x.
\]
Here we have used the initial conditions
\[
u(\cdot,0)=f,
\qquad
\partial_t u(\cdot,0)=h,
\]
together with the exponential decay of \(w\) and \(w_t\) and the same regularization argument.

Substituting this expression for $\mathcal B$ into \eqref{eq:first_id_step4}, we conclude that
\begin{multline*}
\int_0^\infty\int_\Omega \left(\frac{1}{\tilde c^2}-\frac{1}{c^2}\right)\tilde u\,w_{tt}\,\mathrm{d}x\,\mathrm{d}t
+\int_\Omega \left(\frac{\tilde f}{\tilde c^2}-\frac{f}{c^2}\right)w_t(x,0)\,\mathrm{d}x\\
-\int_\Omega \left(\frac{\tilde h}{\tilde c^2}-\frac{h}{c^2}\right)w(x,0)\,\mathrm{d}x
+\int_0^\infty\int_\Omega (\tilde\sigma-\sigma)\nabla\tilde u\cdot\nabla w\,\mathrm{d}x\,\mathrm{d}t
=0,
\end{multline*}
which is exactly \eqref{eq:first_identity}.

The proof is complete.
\end{proof}

\subsection{Step 2: High-frequency Fourier analysis}

Our aim in this step is to obtain the high-frequency limit needed to identify \(\sigma\) in Step~4.
A local rescaling and an exchange of the two transverse variables cancel the initial-velocity term and the constant part of the initial-displacement term.
The resulting transmission equation retains the conductivity in its principal part, so the first test function can isolate the corresponding gradient term in Step~4.

{We use the transform from Subsection~\ref{subsec:temporal_FT} with \(k\in-\mathrm{i}\mathbb R_+\), and study the limit \(\mathrm{i}k\to+\infty\).}
The interior \(H^3(\Omega)\) regularity in Theorem~\ref{thm:general} and the separated form in Admissibility condition~\ref{ass:general_structure} give \(f_0,\tilde f_0\in W^{2,\infty}(a_2,b_2)\) and \(h_0,\tilde h_0\in W^{2,\infty}(a_1,b_1)\) by one-dimensional Sobolev embedding.
Endpoint values such as \(f_0'(b_2)\) are taken from inside the corresponding interval.
We recall \eqref{FourierTransform} in \(\mathbb R^3\):
\begin{equation}\label{eq:FourierTransform_recalled}
-\nabla\cdot(\sigma\nabla\hat u(x,k))-\frac{k^2}{c^2}\hat u(x,k)
=\frac{1}{c^2}\bigl(\mathrm{i}kf+h\bigr).
\end{equation}
{
Since \(\mathrm{i}k>0\), the weak maximum principle \cite[Theorem~8.1]{gilbarg2001elliptic} applied to \eqref{eq:FourierTransform_recalled} gives
\begin{equation}\label{eq:laplace_bound_general}
\|\hat u(\cdot,k)\|_{L^\infty(\mathbb R^3)}
\le \frac{\|f\|_{L^\infty(\mathbb R^3)}}{\mathrm{i}k}
+\frac{\|h\|_{L^\infty(\mathbb R^3)}}{(\mathrm{i}k)^2}.
\end{equation}
The constant on the right and its negative are respectively a supersolution and a subsolution, and the usual truncation argument justifies comparison for the weak solution.
The same estimate holds for the second configuration.
}
For the frequency-domain identity in Lemma~\ref{lem:second_identity_sigma}, define \(\mathbf R\) by the exact decomposition
\begin{equation}\label{eq R}
\hat u(x,k)=-\mathrm{i}\frac{f(x)}k-\frac{h(x)}{k^2}+\mathbf R(x,k),
\end{equation}
with \(\tilde{\mathbf R}\) defined analogously for the second configuration.

{
By Admissibility~\ref{ass:general_structure}, \(c,\sigma,h\) depend only on \(x_1\) {in \(\Omega\)}.
We seek a transverse exchange that preserves the local coefficients and transmission geometry while cancelling the initial-velocity term.
For arbitrary \(r\in(a_1,b_1)\), we centre the scaling at \((r,b_2,b_3)\), so that the two transverse inward distances give a quadrant invariant under this exchange in the high-frequency limit.
Varying \(r\) will then give identities for the full conductivity profile.

\begin{definition}\label{def:general_scaled_coordinates}
For \(r\in(a_1,b_1)\) and \(k\in-\mathrm{i}\mathbb R_+\), the scaled coordinates \(y=(y_1,y_2,y_3)\) centred at \((r,b_2,b_3)\) are defined by
\[
x_1=r+\frac{y_1}{\mathrm{i}k},\qquad
x_2=b_2-\frac{y_2}{\mathrm{i}k},\qquad
x_3=b_3-\frac{y_3}{\mathrm{i}k}.
\]
We write \(y'=(y_2,y_3)\) and \(\mathcal Q=(0,\infty)^2\).
\end{definition}

\begin{remark}\label{rem:general_scaled_coordinates}
Since \(\mathrm{i}k>0\), these are real coordinates that magnify distances from \((r,b_2,b_3)\) by the factor \(\mathrm{i}k\).
The coordinate \(y_1\) is tangential to both faces, while \(y_2,y_3\) are the scaled inward distances from \(x_2=b_2\) and \(x_3=b_3\), respectively.
Thus a fixed bounded set of \(y\) describes an \(O(|k|^{-1})\) neighbourhood in the original variables.
For finite \(k\), the image of \(\Omega\) is a bounded {domain}; only as \(\mathrm{i}k\to+\infty\) does it expand locally to \(\mathbb R\times\mathcal Q\).
Equivalently, on every fixed bounded set of \(y\), membership in \(\Omega\) is given by \(y'\in\mathcal Q\) for all sufficiently large \(\mathrm{i}k\).

The two observation faces are represented by the portions of \(y_2=0,\ y_3>0\) and \(y_3=0,\ y_2>0\) on the boundary of this {domain}, whose finite bounds recede to infinity in the limit.
These are the interfaces between the interior medium and the known exterior; the full-space weak equation matches the field and its physical conormal flux across them.
In the transverse limit, their closures form \(\partial\mathcal Q\), the two nonnegative coordinate half-axes.
Lemma~\ref{lem:odd_high_frequency_limit} will establish the independence of the limiting field from \(y_1\).
\end{remark}
}

{On each fixed bounded set and for sufficiently large \(\mathrm{i}k\), rewriting \eqref{eq:FourierTransform_recalled} in the coordinates of Definition~\ref{def:general_scaled_coordinates} and dividing by \((\mathrm{i}k)^2\) gives the source term}
\[
\frac{\chi_{\mathcal Q}(y')}{c_0(r+y_1/(\mathrm{i}k))^2}
\left[
\frac{f_0(b_2-y_2/(\mathrm{i}k))}{\mathrm{i}k}
+\frac{h_0(r+y_1/(\mathrm{i}k))}{(\mathrm{i}k)^2}
\right].
\]
Since \(c_0,\sigma_0,h_0\) depend only on \(x_1\) and \(\mathcal Q\) is invariant under exchanging \(y_2\) and \(y_3\), the two equations obtained by this exchange have the same coefficients and the same \(h_0\) term.
Subtracting them therefore cancels \(h_0\) and replaces \(f_0(b_2-y_2/(\mathrm{i}k))\) in the displayed source by
\[
f_0\left(b_2-\frac{y_2}{\mathrm{i}k}\right)
-f_0\left(b_2-\frac{y_3}{\mathrm{i}k}\right)
=\frac{f_0'(b_2)(y_3-y_2)}{\mathrm{i}k}
+O\bigl((\mathrm{i}k)^{-2}\bigr),
\]
where the expansion holds for bounded \(y_2,y_3>0\).
Together with the factor \((\mathrm{i}k)^{-1}\) in the displayed source, this gives a source of order \((\mathrm{i}k)^{-2}\) for the difference of the two transformed fields.
We therefore multiply this difference by \((\mathrm{i}k)^2\) and define the antisymmetric field
\begin{equation}\label{eq:odd_field_general}
U_{k,r}(y)
=(\mathrm{i}k)^2\Biggl[
\hat u\left(r+\frac{y_1}{\mathrm{i}k},b_2-\frac{y_2}{\mathrm{i}k},b_3-\frac{y_3}{\mathrm{i}k},k\right)
-\hat u\left(r+\frac{y_1}{\mathrm{i}k},b_2-\frac{y_3}{\mathrm{i}k},b_3-\frac{y_2}{\mathrm{i}k},k\right)\Biggr].
\end{equation}
The source in the equation for \(U_{k,r}\) therefore has a finite limit as \(\mathrm{i}k\to+\infty\).
{
By \eqref{eq:FourierTransform_recalled} and \eqref{eq:odd_field_general}, the normalized difference satisfies
\begin{equation}\label{eq:scaled_odd_equation}
-\nabla_y\cdot(\sigma(x)\nabla_y U_{k,r})+\frac{U_{k,r}}{c(x)^2}
=\frac{\mathrm{i}k}{c(x)^2}
\left[f_0\left(b_2-\frac{y_2}{\mathrm{i}k}\right)
-f_0\left(b_2-\frac{y_3}{\mathrm{i}k}\right)\right]\chi_{\mathcal Q}
\end{equation}
on every fixed bounded set of \(y\) for all sufficiently large \(\mathrm{i}k\).
Here \(c(x)\) and \(\sigma(x)\) are evaluated at the original point corresponding to \(y\) in Definition~\ref{def:general_scaled_coordinates}.
}
Freezing the principal coefficient at \(x_1=r\) gives
\[
\sigma_r(y')=
\begin{cases}
\sigma_0(r),&y'\in\mathcal Q,\\
1,&y'\notin\mathcal Q.
\end{cases}
\]
The coefficient \(c^{-2}\) tends to \(c_0(r)^{-2}\) inside \(\mathcal Q\) and remains equal to one outside.

The following lemma identifies the limiting field and shows that it is independent of \(y_1\).

\Needspace{16\baselineskip}
\begin{lemma}\label{lem:odd_high_frequency_limit}
Assume the regularity and uniform bounds in Theorem~\ref{thm:general}, together with the structural condition in Admissibility~\ref{ass:general_structure}.
Fix \(r\in(a_1,b_1)\), and let \(U_{k,r}\) be defined by \eqref{eq:odd_field_general}.
As \(\mathrm{i}k\to+\infty\), \(U_{k,r}\) converges locally uniformly in \(\mathbb R^3\) and weakly in \(H^1_{\mathrm{loc}}(\mathbb R^3)\) to a function \(U_r(y')\) independent of \(y_1\).
The limit is the unique weak solution in \(\mathbb R^2\) of
\begin{equation}\label{eq:frozen_odd_equation}
-\nabla\cdot(\sigma_r\nabla U_r)
+\bigl[1+(c_0(r)^{-2}-1)\chi_{\mathcal Q}\bigr]U_r
=c_0(r)^{-2}f_0'(b_2)(y_3-y_2)\chi_{\mathcal Q}
\end{equation}
satisfying \(|U_r(y')|\le C(1+|y'|)\) for some constant \(C>0\).
Moreover, for every \(y_3>0\),
\begin{equation}\label{eq:odd_boundary_limit_general}
U_r(0,y_3)
=\lim_{\mathrm{i}k\to+\infty}(\mathrm{i}k)^2\bigl[
\hat u(r,b_2,b_3-y_3/(\mathrm{i}k),k)
-\hat u(r,b_2-y_3/(\mathrm{i}k),b_3,k)\bigr].
\end{equation}
The same conclusions hold for the second configuration.
\end{lemma}

{
Equation \eqref{eq:frozen_odd_equation} is therefore the high-frequency limit of \eqref{eq:scaled_odd_equation}, obtained from the original transformed equation \eqref{eq:FourierTransform_recalled}.
The local bounds, convergence, and uniqueness needed for this limit are proved in Appendix~\ref{app:odd_high_frequency_limit}.
}

The limiting problem and the boundary traces identified above will be used in Step 4 to separate the conductivity from the remaining unknown quantities.

\subsection{Step 3: Frequency-domain identity}

We choose a test wave with coefficient \(1\) in front of its second time derivative.
Using the exact decomposition \eqref{eq R}, the resulting identity places the explicit factors involving \(c,\tilde c\) in the remainder integrals and isolates \(\sigma-\tilde\sigma\) in the gradient pairing.

\begin{lemma}\label{lem:second_identity_sigma}
Let \((f,h,c,\sigma)\) and \((\tilde f,\tilde h,\tilde c,\tilde\sigma)\) satisfy the assumptions on the domain, initial data and coefficients in Lemma~\ref{lem:first_identity}, and let \(u,\tilde u\) be the corresponding weak solutions.
Assume that the observed displacement traces coincide on \(\partial\Omega\times\mathbb R_+\) in the sense of \eqref{eq:2}.
Fix \(k\in-\mathrm{i}\mathbb R_+\), and let \(v\in H^1(\Omega)\) be a weak solution of
\[
-\nabla\cdot(\sigma(x)\nabla v(x))-k^2v(x)=0
\quad\text{in }\Omega.
\]
Then \(w(x,t)=\exp(-\mathrm{i}kt)v(x)\) satisfies
\begin{equation}\label{eq:free_wave_sigma2}
\partial_t^2 w(x,t)-\nabla\cdot(\sigma(x)\nabla w(x,t))=0
\quad\text{in }\Omega\times\mathbb R_+.
\end{equation}
Let \(\mathbf R,\tilde{\mathbf R}\) be defined by the exact decomposition \eqref{eq R} for the two configurations.
Then
\begin{equation}\label{eq:second_identity_sigma}
\begin{aligned}
&\int_\Omega (\sigma-\tilde\sigma)\,\nabla\widehat{\tilde u}(x,k)\cdot\nabla v(x)\,\mathrm{d}x
-\mathrm{i}k\int_\Omega (f-\tilde f)\,v(x)\,\mathrm{d}x
-\int_\Omega (h-\tilde h)\,v(x)\,\mathrm{d}x\\
&\qquad
+k^2\int_\Omega
\left[
\left(1-\frac{1}{c(x)^2}\right)\mathbf R(x,k)
-
\left(1-\frac{1}{\tilde c(x)^2}\right)\tilde{\mathbf R}(x,k)
\right]v(x)\,\mathrm{d}x
=0.
\end{aligned}
\end{equation}
\end{lemma}

\begin{proof}
As in the proof of Lemma~\ref{lem:first_identity}, we regularize the initial data and retain all boundary terms until the transformed fields and their traces converge.
We then cancel these terms using the transformed trace equalities \eqref{eq:general_transformed_cauchy}, supplied by Lemma~\ref{lem:boundary_flux_from_dirichlet}.
All spacetime integrals below are understood through this limit.
We rewrite the two wave equations with coefficient \(1\) in front of the second time derivative and subtract them to obtain
\[
(u-\tilde u)_{tt}-\nabla\cdot(\sigma\nabla u-\tilde\sigma\nabla\tilde u)
=
\left(1-\frac{1}{c^2}\right)u_{tt}
-
\left(1-\frac{1}{\tilde c^2}\right)\tilde u_{tt}.
\]
Multiplying by \(w\) and integrating over \(\Omega\times(0,\infty)\), an integration by parts in \(x\) gives
\begin{equation}\label{eq:step_space_parts_sigma}
\begin{aligned}
&\int_0^\infty\!\!\int_\Omega w (u-\tilde u)_{tt}\,\mathrm{d}x\,\mathrm{d}t
+\int_0^\infty\!\!\int_\Omega \sigma\nabla u\cdot\nabla w\,\mathrm{d}x\,\mathrm{d}t
-\int_0^\infty\!\!\int_\Omega \tilde\sigma\nabla\tilde u\cdot\nabla w\,\mathrm{d}x\,\mathrm{d}t\\
&\qquad
=
\int_0^\infty\!\!\int_\Omega
\left(1-\frac{1}{c^2}\right)w u_{tt}
-
\left(1-\frac{1}{\tilde c^2}\right)w\tilde u_{tt}\,\mathrm{d}x\,\mathrm{d}t,
\end{aligned}
\end{equation}
where the conormal boundary pairing vanishes because \(\sigma\partial_\nu\hat u=\tilde\sigma\partial_\nu\widehat{\tilde u}\) by \eqref{eq:general_transformed_cauchy}.
Using
\[
\sigma\nabla u-\tilde\sigma\nabla\tilde u
=
\sigma\nabla(u-\tilde u)+(\sigma-\tilde\sigma)\nabla\tilde u,
\]
we can rewrite the second and third terms on the left-hand side of \eqref{eq:step_space_parts_sigma} as
\[
\int_0^\infty\!\!\int_\Omega \sigma\nabla(u-\tilde u)\cdot\nabla w\,\mathrm{d}x\,\mathrm{d}t
+
\int_0^\infty\!\!\int_\Omega (\sigma-\tilde\sigma)\nabla\tilde u\cdot\nabla w\,\mathrm{d}x\,\mathrm{d}t.
\]
We integrate the terms involving \(u-\tilde u\) twice in \(t\) and once in \(x\).
The terminal terms vanish by exponential decay.
The spatial boundary pairing contains the trace \(\hat u-\widehat{\tilde u}\), which is zero by \eqref{eq:general_transformed_cauchy}.
Using \eqref{eq:free_wave_sigma2}, we obtain
\begin{multline*}
\int_0^\infty\!\!\int_\Omega w (u-\tilde u)_{tt}\,\mathrm{d}x\,\mathrm{d}t
+
\int_0^\infty\!\!\int_\Omega \sigma\nabla(u-\tilde u)\cdot\nabla w\,\mathrm{d}x\,\mathrm{d}t\\
=
\int_\Omega (f-\tilde f)\,w_t(\cdot,0)\,\mathrm{d}x
-
\int_\Omega (h-\tilde h)\,w(\cdot,0)\,\mathrm{d}x.
\end{multline*}
Substituting this into \eqref{eq:step_space_parts_sigma} gives
\begin{equation}\label{eq:pre_fourier_sigma}
\begin{aligned}
&\int_0^\infty\!\!\int_\Omega (\sigma-\tilde\sigma)\nabla\tilde u\cdot\nabla w\,\mathrm{d}x\,\mathrm{d}t
+\int_\Omega (f-\tilde f)\,w_t(\cdot,0)\,\mathrm{d}x
-\int_\Omega (h-\tilde h)\,w(\cdot,0)\,\mathrm{d}x\\
&\qquad
=
\int_0^\infty\!\!\int_\Omega
\left(1-\frac{1}{c^2}\right)w u_{tt}
-
\left(1-\frac{1}{\tilde c^2}\right)w\tilde u_{tt}\,\mathrm{d}x\,\mathrm{d}t.
\end{aligned}
\end{equation}
We have
\[
w_t(\cdot,0)=-\mathrm{i}k\,v,
\qquad
w(\cdot,0)=v.
\]
The gradient term is the weak pairing
\[
\int_0^\infty\!\!\int_\Omega(\sigma-\tilde\sigma)\nabla\tilde u\cdot\nabla w\,\mathrm{d}x\,\mathrm{d}t
:=\int_\Omega(\sigma-\tilde\sigma)\nabla\widehat{\tilde u}(\cdot,k)\cdot\nabla v\,\mathrm{d}x.
\]
The same regularization gives the \(L^2(\Omega)\) identities
\[
\int_0^\infty u_{tt}(\cdot,t)\exp(-\mathrm{i}kt)\,\mathrm{d}t
=
-k^2\hat u(\cdot,k)-\mathrm{i}k f-h,
\]
\[
\int_0^\infty \tilde u_{tt}(\cdot,t)\exp(-\mathrm{i}kt)\,\mathrm{d}t
=
-k^2\widehat{\tilde u}(\cdot,k)-\mathrm{i}k\tilde f-\tilde h.
\]
Substituting these identities into \eqref{eq:pre_fourier_sigma}, we obtain
\begin{equation*}
\begin{aligned}
&\int_\Omega (\sigma-\tilde\sigma)\nabla\widehat{\tilde u}(x,k)\cdot\nabla v(x)\,\mathrm{d}x\\
&\qquad{}-\mathrm{i}k\int_\Omega (f-\tilde f)\,v(x)\,\mathrm{d}x
-\int_\Omega (h-\tilde h)\,v(x)\,\mathrm{d}x\\
&\qquad
=
\int_\Omega \left(1-\frac{1}{c^2}\right)\Bigl(-k^2\hat u-\mathrm{i}k f-h\Bigr)v\,\mathrm{d}x\\
&\qquad\quad{}-\int_\Omega \left(1-\frac{1}{\tilde c^2}\right)\Bigl(-k^2\widehat{\tilde u}-\mathrm{i}k\tilde f-\tilde h\Bigr)v\,\mathrm{d}x.
\end{aligned}
\end{equation*}
Finally, the exact decomposition \eqref{eq R} gives
\[
-k^2\hat u-\mathrm{i}k f-h=-k^2\mathbf R,
\qquad
-k^2\widehat{\tilde u}-\mathrm{i}k\tilde f-\tilde h=-k^2\tilde{\mathbf R}.
\]
This gives \eqref{eq:second_identity_sigma} for each fixed \(k\), with all remainder terms retained.

The proof is complete.
\end{proof}


{
Within the unified framework of Subsection~\ref{subsec:framework}, the first test family uses the Green identity argument underlying the second integral identity \eqref{eq:second_identity_sigma}, applied to the high-frequency limiting equations in Lemma~\ref{lem:odd_high_frequency_limit}, to identify \(\sigma\).
The second and third families use exact test waves in the first integral identity \eqref{eq:first_identity} at high and low temporal frequencies, respectively, to extract weighted source information and then determine \(h/c^2\) by zero-frequency CGO analysis.
}

\begingroup
\subsection{Step 4: Recovery of \texorpdfstring{$\sigma$}{sigma}}\label{subsec:general_sigma_recovery}

\leavevmode{
For each \(r\in(a_1,b_1)\), we apply the Green identity argument used in Lemma~\ref{lem:second_identity_sigma} to the equations for the limiting fields \(U_r,\tilde U_r\) in Lemma~\ref{lem:odd_high_frequency_limit} to recover \(\sigma_0(r)\).
We choose a test that satisfies the limiting reference equation \eqref{eq:first_odd_reference_general}, decays exponentially in \(\mathcal Q\), and matches the antisymmetry of the construction \eqref{eq:odd_field_general}.
This leads to the difference of two exchanged exponential factors in the following definition.
}

\begingroup
\begin{definition}\label{def:general_test_families}
With \(y'\) and \(\mathcal Q\) as in Definition~\ref{def:general_scaled_coordinates}, fix \(r\in(a_1,b_1)\) and, for sufficiently large \(s>0\), define
\begin{equation}\label{eq:first_odd_test_general}
v_s^{(1)}(y')=\exp(-p_sy_2-q_sy_3)-\exp(-q_sy_2-p_sy_3),
\end{equation}
where \(p_s=s(1+\mathrm{i})\) and \(q_s=\sqrt{\sigma_0(r)^{-1}-2\mathrm{i}s^2}\) is chosen with \(\operatorname{Re}q_s>0\) and \(\operatorname{Im}q_s<0\).
This function is antisymmetric under the exchange of \(y_2\) and \(y_3\) and satisfies
\begin{equation}\label{eq:first_odd_reference_general}
(-\sigma_0(r)\Delta_{y'}+1)v_s^{(1)}=0\qquad\text{in }\mathcal{Q}.
\end{equation}
\end{definition}
\endgroup

{
The reference equation \eqref{eq:first_odd_reference_general} is obtained from the spatial test equation in Lemma~\ref{lem:second_identity_sigma} by using the coordinates of Definition~\ref{def:general_scaled_coordinates}, dividing by \((\mathrm{i}k)^2\), and taking the high-frequency limit for functions independent of \(y_1\).
The temporal limit has already been taken in the fields \(U_r,\tilde U_r\); the test \(v_s^{(1)}\) now varies only with the spatial parameter \(s\).
}

Proposition~\ref{prop:edge_test_identity} establishes the resulting identity \eqref{eq:odd_CGO_difference_general} for each fixed \(s\).
{With these fields fixed, Lemma~\ref{lem:odd_CGO_pairing_general} uses the endpoint condition in Admissibility~\ref{ass:general_structure} for both configurations to obtain a nonzero leading gradient pairing along a sequence \(s\to\infty\), with lower-order bounds for the other terms in \eqref{eq:odd_CGO_difference_general}.}
Hence \(\sigma_0(r)=\tilde\sigma_0(r)\) for every \(r\), giving \(\sigma=\tilde\sigma\) in \(\Omega\).

\begingroup
\subsubsection{An integral identity for the first test function}

\begin{proposition}\label{prop:edge_test_identity}
Suppose that the two configurations in Theorem~\ref{thm:general} give the same displacement trace on \(\partial\Omega\times\mathbb R_+\).
{For each \(r\in(a_1,b_1)\), let \(U_r,\tilde U_r\) be the limiting fields in Lemma~\ref{lem:odd_high_frequency_limit} and \(v_s^{(1)}\) the first test function of Definition~\ref{def:general_test_families}.}
Then the Green identity argument used to derive \eqref{eq:second_identity_sigma} in Lemma~\ref{lem:second_identity_sigma}, applied to the limiting transmission equations, gives, for all sufficiently large \(s\),
\begin{equation}\label{eq:odd_CGO_difference_general}
\begin{aligned}
&(\sigma_0(r)-\tilde\sigma_0(r))\int_{\mathcal Q}\nabla\tilde U_r\cdot\nabla v_s^{(1)}\\
&=\bigl[c_0(r)^{-2}f_0'(b_2)-\tilde c_0(r)^{-2}\tilde f_0'(b_2)\bigr]\int_{\mathcal Q}(y_3-y_2)v_s^{(1)}\\
&\quad-\int_{\mathcal Q}\bigl[(c_0(r)^{-2}-1)U_r-(\tilde c_0(r)^{-2}-1)\tilde U_r\bigr]v_s^{(1)}.
\end{aligned}
\end{equation}
\end{proposition}

\begin{proof}
We apply the Green identity argument used in Lemma~\ref{lem:second_identity_sigma} directly to the limiting equation \eqref{eq:frozen_odd_equation} and its tilded counterpart.
Lemmas~\ref{lem:boundary_flux_from_dirichlet} and~\ref{lem:odd_high_frequency_limit}, together with \eqref{eq:odd_field_general}, give \(U_r=\tilde U_r\) outside \(\overline{\mathcal Q}\).
Their difference therefore has zero trace on \(\partial\mathcal Q\).

For fixed \(s\), multiply \(v_s^{(1)}\) by a smooth cutoff equal to one on \(|y'|\le R\) and zero on \(|y'|\ge2R\).
Extend this truncated test evenly across the coordinate axes to a compactly supported function in \(H^1(\mathbb R^2)\).
Use this extension to test the two limiting weak equations and subtract, as in \eqref{eq:step_space_parts_sigma}.
The exterior integrals cancel because the coefficients and fields agree there, independently of the extension.

The Caccioppoli energy estimate for weak solutions \cite[Problem~8.2]{gilbarg2001elliptic}, together with the linear growth in Lemma~\ref{lem:odd_high_frequency_limit}, bounds the gradient energy of either field on \(R<|y'|<2R\) by \(C(R^2+R^4)\) for \(R\ge1\).
For fixed \(s\), the exponential decay of \(v_s^{(1)}\) and its derivatives in \eqref{eq:first_odd_test_general} therefore removes all cutoff terms as \(R\to\infty\).
The resulting weak difference identity for the limiting fields takes the form
\[
\begin{aligned}
&(\sigma_0(r)-\tilde\sigma_0(r))\int_{\mathcal Q}\nabla\tilde U_r\cdot\nabla v_s^{(1)}
+\sigma_0(r)\int_{\mathcal Q}\nabla(U_r-\tilde U_r)\cdot\nabla v_s^{(1)}\\
&\quad+\int_{\mathcal Q}\bigl[c_0(r)^{-2}U_r-\tilde c_0(r)^{-2}\tilde U_r\bigr]v_s^{(1)}\\
&=\bigl[c_0(r)^{-2}f_0'(b_2)-\tilde c_0(r)^{-2}\tilde f_0'(b_2)\bigr]
\int_{\mathcal Q}(y_3-y_2)v_s^{(1)}.
\end{aligned}
\]
In \(\mathcal Q\), we integrate the term containing \(\sigma_0(r)\nabla(U_r-\tilde U_r)\) by parts using the zero trace.
Since \(\sigma_0(r)\Delta v_s^{(1)}=v_s^{(1)}\) by \eqref{eq:first_odd_reference_general}, this term contributes \(-\int_{\mathcal Q}(U_r-\tilde U_r)v_s^{(1)}\).
Combining this contribution with the mass terms gives the factors \(c_0(r)^{-2}-1\) and \(\tilde c_0(r)^{-2}-1\), and hence \eqref{eq:odd_CGO_difference_general}.
\end{proof}
\endgroup

\subsubsection{Decoupling as \texorpdfstring{\(s\to\infty\)}{s tends to infinity}}

\leavevmode{The identity \eqref{eq:odd_CGO_difference_general} holds for every sufficiently large choice of the spatial parameter \(s\) in the test function \(v_s^{(1)}\) of Definition~\ref{def:general_test_families}.}
We now keep \(U_r,\tilde U_r\) fixed and let \(s\to\infty\), so that the exponential factors in \(v_s^{(1)}\) concentrate the test near \(y'=0\) on the scale \(s^{-1}\).
Our aim is to show that the source and zeroth-order pairings on the right-hand side are negligible relative to the gradient pairing multiplying \(\sigma_0(r)-\tilde\sigma_0(r)\).
{The following lemma supplies a nonzero gradient pairing and the relative estimates required for this comparison.}

\Needspace{14\baselineskip}
\begin{lemma}\label{lem:odd_CGO_pairing_general}
{Let \(U_r,\tilde U_r\) be the limiting fields in Lemma~\ref{lem:odd_high_frequency_limit} at a fixed \(r\in(a_1,b_1)\), and let \(v_s^{(1)}\) be the test function in Definition~\ref{def:general_test_families}.}
{Assume that \(f_0'(b_2)\ne0\) and \(\tilde f_0'(b_2)\ne0\), as required by Admissibility~\ref{ass:general_structure}.}
Then there exists a sequence \(s_j\to\infty\) such that
\begin{equation}\label{eq:odd_nonzero_gradient_general}
\int_{\mathcal Q}\nabla\tilde U_r\cdot\nabla v_{s_j}^{(1)}\ne0
\end{equation}
and
\begin{equation}\label{eq:odd_lower_order_pairings_general}
\frac{\left|\int_{\mathcal Q}U_rv_{s_j}^{(1)}\right|
+\left|\int_{\mathcal Q}\tilde U_rv_{s_j}^{(1)}\right|
+\left|\int_{\mathcal Q}(y_3-y_2)v_{s_j}^{(1)}\right|}
{\left|\int_{\mathcal Q}\nabla\tilde U_r\cdot\nabla v_{s_j}^{(1)}\right|}
\longrightarrow0.
\end{equation}
\end{lemma}

{
In \eqref{eq:odd_CGO_difference_general}, \eqref{eq:odd_nonzero_gradient_general} permits division by the gradient pairing, and \eqref{eq:odd_lower_order_pairings_general} makes all three pairings on the right negligible along the same sequence.
The proof of these estimates is given in Appendix~\ref{app:odd_CGO_pairing_general}.
}

We now combine Proposition~\ref{prop:edge_test_identity} and Lemma~\ref{lem:odd_CGO_pairing_general} to recover the conductivity.
Fix an arbitrary \(r\in(a_1,b_1)\), and keep the limiting fields \(U_r,\tilde U_r\) from Lemma~\ref{lem:odd_high_frequency_limit} fixed.
The exchange of variables in \eqref{eq:odd_field_general} has eliminated the initial-velocity term while retaining the conductivity in the principal part.
{By Lemma~\ref{lem:boundary_flux_from_dirichlet}, the equality of the passive observations gives the common exterior field used in Proposition~\ref{prop:edge_test_identity}, so \eqref{eq:odd_CGO_difference_general} holds for every sufficiently large \(s\) in the first test family of Definition~\ref{def:general_test_families}.}
The time-frequency limit has already been taken in defining these fields; we now vary only the spatial parameter \(s\).

Choose the sequence \(s_j\to\infty\) supplied by Lemma~\ref{lem:odd_CGO_pairing_general} for this \(r\).
For all sufficiently large \(j\), \eqref{eq:odd_CGO_difference_general} applies with \(s=s_j\), and \eqref{eq:odd_nonzero_gradient_general} allows us to divide it by the nonzero gradient pairing \(\int_{\mathcal Q}\nabla\tilde U_r\cdot\nabla v_{s_j}^{(1)}\).
The source coefficient \(c_0(r)^{-2}f_0'(b_2)-\tilde c_0(r)^{-2}\tilde f_0'(b_2)\) and the two zeroth-order coefficients \(c_0(r)^{-2}-1\), \(\tilde c_0(r)^{-2}-1\) are finite and independent of \(j\).
Taking absolute values and applying the triangle inequality therefore gives
\[
|\sigma_0(r)-\tilde\sigma_0(r)|
\le C\,
\frac{\left|\int_{\mathcal Q}U_rv_{s_j}^{(1)}\right|
+\left|\int_{\mathcal Q}\tilde U_rv_{s_j}^{(1)}\right|
+\left|\int_{\mathcal Q}(y_3-y_2)v_{s_j}^{(1)}\right|}
{\left|\int_{\mathcal Q}\nabla\tilde U_r\cdot\nabla v_{s_j}^{(1)}\right|},
\]
where \(C\) is independent of \(j\).
All three terms in the numerator are controlled relative to this same denominator by \eqref{eq:odd_lower_order_pairings_general}.
Thus the right-hand side tends to zero as \(j\to\infty\), whereas the left-hand side is independent of \(j\), and hence
\[
\sigma_0(r)=\tilde\sigma_0(r).
\]

Since \(r\in(a_1,b_1)\) was arbitrary and \(\sigma(x)=\sigma_0(x_1)\), \(\tilde\sigma(x)=\tilde\sigma_0(x_1)\) in \(\Omega\) by Admissibility~\ref{ass:general_structure}, we obtain
\begin{equation}\label{eq uniq sigma}
\sigma=\tilde\sigma\qquad\text{a.e. in }\Omega.
\end{equation}
The first test family thus separates the conductivity difference from every source and zeroth-order contribution in \eqref{eq:odd_CGO_difference_general}.
In this last passage to the limit, the unknown coefficients are used only as fixed finite factors, so their values need not have been recovered beforehand.
The following steps use the common conductivity in \eqref{eq uniq sigma} to separate these remaining quantities.
\endgroup

\begingroup
\Needspace{12\baselineskip}
\begingroup
\subsection{Step 5: High-frequency information on \texorpdfstring{\(f\) and \(h\)}{f and h}}\label{subsec:general_source_tests}

\leavevmode{
Having established the common conductivity in \eqref{eq uniq sigma}, we construct the exact test waves used in the first integral identity at both high and low temporal frequencies.
Let \(v_0\in H^1(\Omega)\) satisfy \(\nabla\cdot(\sigma\nabla v_0)=0\) in \(\Omega\).
{
For \(k\in-\mathrm{i}\mathbb R_+\), define \(v_k\in H^1(\Omega)\) and \(w_k\) by
\begin{equation}\label{eq:general_exact_tests}
\begin{gathered}
\begin{cases}
-\nabla\cdot(\sigma\nabla v_k)-k^2c^{-2}v_k=0,&\text{in }\Omega,\\
v_k=v_0,&\text{on }\partial\Omega,
\end{cases}\\
w_k(x,t)=\exp(-\mathrm{i}kt)v_k(x).
\end{gathered}
\end{equation}
}
Since \(-k^2>0\), this Dirichlet problem is coercive and has a unique weak solution.
The wave \(w_k\) satisfies the test equation in Lemma~\ref{lem:first_identity} and is therefore admissible in \eqref{eq:first_identity}.
}
\endgroup

{
The conductivity pairing vanishes by \eqref{eq uniq sigma}; substituting \(\partial_tw_k(x,0)=-\mathrm{i}kv_k(x)\) and \(\partial_{tt}w_k=-k^2w_k\) gives, for every \(k\in-\mathrm{i}\mathbb R_+\),
\begin{equation}\label{eq:first_low_freq_general_c}
\begin{aligned}
\int_\Omega\left[\frac{h}{c^2}-\frac{\tilde h}{\tilde c^2}
+\mathrm{i}k\left(\frac{f}{c^2}-\frac{\tilde f}{\tilde c^2}\right)\right]v_k\,\mathrm{d}x
=-k^2\int_\Omega\left(\frac1{c^2}-\frac1{\tilde c^2}\right)
\widehat{\tilde u}(x,k)v_k(x)\,\mathrm{d}x.
\end{aligned}
\end{equation}
}
\medskip
\subsubsection{The second test function}\label{subsubsec:general_second_test}

\leavevmode{
By Admissibility~\ref{ass:general_structure}, the interior profiles are independent of \(x_3\), so we localize in this direction while keeping \(x_1,x_2\) arbitrary.
The inward normal coordinate is \(y=\mathrm{i}k(b_3-x_3)\), as in Definition~\ref{def:general_scaled_coordinates}.
}

\begingroup
\begin{definition}\label{def:general_second_test}
Fix a \(\sigma\)-harmonic function \(v_0\in H^1(\Omega)\) with \(v_0(\cdot,\cdot,b_3)\in C_c^\infty((a_1,b_1)\times(a_2,b_2))\) and zero trace on the remaining boundary.
{Let \(v_k,w_k\) be defined by \eqref{eq:general_exact_tests}.}
These form the second test family, used as \(|k|\to\infty\) with \(k\in-\mathrm{i}\mathbb R_+\).
For fixed \(r\in(a_1,b_1)\), define the decaying normal profile
\[
v^{(2)}(y)=\exp\!\left(-\frac{y}{c_0(r)\sqrt{\sigma_0(r)}}\right),\qquad y>0.
\]
It satisfies
\[
-\sigma_0(r)(v^{(2)})''+c_0(r)^{-2}v^{(2)}=0,\qquad y>0,
\]
with \(v^{(2)}(0)=1\) and \(v^{(2)}(y)\to0\) as \(y\to\infty\).
\end{definition}
\endgroup

{
Every smooth boundary trace with the stated support has a \(\sigma\)-harmonic extension, so these boundary values can be prescribed freely.
{
To extract both source terms from \eqref{eq:first_low_freq_general_c}, we need normal approximations for the transformed field in \eqref{eq:FourierTransform_recalled} and the exact test in \eqref{eq:general_exact_tests}.
The following lemma gives the two error estimates; we compute the normal weight below and justify the errors in Appendix~\ref{app:one_face_expansion}.
}
}

\Needspace{18\baselineskip}
\begin{lemma}\label{lem:one_face_expansion}
\leavevmode{
Under the assumptions of Theorem~\ref{thm:general}, evaluate \(v^{(2)}\) with \(r=x_1\).
As \(|k|\to\infty\), uniformly for \((x_1,x_2)\) in compact subsets of \((a_1,b_1)\times(a_2,b_2)\) and on a fixed inward neighbourhood of \(x_3=b_3\), including the boundary, we have
\begin{equation}\label{eq:one_face_expansion}
(\mathrm{i}k)\hat u(x,k)
=\left(f_0(x_2)+\frac{h_0(x_1)}{\mathrm{i}k}\right)
\left[1-\frac{c_0(x_1)v^{(2)}(y)}{\sqrt{\sigma_0(x_1)}+c_0(x_1)}\right]
+O(|k|^{-2}),
\end{equation}
where \(y=\mathrm{i}k(b_3-x_3)\).

For each fixed datum \(v_0\) chosen in Definition~\ref{def:general_second_test}, the corresponding exact test \(v_k\) also satisfies, for all sufficiently large \(|k|\),
\begin{equation}\label{eq:normal_test_error_general}
\left\|v_k-v_0(x_1,x_2,b_3)v^{(2)}\!\left(\mathrm{i}k(b_3-x_3)\right)\right\|_{L^1(\Omega)}
\le C|k|^{-3}.
\end{equation}
The constants are independent of \(k\), with dependence only on \(\Omega\), the fixed compact subset and boundary trace, the bounds in \eqref{eq:general_bounds}, the \(H^3(\Omega)\) norms of \(f,h\), and the \(W^{2,\infty}(\Omega)\) norms of \(c,\sigma\).
}
\end{lemma}

\begingroup
\begin{proof}
We first compute the normal weight in the field expansion.
Fix \(x_1=r\in(a_1,b_1)\) and \(x_2\in(a_2,b_2)\).
{Multiplying \eqref{eq:FourierTransform_recalled} by \(\mathrm{i}k\), using \(y=\mathrm{i}k(b_3-x_3)\), and dividing by \((\mathrm{i}k)^2\) gives the normal equation for the profile \(v(y)\) of \((\mathrm{i}k)\hat u\) after freezing the tangential variables and retaining only normal derivatives:}
\[
-\sigma_0(r)v''+c_0(r)^{-2}v=c_0(r)^{-2}\left(f_0(x_2)+\frac{h_0(r)}{\mathrm{i}k}\right),\qquad y>0.
\]
We take the solution of at most polynomial growth for \(y>0\) and match it to \(v(0)\exp(y)\) for \(y<0\).
The value is continuous at zero, and equality of the physical conormal fluxes gives
\[
\sigma_0(r)v'(0+)=v(0).
\]
{The decaying normal profile \(v^{(2)}\) defined in Definition~\ref{def:general_second_test} solves the corresponding homogeneous equation.}
Multiplying the normal equation by this test and integrating by parts yields
\begin{equation}\label{eq:face_normal_moment}
\begin{gathered}
c_0(r)^{-2}\int_0^\infty\left(f_0(x_2)+\frac{h_0(r)}{\mathrm{i}k}\right)v^{(2)}(y)\,\mathrm{d}y
=\left(1+\frac{\sqrt{\sigma_0(r)}}{c_0(r)}\right)v(0),\\
v(0)=\frac{\sqrt{\sigma_0(r)}}{\sqrt{\sigma_0(r)}+c_0(r)}\left(f_0(x_2)+\frac{h_0(r)}{\mathrm{i}k}\right).
\end{gathered}
\end{equation}
The boundary term at infinity vanishes by exponential decay, and \(\int_0^\infty v^{(2)}(y)\,\mathrm{d}y=c_0(r)\sqrt{\sigma_0(r)}\) gives the second equality.

{
The corresponding normal profile is
\[
v(y)=v(0)v^{(2)}(y)+\left(f_0(x_2)+\frac{h_0(r)}{\mathrm{i}k}\right)\bigl(1-v^{(2)}(y)\bigr).
\]
Substituting \eqref{eq:face_normal_moment} gives the two explicit terms in \eqref{eq:one_face_expansion}.
Appendix~\ref{app:one_face_expansion} establishes the error bounds in \eqref{eq:one_face_expansion} and \eqref{eq:normal_test_error_general}.
}
\end{proof}
\endgroup

{
{
We now use Lemma~\ref{lem:one_face_expansion} in the first integral identity \eqref{eq:first_low_freq_general_c} to extract the weighted source terms.
Fix any datum \(v_0\) from Definition~\ref{def:general_second_test}, independently of \(k\), and substitute the normal approximations for \(v_k\) and \((\mathrm{i}k)\widehat{\tilde u}\).

{
The source factors in \eqref{eq:first_low_freq_general_c} are \(O(|k|)\), and \eqref{eq:laplace_bound_general} gives \(|k|^2\|\widehat{\tilde u}(\cdot,k)\|_{L^\infty(\Omega)}=O(|k|)\).
Consequently, replacing \(v_k\) by its normal approximation using \eqref{eq:normal_test_error_general} produces an \(O(|k|^{-2})\) error.
}
For the field term, \eqref{eq:one_face_expansion} gives an \(O(|k|^{-3})\) error in \(\widehat{\tilde u}\); multiplication by \(-k^2\) and integration against the normal test, whose \(L^1\)-norm is \(O(|k|^{-1})\), again give \(O(|k|^{-2})\).

Set \(y=\mathrm{i}k(b_3-x_3)\) and evaluate \(v^{(2)}\) with \(r=x_1\).
The contribution outside the fixed neighbourhood and the error from extending the normal integrals to \((0,\infty)\) are exponentially small.
The mixed normal integral is
}
\[
\int_0^\infty\exp\!\left(-\frac{y}{c_0(r)\sqrt{\sigma_0(r)}}-\frac{y}{\tilde c_0(r)\sqrt{\sigma_0(r)}}\right)\,\mathrm{d}y
=\frac{c_0(r)\tilde c_0(r)\sqrt{\sigma_0(r)}}{c_0(r)+\tilde c_0(r)}.
\]
{Combining these normal integrals with \eqref{eq:face_normal_moment}, we obtain the following high-frequency form of the first integral identity \eqref{eq:first_low_freq_general_c}:}
\begin{equation}\label{eq:high_frequency_first_moment_general}
\begin{aligned}
&\int_{a_1}^{b_1}\int_{a_2}^{b_2}v_0(x_1,x_2,b_3)\left(1+\frac{\sqrt{\sigma_0(x_1)}}{c_0(x_1)}\right)
\Biggl[\frac{\sqrt{\sigma_0(x_1)}}{\sqrt{\sigma_0(x_1)}+c_0(x_1)}
\left(f_0(x_2)+\frac{h_0(x_1)}{\mathrm{i}k}\right)\\
&\qquad-\frac{\sqrt{\sigma_0(x_1)}}{\sqrt{\sigma_0(x_1)}+\tilde c_0(x_1)}
\left(\tilde f_0(x_2)+\frac{\tilde h_0(x_1)}{\mathrm{i}k}\right)\Biggr]\,\mathrm{d}x_2\,\mathrm{d}x_1
=O(|k|^{-2}).
\end{aligned}
\end{equation}
{
Letting \(|k|\to\infty\) in \eqref{eq:high_frequency_first_moment_general} removes the velocity and remainder terms.
The fixed trace \(v_0(\cdot,\cdot,b_3)\) ranges over \(C_c^\infty((a_1,b_1)\times(a_2,b_2))\), and the prefactor \(1+\sqrt{\sigma_0(x_1)}/c_0(x_1)\) is positive.
Hence continuity gives, for every \((x_1,x_2)\in(a_1,b_1)\times(a_2,b_2)\),
}
\begin{equation}\label{eq:f_weighted_face_identity}
\frac{\sqrt{\sigma_0(x_1)}}{\sqrt{\sigma_0(x_1)}+c_0(x_1)}f_0(x_2)
=\frac{\sqrt{\sigma_0(x_1)}}{\sqrt{\sigma_0(x_1)}+\tilde c_0(x_1)}\tilde f_0(x_2).
\end{equation}
}

\leavevmode{
{
We next use \eqref{eq:f_weighted_face_identity} to cancel the displacement terms in \eqref{eq:high_frequency_first_moment_general} and multiply by \(\mathrm{i}k\).
The remainder is now \(O(|k|^{-1})\) and tends to zero as \(|k|\to\infty\), so the same argument gives, for every \(x_1\in(a_1,b_1)\),
}
\begin{equation}\label{eq:h_weighted_face_identity}
\frac{\sqrt{\sigma_0(x_1)}}{\sqrt{\sigma_0(x_1)}+c_0(x_1)}h_0(x_1)
=\frac{\sqrt{\sigma_0(x_1)}}{\sqrt{\sigma_0(x_1)}+\tilde c_0(x_1)}\tilde h_0(x_1).
\end{equation}
}

{
By the separated form in Admissibility~\ref{ass:general_structure}, \eqref{eq:f_weighted_face_identity} and \eqref{eq:h_weighted_face_identity} give the two weighted source relations throughout \(\Omega\).
}
\endgroup

\begingroup
\begingroup
\begingroup
\subsection{Step 6: Low-frequency information on \texorpdfstring{\(h/c^2\)}{h/c2}}\label{subsec:general_low_frequency}
\leavevmode
\endgroup
To separate the sources from the remaining coefficient \(c\), we now take the low-frequency limit of \eqref{eq:first_low_freq_general_c}.
The third test family gives an orthogonality identity for \(h/c^2-\tilde h/\tilde c^2\), and zero-frequency CGO solutions then show that the two quotients agree.
We combine this quotient identity with the high-frequency relations in Subsection~\ref{subsec:general_final_recovery} to recover \(c\), \(h\), and \(f\).

\endgroup

\subsubsection{The third test function}\label{subsubsec:general_third_test}

\begingroup
\leavevmode\par

\begingroup
\begin{definition}\label{def:general_third_test}
Fix any \(\sigma\)-harmonic function \(v_0\in H^1(\Omega)\).
Let \(v_k,w_k\) be defined by \eqref{eq:general_exact_tests}.
These form the third test family, used as \(k\to0\) along \(-\mathrm{i}\mathbb R_+\).
\end{definition}
\endgroup

\endgroup

\subsubsection{Low-frequency time-domain decoupling}

\begin{proposition}\label{prop:c_identity}
For every \(\sigma\)-harmonic function \(v_0\in H^1(\Omega)\),
\begin{equation}\label{eq:c_identity_xi_rigorous}
\int_\Omega
\left(\frac{h}{c^2}-\frac{\tilde h}{\tilde c^2}\right)
v_0\,\mathrm dx=0.
\end{equation}
\end{proposition}

\begingroup
\begin{proof}
Fix \(v_0\), and let \(v_k,w_k\) be the tests in Definition~\ref{def:general_third_test}.
Subtracting the equation for \(v_0\) from the Dirichlet equation in \eqref{eq:general_exact_tests} and testing with \(v_k-v_0\in H_0^1(\Omega)\), we use \(-k^2>0\) and Poincar\'e's inequality to obtain
\[
\|v_k-v_0\|_{H^1(\Omega)}\le C|k|^2\|v_0\|_{L^2(\Omega)},\qquad 0<|k|\le1.
\]
Thus \(v_k\to v_0\) in \(H^1(\Omega)\) as \(k\to0\), with \(v_k\) bounded in \(L^2(\Omega)\).

The global energy identity for \(\hat u(\cdot,k)\) and the Sobolev inequality in \(\mathbb{R}^3\) imply
\[
\begin{aligned}
\min\{1,\sigma_-\}\|\nabla\hat u(\cdot,k)\|_2^2
&\le\|c^{-2}(\mathrm{i}kf+h)\|_{6/5}\|\hat u(\cdot,k)\|_6\\
&\le C\|c^{-2}(\mathrm{i}kf+h)\|_{6/5}\|\nabla\hat u(\cdot,k)\|_2.
\end{aligned}
\]
Since the sources are bounded and compactly supported,
\[
\|\nabla\hat u(\cdot,k)\|_{L^2(\mathbb{R}^3)}+\|\hat u(\cdot,k)\|_{L^2(\Omega)}\le C,\qquad 0<|k|\le1.
\]
The same estimate holds for the tilded field.
Consequently, the right-hand side of \eqref{eq:first_low_freq_general_c} is \(O(|k|^2)\).
The source term multiplied by \(\mathrm{i}k\) also tends to zero.
Letting \(k\to0\) along \(-\mathrm{i}\mathbb R_+\) in \eqref{eq:first_low_freq_general_c}, the convergence \(v_k\to v_0\) in the remaining source pairing gives \eqref{eq:c_identity_xi_rigorous}.
\end{proof}
\endgroup

\subsubsection{Zero-frequency analysis}

\leavevmode
We now choose \(v_0\) in Proposition~\ref{prop:c_identity} from the zero-frequency CGO family.
The temporal limit \(k\to0\) has already been taken, so the following construction uses a large spatial phase at zero temporal frequency.

The corresponding spatial equation is
\endgroup

\begin{equation*}
\nabla\cdot(\sigma\nabla v_\xi)=0
\qquad\text{in }\Omega.
\end{equation*}

Introducing the Liouville transform
\begin{equation}\label{eq:u0_def}
u_\xi(x):=\sigma(x)^{1/2}v_\xi(x),
\end{equation}
we obtain
\begin{equation}\label{eq:u0_eq}
-\Delta u_\xi+q_0u_\xi=0
\qquad\text{in }\Omega,
\end{equation}
where
\begin{equation}\label{eq:q0_def}
q_0:=\sigma^{-1/2}\Delta(\sigma^{1/2}).
\end{equation}

The stated regularity and positivity of \(\sigma\) give \(q_0\in W^{2,\infty}(\Omega)\).
The following lemma constructs the \(\sigma\)-harmonic tests for Proposition~\ref{prop:c_identity} and controls both the CGO correction and its gradient.

\begin{lemma}\label{lem:zero_freq_CGO_decay}
Assume that \(\sigma\in W^{4,\infty}(\Omega)\) satisfies the bounds in \eqref{eq:general_bounds}.

For \(\delta=1/2\), there exists a constant
 \(C_{\Omega,\sigma}>0\) such that, for every \(\xi\in\mathbb C^3\) with \(\xi\cdot\xi=0\) and \(|\xi|\) sufficiently large, equation \eqref{eq:u0_eq} admits a solution of the form
\begin{equation}\label{eq:u_zero_freq_cgo}
u_\xi(x)=e^{\xi\cdot x}(1+\psi_\xi(x)).
\end{equation}
Moreover,
\begin{equation}\label{eq:psi_zero_freq_Linf}
\|\psi_\xi\|_{L^\infty(\Omega)}+\|\nabla\psi_\xi\|_{L^\infty(\Omega)}
\le
C_{\Omega,\sigma}|\xi|^{-\delta},
\end{equation}
which also gives \(\|\nabla\psi_\xi/|\xi|\|_{L^\infty(\Omega)}\le C_{\Omega,\sigma}|\xi|^{-\delta-1}\).
The constant \(C_{\Omega,\sigma}\) and the lower threshold for \(|\xi|\) depend only on \(\Omega\), the a priori \(W^{4,\infty}(\Omega)\) bound for \(\sigma\), and the bounds in \eqref{eq:general_bounds}.
Equivalently,
\begin{equation}\label{eq:v_zero_freq_cgo}
v_\xi(x)=\sigma(x)^{-1/2}e^{\xi\cdot x}(1+\psi_\xi(x))
\end{equation}
solves
\begin{equation}\label{eq:v_zero_freq_eq_full}
\nabla\cdot(\sigma(x)\nabla v_\xi(x))=0
\quad\text{in }\Omega.
\end{equation}
\end{lemma}

\begingroup
\begin{proof}
By \eqref{eq:q0_def} and the assumptions on \(\sigma\), a Sobolev extension followed by a smooth cutoff gives a compactly supported \(\widetilde q_0\in W^{2,\infty}(\mathbb R^3)\), equal to \(q_0\) in \(\Omega\), with norm controlled by the a priori bounds in the statement.
Take \(p=4\) and \(\tilde p=4/3\), so that \(1/\tilde p-1/p=1/2\).
Compact support and the product rule give \(\widetilde q_0\in H^{2,\tilde p}(\mathbb R^3)\) and bounded multiplication by \(\widetilde q_0\) from \(H^{2,p}(\mathbb R^3)\) to \(H^{2,\tilde p}(\mathbb R^3)\), with both bounds controlled by the same a priori quantities.

The solution operator \(G_\xi\) for \(\Delta+2\xi\cdot\nabla\) satisfies
\[
\|G_\xi\|_{H^{2,\tilde p}(\mathbb R^3)\to H^{2,p}(\mathbb R^3)}\le C|\xi|^{-1/2}
\]
by \cite[Lemma 3.1 and the proof of Proposition 3.1]{cakoni2020corner}.
For sufficiently large \(|\xi|\), the map \(\psi_\xi\mapsto G_\xi\bigl(\widetilde q_0(1+\psi_\xi)\bigr)\) is therefore a contraction on \(H^{2,p}(\mathbb R^3)\).
Its fixed point satisfies
\[
(\Delta+2\xi\cdot\nabla)\psi_\xi=\widetilde q_0(1+\psi_\xi),\qquad
\|\psi_\xi\|_{H^{2,p}(\mathbb R^3)}\le C_{\Omega,\sigma}|\xi|^{-1/2}.
\]
Since \(\widetilde q_0=q_0\) in \(\Omega\), the function \(u_\xi\) in \eqref{eq:u_zero_freq_cgo} solves \eqref{eq:u0_eq}.

The embedding \(H^{2,4}(\Omega)\hookrightarrow C^{1,1/4}(\overline\Omega)\) yields \eqref{eq:psi_zero_freq_Linf}.
Reversing the Liouville transform \eqref{eq:u0_def} gives \eqref{eq:v_zero_freq_cgo} and \eqref{eq:v_zero_freq_eq_full}.

The proof is complete.\qedhere\par
\end{proof}
\endgroup

Take \(v_0=v_\xi\) from \eqref{eq:v_zero_freq_cgo} in Lemma~\ref{lem:zero_freq_CGO_decay}.
This is a \(\sigma\)-harmonic datum for the third test family in Definition~\ref{def:general_third_test}, so Proposition~\ref{prop:c_identity} applies.
Using the separated form in Admissibility condition~\ref{ass:general_structure} and the common conductivity \eqref{eq uniq sigma}, we obtain

\begin{equation}\label{eq:H_identity}
\int_\Omega H(x_1)e^{\xi\cdot x}(1+\psi_\xi(x))\,\mathrm dx=0,
\end{equation}
where
\begin{equation}\label{eq:H_def}
H(x_1):=
\left(\frac{h_0(x_1)}{c_0(x_1)^2}
-\frac{\tilde h_0(x_1)}{\tilde c_0(x_1)^2}\right)
\sigma_0(x_1)^{-1/2}.
\end{equation}

By Admissibility condition~\ref{ass:general_structure}, \(H\) in \eqref{eq:H_def} depends only on \(x_1\), so we seek its Fourier transform at every real frequency.
Fix \(\lambda\in\mathbb R\) and keep it as the \(x_1\) frequency.
We choose the transverse components to satisfy the null condition in Lemma~\ref{lem:zero_freq_CGO_decay} and to make the integral of the \(x_3\) exponential factor nonzero.
The choice below allows \(s\to\infty\) for every fixed \(\lambda\), retaining the full range of Fourier frequencies in \(x_1\).
Specifically, take \(s>|\lambda|\) along the unbounded sequence characterized by
\[
\mu_{\lambda,s}=\sqrt{s^2-\lambda^2},\qquad
\exp\!\bigl(\mathrm{i}\mu_{\lambda,s}(b_3-a_3)\bigr)=-1.
\]
For these values of \(s\), take
\begin{equation}\label{eq:xi_lambda_s_3d}
\xi_{\lambda,s}=(\mathrm{i}\lambda,s,\mathrm{i}\mu_{\lambda,s}).
\end{equation}
Then
\begin{equation*}
\xi_{\lambda,s}\cdot\xi_{\lambda,s}=0,
\end{equation*}
and
\begin{equation*}
\exp(\xi_{\lambda,s}\cdot x)=\exp(\mathrm{i}\lambda x_1)\exp(s x_2)\exp(\mathrm{i}\mu_{\lambda,s}x_3).
\end{equation*}
Since \(|\xi_{\lambda,s}|=\sqrt{2}s\), Lemma~\ref{lem:zero_freq_CGO_decay} applies for all sufficiently large \(s\) along this sequence.

\begingroup
\subsubsection{Determination of \texorpdfstring{\(h/c^2\)}{h/c squared}}
We now apply Fourier analysis to \eqref{eq:H_identity}.
\endgroup

\begingroup
For \(\xi=\xi_{\lambda,s}\), write \(\psi_{\lambda,s}=\psi_{\xi_{\lambda,s}}\) and \(\widehat H(\lambda)=\int_{a_1}^{b_1}H(x_1)\exp(\mathrm{i}\lambda x_1)\,\mathrm{d}x_1\).
Separating the principal term in \eqref{eq:H_identity} gives
\begin{equation}\label{eq:H_main_error_zero}
\begin{aligned}
&\widehat H(\lambda)\int_{a_2}^{b_2}\exp(sx_2)\,\mathrm{d}x_2\int_{a_3}^{b_3}\exp(\mathrm{i}\mu_{\lambda,s}x_3)\,\mathrm{d}x_3\\
&\qquad=-\int_\Omega H(x_1)\exp(\xi_{\lambda,s}\cdot x)\psi_{\lambda,s}(x)\,\mathrm{d}x.
\end{aligned}
\end{equation}
The choice of \(\mu_{\lambda,s}\) ensures that the transverse factors satisfy
\begin{equation}\label{eq:I1_factor}
\int_{a_2}^{b_2}\exp(sx_2)\,\mathrm{d}x_2>0,\qquad
\left|\int_{a_3}^{b_3}\exp(\mathrm{i}\mu_{\lambda,s}x_3)\,\mathrm{d}x_3\right|=\frac2{\mu_{\lambda,s}}.
\end{equation}
The oscillatory integral in \eqref{eq:I1_factor} is of order \(s^{-1}\).
To compare the remainder with the principal term, we integrate once by parts in \(x_3\), using the independence of \(H\) from \(x_3\).
This gives the same factor \(\mu_{\lambda,s}^{-1}\) in the boundary term and the integral containing \(\partial_{x_3}\psi_{\lambda,s}\).
The boundary values contribute at most \(2\|\psi_{\lambda,s}\|_{L^\infty(\Omega)}\), while the interior integral in \(x_3\) is bounded by \((b_3-a_3)\|\partial_{x_3}\psi_{\lambda,s}\|_{L^\infty(\Omega)}\).
Thus, with \(C\) independent of \(s\),
\begin{equation}\label{eq:I2_bound_final_new}
\begin{aligned}
&\left|\int_\Omega H(x_1)\exp(\xi_{\lambda,s}\cdot x)\psi_{\lambda,s}(x)\,\mathrm{d}x\right|\\
&\quad\le C\frac{\int_{a_2}^{b_2}\exp(sx_2)\,\mathrm{d}x_2}{\mu_{\lambda,s}}
\left(\|\psi_{\lambda,s}\|_{L^\infty(\Omega)}+\|\partial_{x_3}\psi_{\lambda,s}\|_{L^\infty(\Omega)}\right).
\end{aligned}
\end{equation}
Dividing \eqref{eq:H_main_error_zero} by the nonzero transverse factors in \eqref{eq:I1_factor} and using \eqref{eq:I2_bound_final_new}, we obtain
\[
|\widehat H(\lambda)|\le C\left(\|\psi_{\lambda,s}\|_{L^\infty(\Omega)}+\|\partial_{x_3}\psi_{\lambda,s}\|_{L^\infty(\Omega)}\right).
\]
By \eqref{eq:psi_zero_freq_Linf} in Lemma~\ref{lem:zero_freq_CGO_decay} and \(|\xi_{\lambda,s}|=\sqrt{2}s\), the right-hand side tends to zero as \(s\to\infty\) along the chosen sequence.
Since \(\widehat H(\lambda)\) is independent of \(s\), it follows that
\endgroup
\begin{equation}\label{eq:Hhat_zero_final}
\widehat H(\lambda)=\int_{a_1}^{b_1}
H(x_1)e^{\mathrm{i}\lambda x_1}\,\mathrm{d}x_1=0
\qquad
\text{for every }\lambda\in\mathbb R.
\end{equation}

Viewing \(H\) as a function on \(\mathbb R\) by zero extension outside \((a_1,b_1)\), we have \(H\in L^1(\mathbb R)\).
Since its Fourier transform vanishes identically by \eqref{eq:Hhat_zero_final}, the injectivity of the Fourier transform yields
\begin{equation}\label{eq:H_x1_identity}
H(x_1)=0
\qquad\text{for a.e. }x_1\in(a_1,b_1).
\end{equation}

\begingroup
Using \eqref{eq:H_x1_identity} in \eqref{eq:H_def} and recalling \(\sigma_0>0\), we obtain
\begin{equation}\label{eq:h_quotient_identity}
\frac{h_0(x_1)}{c_0(x_1)^2}
=\frac{\tilde h_0(x_1)}{\tilde c_0(x_1)^2}
\qquad\text{for }x_1\in(a_1,b_1).
\end{equation}
The equality holds everywhere by continuity.

\Needspace{8\baselineskip}
\begingroup
\subsection{Step 7: Joint recovery of \texorpdfstring{\(c\), \(h\), and \(f\)}{c, h, and f}}\label{subsec:general_final_recovery}
\leavevmode
\endgroup
\leavevmode
\hyperref[subsec:general_sigma_recovery]{Step~4} gives the common conductivity \(\sigma_0(x_1)=\tilde\sigma_0(x_1)>0\) for every \(x_1\in(a_1,b_1)\), as recorded in \eqref{eq uniq sigma}.
The regularity in Theorem~\ref{thm:general} and the separated form in Admissibility~\ref{ass:general_structure} allow us to work with continuous representatives of all profiles, so the identities below hold pointwise.

\hyperref[subsec:general_source_tests]{Step~5} gives the high-frequency weighted displacement and velocity relations \eqref{eq:f_weighted_face_identity} and \eqref{eq:h_weighted_face_identity}.
The latter reads
\[
\frac{\sqrt{\sigma_0(x_1)}}{\sqrt{\sigma_0(x_1)}+c_0(x_1)}h_0(x_1)
=\frac{\sqrt{\sigma_0(x_1)}}{\sqrt{\sigma_0(x_1)}+\tilde c_0(x_1)}\tilde h_0(x_1).
\]
\hyperref[subsec:general_low_frequency]{Step~6} gives the low-frequency quotient identity \eqref{eq:h_quotient_identity},
\[
\frac{h_0(x_1)}{c_0(x_1)^2}
=\frac{\tilde h_0(x_1)}{\tilde c_0(x_1)^2}.
\]
We combine these two velocity identities to determine \(c\); the quotient identity then determines \(h\), and the weighted displacement relation determines \(f\).

We first recover \(c\).
Fix \(x_1\in(a_1,b_1)\) with \(h_0(x_1)\ne0\).
The weights in the high-frequency equality are positive by \eqref{eq:general_bounds}, so \(\tilde h_0(x_1)\ne0\) as well, and their common weighted value is nonzero.
We may therefore divide each side of the quotient identity by the corresponding side of the weighted identity.
On the first side, the source factor cancels as follows:
\[
\frac{h_0(x_1)}{c_0(x_1)^2}
\frac{\sqrt{\sigma_0(x_1)}+c_0(x_1)}{\sqrt{\sigma_0(x_1)}h_0(x_1)}
=\frac{\sqrt{\sigma_0(x_1)}+c_0(x_1)}{\sqrt{\sigma_0(x_1)}c_0(x_1)^2}.
\]
The same calculation on the tilded side, followed by expansion of the two numerators, gives
\[
\frac{1}{c_0(x_1)^2}
+\frac{1}{c_0(x_1)\sqrt{\sigma_0(x_1)}}
=
\frac{1}{\tilde c_0(x_1)^2}
+\frac{1}{\tilde c_0(x_1)\sqrt{\sigma_0(x_1)}}.
\]
For each fixed positive \(\sigma_0(x_1)\), the function
\[
t\longmapsto t^{-2}+\sigma_0(x_1)^{-1/2}t^{-1}
\]
has derivative \(-2t^{-3}-\sigma_0(x_1)^{-1/2}t^{-2}<0\) for \(t>0\), so it is strictly decreasing.
Since \(c_0\) and \(\tilde c_0\) are positive, strict monotonicity gives \(c_0(x_1)=\tilde c_0(x_1)\) at every point where \(h_0(x_1)\ne0\).

We next extend the equality across the zeros of \(h_0\).
The condition \(\operatorname{supp}(h)=\overline\Omega\) in Admissibility~\ref{ass:general_structure} makes the set where \(h_0\ne0\) dense in \((a_1,b_1)\).
Indeed, if \(h_0\) vanished on an open subinterval, then \(h=h_0(x_1)\chi_\Omega\) would vanish on the corresponding open {subset} of \(\Omega\), contrary to this support condition.
Every point of \((a_1,b_1)\) is therefore a limit of points where the two coefficient profiles \(c_0\) and \(\tilde c_0\) have already been shown to agree.
Continuity of \(c_0\) and \(\tilde c_0\) gives their equality at the limiting point as well.
Thus \(c_0=\tilde c_0\) on the whole interval, and their separated forms give
\begin{equation}\label{eq uniq c}
c=\tilde c\qquad\text{in }\Omega.
\end{equation}

We now recover \(h\) on the whole interval, including the points where it vanishes.
Substituting the common coefficient profile \(c_0\) into the quotient identity from \hyperref[subsec:general_low_frequency]{Step~6} gives
\[
\frac{h_0(x_1)-\tilde h_0(x_1)}{c_0(x_1)^2}=0,
\qquad x_1\in(a_1,b_1).
\]
Multiplication by \(c_0(x_1)^2>0\) yields \(h_0(x_1)=\tilde h_0(x_1)\) at every point.
Since \(h(x)=h_0(x_1)\) and \(\tilde h(x)=\tilde h_0(x_1)\) in \(\Omega\), we obtain
\begin{equation}\label{eq uniq h}
h=\tilde h\qquad\text{in }\Omega.
\end{equation}

Finally, we use the high-frequency displacement relation \eqref{eq:f_weighted_face_identity} from \hyperref[subsec:general_source_tests]{Step~5}.
The two configurations now have the same profiles \(\sigma_0\) and \(c_0\), so the two weights are identical, and the relation becomes
\[
\frac{\sqrt{\sigma_0(x_1)}}{\sqrt{\sigma_0(x_1)}+c_0(x_1)}
\bigl(f_0(x_2)-\tilde f_0(x_2)\bigr)=0.
\]
This holds for every \((x_1,x_2)\in(a_1,b_1)\times(a_2,b_2)\).
Fixing any \(x_1\in(a_1,b_1)\), the common weight is strictly positive, and cancelling it gives \(f_0(x_2)=\tilde f_0(x_2)\) for every \(x_2\in(a_2,b_2)\).
The separated forms of the initial displacements therefore yield
\begin{equation}\label{eq uniq f}
f=\tilde f\qquad\text{in }\Omega.
\end{equation}
\endgroup

\medskip
Combining \eqref{eq uniq sigma}, \eqref{eq uniq c}, \eqref{eq uniq h}, and \eqref{eq uniq f}, we complete the proof of Theorem~\ref{thm:general}.
\end{proof}

\section{\texorpdfstring{{Proofs for the parabolic and Schr\"odinger models}}{Proofs for the parabolic and Schr\"odinger models}}\label{sec:parabolic}

\begingroup
The parabolic and Schr\"odinger models give two further applications of the unified framework in Subsection~\ref{subsec:framework}.
The proofs below combine integral identities with suitable adjoint test functions and frequency asymptotics to separate the source and the medium.
Throughout the proofs, we work in \(\mathbb R^3\) and use the prescribed profiles with their zero extensions outside \(I=(a,b)\).
The endpoint value \(\tilde\phi(b)\) is taken from inside \(I\).
\endgroup

\begingroup
For example, the unit profile \(\chi_I\) is admissible for the source and for the potential profile in these models.
In the heat case one may take \(\mu=1+(M(x')-1)\chi_I(x_3)\) with \(M\) uniformly positive and \(M-1\in C_c(\mathbb R^2)\), in which case the quotient then has profile \(\chi_I\).
\endgroup

\subsection{\texorpdfstring{{Proof for the parabolic model}}{Proof for the parabolic model}}

\begingroup
We work with the parabolic model \eqref{eq:parabolic-isotropic} and the passive boundary observation \eqref{eq:parabolic-observation}.

\begin{proof}[Proof of Theorem~\ref{ext:parabolic-model}]
The first integral identity will recover \(f/\mu\) at low frequency, and the second will separate \(f\) and \(\mu\) at high frequency.

\subsubsection{Step 1: Two integral identities}

Let \(u,\tilde u\) be the heat solutions in \(L^2(\mu^{-1}\,\mathrm dx)\) and \(L^2(\tilde\mu^{-1}\,\mathrm dx)\), respectively.
Both are strongly continuous at \(t=0\), and the weighted energy estimate gives \(\sup_{t\ge0}\|u(t)\|_2\le C\|f\|_2\), with the same bound for \(\tilde u\).
The integrated heat equation gives \(-\Delta\int_0^t u(\cdot,r)\,\mathrm dr=(f-u(t))/\mu\), so the time integral is continuous in \(H^2(\mathbb R^3)\) and vanishes initially.
We interpret the observed trace as the distributional time derivative of its boundary trace, as in Subsection~\ref{subsec:weak-boundary-data}.
Equal observations give equal normal derivatives by Remark~\ref{rem:other_boundary_flux}.
All tests below have an \(H^2(\Omega)\) spatial factor and an exponentially decaying time factor.
For a test function \(w\) satisfying the adjoint parabolic equation
\[
-\frac{1}{\mu(x)}\partial_t w(x,t)-\Delta w(x,t)=0,
\qquad
(x,t)\in\Omega\times\mathbb R_+,
\]
integration by parts as in the proof of Lemma~\ref{lem:integral_identity} gives
\begin{equation}\label{eq:parabolic-firstID}
\int_0^\infty\!\!\int_\Omega
\left(\frac{1}{\mu}-\frac{1}{\tilde\mu}\right)\tilde u\,\partial_t w\,\mathrm{d}x\,\mathrm{d}t
+
\int_\Omega
\left(\frac{f}{\mu}-\frac{\tilde f}{\tilde\mu}\right)w(x,0)\,\mathrm{d}x
=0.
\end{equation}

For weak solutions, we first test on a finite time interval, where the equal weak Cauchy data cancel the spatial boundary terms.
Strong continuity at \(t=0\) gives the initial term in \eqref{eq:parabolic-firstID}.
The energy bound and the exponential decay of the test make the terminal term vanish as the upper endpoint tends to infinity.

For the second identity, take \(s>0\) and
\[
w_s(x,t)=e^{-st}v(x),
\qquad
v(x)=e^{-\mathrm{i}\zeta\cdot x},
\qquad
\zeta\cdot\zeta=-s,
\]
so that
\[
-\partial_t w_s-\Delta w_s=0.
\]
Applying the same Green formula to the two equations gives
\begin{multline}\label{eq:parabolic-secondID}
\int_\Omega
\left(\frac{f}{\mu}-\frac{\tilde f}{\tilde\mu}\right)w_s(x,0)\,\mathrm{d}x
\\ =
\int_0^\infty\!\!\int_\Omega
\left(1-\frac{1}{\mu}\right)u\,\partial_t w_s\,\mathrm{d}x\,\mathrm{d}t
-
\int_0^\infty\!\!\int_\Omega
\left(1-\frac{1}{\tilde\mu}\right)\tilde u\,\partial_t w_s\,\mathrm{d}x\,\mathrm{d}t.
\end{multline}

\subsubsection{Step 2: Recovery of the quotient \(f/\mu\)}

To exploit the first identity, we use a decaying adjoint solution of the form
\[
w(x,t)=e^{-\rho_0^2 t+\mathrm i\rho'\cdot x}(1+\psi_\rho(x)),\qquad \rho_0>0 .
\]
Fix a null phase \(\rho'\) from Lemma~\ref{lem G zeta}.
The adjoint equation reduces to \((\Delta+2\mathrm i\rho'\cdot\nabla)\psi_\rho=\rho_0^2\mu^{-1}(1+\psi_\rho)\).
For small \(\rho_0\), the right inverse in that lemma gives a contraction on \(L^2(Q)\), exactly as in Lemma~\ref{lem rho_0}.
Thus \(\psi_\rho\in H^2(Q)\) and \(\|\psi_\rho\|_{L^2(Q)}\le C\rho_0^2\).

Write \(\widehat u(s)=\int_0^\infty u(\cdot,t)e^{-st}\,\mathrm dt\).
Testing \((-\Delta+s/\tilde\mu)\widehat{\tilde u}=\tilde f/\tilde\mu\) by \(\overline{\widehat{\tilde u}}\) and using the Sobolev inequality on \(\mathbb R^3\) yields
\[
\|\widehat{\tilde u}(s)\|_{L^2(\Omega)}
\le C\|\nabla\widehat{\tilde u}(s)\|_{L^2(\mathbb R^3)}
\le C\|\tilde f/\tilde\mu\|_{L^{6/5}(\mathbb R^3)} .
\]
Since \(\partial_t w=-\rho_0^2w\), this bound at \(s=\rho_0^2\) and the Cauchy--Schwarz inequality in \eqref{eq:parabolic-firstID} give
\[
\begin{aligned}
\left|\int_\Omega\left(\frac f\mu-\frac{\tilde f}{\tilde\mu}\right)
e^{\mathrm i\rho'\cdot x}\,\mathrm dx\right|
&\le C_{\rho'}\bigl(\rho_0^2\|\widehat{\tilde u}(\rho_0^2)\|_{L^2(\Omega)}
+\|\psi_\rho\|_{L^2(Q)}\bigr)\\
&\le C_{\rho'}\rho_0^2 .
\end{aligned}
\]
Letting \(\rho_0\to0^+\) therefore makes the integral on the left vanish for each fixed \(\rho'\).
As in the proof of Theorem~\ref{thm0}, the common separated profile factors this integral into a transverse Fourier coefficient and the nonzero profile integral in \eqref{eq:nonde}.
The coefficients on the shifted transverse lattice all vanish, so completeness gives
\begin{equation}\label{eq:parabolic-quotient}
\frac{f(x)}{\mu(x)}=\frac{\tilde f(x)}{\tilde\mu(x)}
\qquad\text{a.e. in }\Omega .
\end{equation}

\subsubsection{Step 3: High-frequency separation of \(f\) and \(\mu\)}

We now use the additional assumptions for joint recovery.
The sources have the separated forms \(f=p(x')\tilde\phi(x_3)\), \(\tilde f=\tilde p(x')\tilde\phi(x_3)\), with the same prescribed profile \(\tilde\phi\).
The source profile \(\tilde\phi\) need not coincide with the common profile of \(f/\mu\) and \(\tilde f/\tilde\mu\).
We have \(\tilde\phi(b)\ne0\), and both diffusivities equal one above \(x_3=b\).
The quotient \(f/\mu\) is already known by \eqref{eq:parabolic-quotient}.
We will obtain a second relation between the source and the diffusivity at the upper face \(x_3=b\).
Together these two relations determine \(p(x')\), which determines the source throughout the strip because \(\tilde\phi\) is prescribed.

In \(L^2(\mathbb R^3)\), write
\begin{equation}\label{eq:parabolic-asymptotic-expansion}
\widehat u(x,s)=\frac{f(x)}s+\mathbf R(x,s).
\end{equation}
Strong continuity at \(t=0\) and the energy bound give \(s\|\mathbf R(\cdot,s)\|_2\to0\).
This convergence does not determine the trace at \(x_3=b\), where the zero extension of the source may jump.
{We first obtain equality of the transformed traces from \eqref{eq:parabolic-secondID}, and then compute their high-frequency limits using the trace estimates and density argument in the proof of Proposition~\ref{prop:TATPAT_endface_limit}, as detailed below.}

Fix \(s>0\) and \(m\in\mathbb R^2\).
Choose \(\zeta=(m,\mathrm i\sqrt{s+|m|^2})\) in the free test from Step 1 and multiply it by \(e^{-b\sqrt{s+|m|^2}}\).
The normalized test has the form
\[
w_s(x,t)=e^{-st}v(x),\qquad
v(x)=e^{-\mathrm im\cdot x'+(x_3-b)\sqrt{s+|m|^2}}.
\]
Using \(\partial_t w_s=-s w_s\) and the definition of the Laplace transform, \eqref{eq:parabolic-secondID} becomes
\[
\begin{aligned}
&\int_\Omega\bigl[f/\mu+s(1-\mu^{-1})\widehat u\bigr]v\,\mathrm dx\\
&\quad=\int_\Omega\bigl[\tilde f/\tilde\mu+s(1-\tilde\mu^{-1})\widehat{\tilde u}\bigr]v\,\mathrm dx .
\end{aligned}
\]

To identify these integrals, rewrite the transformed heat equation as
\[
(-\Delta+s)\widehat u=f/\mu+s(1-\mu^{-1})\widehat u.
\]
The right-hand side is compactly supported in \(\Omega\cap\{x_3\le b\}\), since the source vanishes and the diffusivity equals one above the face.
Taking the Fourier transform in \(x'\) reduces the operator on the left to \(-\partial_3^2+s+|m|^2\).
The decaying Green kernel of this one-dimensional operator is
\[
\frac{e^{-|x_3-y_3|\sqrt{s+|m|^2}}}{2\sqrt{s+|m|^2}}.
\]
At \(x_3=b\), every source point has \(y_3\le b\), so the absolute value in the exponent is \(b-y_3\).
The resulting representation is
\[
\begin{aligned}
&2\sqrt{s+|m|^2}\int_{\mathbb R^2}\widehat u(x',b,s)e^{-\mathrm im\cdot x'}\,\mathrm dx'\\
&\quad=\int_\Omega
  \bigl[f/\mu+s(1-\mu^{-1})\widehat u\bigr]
  e^{-\mathrm im\cdot x'+(x_3-b)\sqrt{s+|m|^2}}\,\mathrm dx .
\end{aligned}
\]
The same representation holds for \(\widehat{\tilde u}\).
Its right-hand side agrees with that for \(\widehat u\) by \eqref{eq:parabolic-secondID} after the time integration.
Thus the Fourier transforms of the two face traces agree for every \(m\in\mathbb R^2\).
The transformed equations give \(H^2(\mathbb R^3)\) regularity, so the face traces are continuous.
Far from \(\Omega\), the full-space free resolvent kernel gives exponential decay because the right-hand side is compactly supported.
The face traces therefore belong to \(L^1(\mathbb R^2)\).
Fourier uniqueness gives equality almost everywhere on the face, and continuity gives
\begin{equation}\label{eq:parabolic-face-trace}
\widehat u(x',b,s)=\widehat{\tilde u}(x',b,s),
\qquad x'\in\mathbb R^2,\quad s>0 .
\end{equation}

We next compute the limit of the trace as \(s\to\infty\).
Multiplying the transformed heat equation by \(s\) gives
\[
(-\Delta+s/\mu)(s\widehat u)=s f/\mu.
\]
{
With \(c=\sqrt\mu\) and \(\mathrm ik=\sqrt s\), this coincides with the equation obtained by multiplying \eqref{eq:PDE_hat_u_three} by \(\mathrm ik\).
{The uniform trace estimate in the proof of Proposition~\ref{prop:TATPAT_endface_limit} applies with these coefficients.}
Moreover, \(\mu\) has a continuous interior trace at the flat face, and \(f=0,\ \mu=1\) above it.
The same uniform \(L^2\) trace estimate and density argument therefore give, for \(p\in L^2(\mathbb R^2)\),
\begin{equation}\label{eq:parabolic-face-limit}
s\widehat u(\cdot,b,s)
\longrightarrow\frac{p(\cdot)\tilde\phi(b)}{1+\sqrt{\mu(\cdot,b-)}}
\quad\text{in }L^2_{\mathrm{loc}}(\mathbb R^2)
\end{equation}
as \(s\to\infty\).
}
The corresponding formula holds for the second configuration.
{In particular, \eqref{eq:parabolic-face-trace} and \eqref{eq:parabolic-face-limit} give equality of these limits for almost every \(x'\).}

We finally combine this information with the recovered quotient.
The equality in \eqref{eq:parabolic-quotient} also holds outside \(\Omega\), where both sources vanish.
{By Fubini's theorem, the quotient equality holds in \(L^2(\mathbb R^2)\) for almost every \(x_3\in(a,b)\).
The continuity of \(\tilde\phi\) and the one-sided continuity and uniform positive bounds of the diffusivities allow passage to \(x_3=b-\) in \(L^2(\mathbb R^2)\) by dominated convergence, as in \eqref{eq:TATPAT_quotient_face_limit}.
Indeed, the absolute values of \(p(\cdot)\tilde\phi(x_3)/\mu(\cdot,x_3)\) are bounded by a constant times \(|p|\), and similarly for the second configuration.
Taking the limit along slices on which the quotient equality holds, we obtain, for almost every \(x'\in\mathbb R^2\),}
\begin{equation}\label{eq:parabolic-face-relations}
\begin{aligned}
\frac{p(x')\tilde\phi(b)}{1+\sqrt{\mu(x',b-)}}
&=\frac{\tilde p(x')\tilde\phi(b)}{1+\sqrt{\tilde\mu(x',b-)}},\\
\frac{p(x')\tilde\phi(b)}{\mu(x',b-)}
&=\frac{\tilde p(x')\tilde\phi(b)}{\tilde\mu(x',b-)} .
\end{aligned}
\end{equation}
If the common value in the second equality of \eqref{eq:parabolic-face-relations} is zero, positivity of the diffusivities and \(\tilde\phi(b)\ne0\) imply \(p(x')=\tilde p(x')=0\).
Otherwise, division of the first equality by the second gives
\[
\frac{\mu(x',b-)}{1+\sqrt{\mu(x',b-)}}
=\frac{\tilde\mu(x',b-)}{1+\sqrt{\tilde\mu(x',b-)}}.
\]
The function \(\mu/(1+\sqrt\mu)\) is strictly increasing for \(\mu>0\).
Hence \(\mu(x',b-)=\tilde\mu(x',b-)\), and substitution into the quotient equality gives \(p(x')=\tilde p(x')\).
{Both cases therefore give \(p=\tilde p\) almost everywhere in \(\mathbb R^2\).}
Since the source profile is common, it follows that \(f=\tilde f\) almost everywhere in \(\Omega\).

It remains to determine the diffusivity on the support of the source.
Substituting \(f=\tilde f\) into \eqref{eq:parabolic-quotient} gives \((\mu-\tilde\mu)f=0\) almost everywhere.
{If the continuous difference \(\mu-\tilde\mu\) were nonzero at a point of \(\operatorname{supp}(f)\) inside the strip, it would be bounded away from zero in a neighbourhood of that point.
The product equality would then force \(f=0\) almost everywhere there, contradicting the definition of the support of an \(L^2\) function.
Thus the diffusivities agree on this part of the source support.}
The remaining points of \(\operatorname{supp}f\) lie on the two end planes, which have measure zero.
Thus
\[
f=\tilde f\quad\text{a.e. in }\Omega,\qquad
\mu=\tilde\mu\quad\text{a.e. on }\operatorname{supp}(f).\qedhere
\]
\end{proof}
\endgroup

\subsection{\texorpdfstring{{Proof for the Schr\"odinger model}}{Proof for the Schr\"odinger model}}

\begingroup
\begin{proof}[Proof of Theorem~\ref{ext:schrodinger-model}]
We first use an adjoint test to recover the source \(f\), and then use free tests to recover \(V\) on the support of \(f\).

\subsubsection{Step 1: Two integral identities}

Let \(\psi\) and \(\tilde\psi\) be the solutions of \eqref{eq:schrodinger-potential} corresponding to \((f,V)\) and \((\tilde f,\tilde V)\), respectively.
For real bounded \(V\), the self-adjoint operator \(-\Delta+V\) has domain \(H^2(\mathbb R^3)\).
Its evolution preserves the \(L^2\) norm, so \(\|\psi(t)\|_2=\|f\|_2\), and likewise for the second configuration.
The time integral of \(\psi\) is continuous in \(H^2\), since applying \(-\Delta+V\) to it gives \(\mathrm i(\psi(t)-f)\), and it vanishes initially.
It also grows at most linearly in \(H^2\).
In the weak sense, the observation \eqref{eq:schrodinger-observation} is the distributional time derivative of the boundary trace of \(\int_0^t\psi(\cdot,r)\,\mathrm dr\).
Equality of the weak observations therefore gives equality of the boundary traces of the two time integrals, as in the weak-trace construction of Subsection~\ref{subsec:weak-boundary-data}.

For \(\operatorname{Im}k>0\), define \(\check\psi(k)=\int_0^\infty e^{\mathrm ikt}\psi(t)\,\mathrm dt\).
Integration of the equation gives
\begin{equation}\label{eq:schrodinger-transform}
(-\Delta+V-k)\check\psi(k)=-\mathrm i f,
\qquad
\check\psi(k)=-\mathrm i(-\Delta+V-k)^{-1}f.
\end{equation}
In particular, \(\check\psi(k)\in H^2(\mathbb R^3)\).
Integrating the time primitive by parts shows that its boundary trace determines the trace of \(\check\psi(k)\).
Thus the transforms for the two configurations have equal Dirichlet traces.
Their difference satisfies \((-\Delta-k)(\check\psi-\check{\tilde\psi})=0\) in the exterior because both sources and potentials vanish there.
The imaginary part of the exterior energy identity, with zero Dirichlet data and \(\operatorname{Im}k>0\), gives equality of the exterior fields.
Global \(H^2\) regularity then gives equality of the interior normal derivatives.
This is the transformed version of the Cauchy-data reduction in Remark~\ref{rem:other_boundary_flux}.
The bilinear Green identities therefore apply to the separated tests below, including when the initial state jumps at \(a,b\).

The adjoint test \(w\) used in Step 2 has an \(H^2(\Omega)\) spatial factor and the decaying time factor \(e^{\mathrm ikt}\), with \(\operatorname{Im}k>0\).
It satisfies the formal adjoint Schr\"odinger equation for the bilinear pairing without complex conjugation,
\[
-\mathrm{i}\partial_t w(x,t)+\Delta w(x,t)-V(x)w(x,t)=0,
\qquad
(x,t)\in\Omega\times\mathbb R_+,
\]
and the \(L^2\) bounds give integrable pairings and a vanishing terminal pairing against \(\psi-\tilde\psi\).
Integration by parts as in the proof of Lemma~\ref{lem:integral_identity} gives
\begin{equation}\label{eq:schrodinger-firstID}
\int_0^\infty\!\!\int_\Omega
(V(x)-\tilde V(x))\tilde\psi(x,t)w(x,t)\,\mathrm{d}x\,\mathrm{d}t
+
\mathrm{i}\int_\Omega (f(x)-\tilde f(x))w(x,0)\,\mathrm{d}x
=0.
\end{equation}
For this test, the identity can equivalently be obtained from \eqref{eq:schrodinger-transform} for the two configurations and their equal Cauchy data.

For the second identity, we write
\[
\mathrm{i}\partial_t\psi+\Delta\psi=V\psi
\]
and test the difference of the two equations against
\[
w_k(x,t)=e^{\mathrm{i}kt}v(x),
\qquad
-\mathrm{i}\partial_t w_k+\Delta w_k=0.
\]
We take \(\operatorname{Im}k>0\) and \((-\Delta-k)v=0\), so the time factor decays and the terminal term vanishes.
The same Green formula gives
\begin{multline}\label{eq:schrodinger-secondID}
-\mathrm{i}\int_\Omega (f(x)-\tilde f(x))v(x)\,\mathrm{d}x
\\ =
\int_0^\infty\!\!\int_\Omega V(x)\psi(x,t)w_k(x,t)\,\mathrm{d}x\,\mathrm{d}t
-
\int_0^\infty\!\!\int_\Omega \tilde V(x)\tilde\psi(x,t)w_k(x,t)\,\mathrm{d}x\,\mathrm{d}t.
\end{multline}

\subsubsection{Step 2: Recovery of \(f\)}

Fix a transverse frequency \(m\in\mathbb R^2\), let \(\kappa\to+\infty\), and take
\begin{equation}\label{eq:schrodinger-free-phase}
k=|m|^2+(\kappa+\mathrm i)^2,\qquad
v(x)=e^{\mathrm im\cdot x'+\mathrm i(\kappa+\mathrm i)x_3}.
\end{equation}
Then \((-\Delta-k)v=0\), \(\operatorname{Im}k=2\kappa\), and \(v\) is bounded uniformly on \(\Omega\).
All limits in this argument are taken with \(m\) fixed.
Self-adjointness gives the resolvent bound
\begin{equation}\label{eq:schrodinger-resolvent-bound}
\|(-\Delta+V-k)^{-1}\|_{L^2\to L^2}\le\frac1{2\kappa}.
\end{equation}
For a fixed \(g\in L^2(\mathbb R^3)\), the spectral theorem gives
\begin{equation}\label{eq:schrodinger-fixed-data}
\|(-\Delta+V-k)^{-1}g\|_2=o(\kappa^{-1}).
\end{equation}
Indeed, on the real spectral axis the multiplier \(\kappa/(\lambda-k)\) is bounded by \(1/2\) and tends to zero at every fixed \(\lambda\), so dominated convergence applies to the spectral measure of \(g\).
The same argument applies to \(\bar k\) and to \(\tilde V\).
We use \eqref{eq:schrodinger-fixed-data} for fixed data and \eqref{eq:schrodinger-resolvent-bound} for data depending on \(\kappa\).

Use \eqref{eq:schrodinger-firstID} with the adjoint test
\[
w(x,t)=e^{\mathrm ikt}\bigl[v-(-\Delta+V-k)^{-1}(Vv)\bigr](x).
\]
Since \((-\Delta+V-k)v=Vv\), applying this operator to the spatial factor gives zero.
The factor belongs to \(H^2(\Omega)\), and \eqref{eq:schrodinger-resolvent-bound} gives
\[
\|w(\cdot,0)\|_{L^2(\Omega)}
\le\|v\|_{L^2(\Omega)}+\frac1{2\kappa}\|Vv\|_2\le C_m.
\]
Here \(Vv\) is uniformly bounded in \(L^2(\mathbb R^3)\), since \(V\) is bounded and compactly supported.
Performing the time integration in the first term of \eqref{eq:schrodinger-firstID}, we obtain
\[
\begin{aligned}
&\left|\int_0^\infty\!\!\int_\Omega(V-\tilde V)\tilde\psi\,w\,\mathrm dx\,\mathrm dt\right|\\
&\quad\le\|V-\tilde V\|_\infty
\|\check{\tilde\psi}(k)\|_2\|w(\cdot,0)\|_{L^2(\Omega)}
=o(\kappa^{-1}),
\end{aligned}
\]
because \(\check{\tilde\psi}(k)=-\mathrm i(-\Delta+\tilde V-k)^{-1}\tilde f\) by \eqref{eq:schrodinger-transform} and \eqref{eq:schrodinger-fixed-data} applies to the fixed datum \(\tilde f\).
For the corrector in the initial term, duality moves the resolvent onto the fixed source difference.
Since the pairing is bilinear, the Hilbert space adjoint acts on its complex conjugate, and \eqref{eq:schrodinger-fixed-data} at \(\bar k\) gives
\[
\begin{aligned}
&\left|\int_\Omega(f-\tilde f)(-\Delta+V-k)^{-1}(Vv)\,\mathrm dx\right|\\
&\quad\le
\|(-\Delta+V-\bar k)^{-1}\overline{(f-\tilde f)}\|_2\|Vv\|_2
=o(\kappa^{-1}).
\end{aligned}
\]
Substitution into \eqref{eq:schrodinger-firstID} therefore gives \(\int_\Omega(f-\tilde f)v\,\mathrm dx=o(\kappa^{-1})\).
Using the separated forms in Theorem~\ref{ext:schrodinger-model}, this becomes
\begin{equation}\label{eq:schrodinger-source-product}
\left|\int_{\mathbb R^2}(p-\tilde p)e^{\mathrm im\cdot x'}\,\mathrm dx'\right|
\left|\int_a^b\tilde\phi(t)e^{\mathrm i(\kappa+\mathrm i)t}\,\mathrm dt\right|
=o(\kappa^{-1}).
\end{equation}
Two integrations by parts give the longitudinal integral as
\[
\frac{\tilde\phi(b)e^{\mathrm i(\kappa+\mathrm i)b}-\tilde\phi(a)e^{\mathrm i(\kappa+\mathrm i)a}}
     {\mathrm i(\kappa+\mathrm i)}+O(\kappa^{-2}).
\]
If \(\tilde\phi(a)=0\), the numerator has the constant positive modulus \(|\tilde\phi(b)|e^{-b}\).
Otherwise, choose a sequence \(\kappa\to\infty\) on which the two quantities being subtracted have arguments differing by \(\pi\).
Such a sequence exists because their relative argument changes by \(\kappa(b-a)\), with \(b>a\).
Along this sequence the numerator has modulus \(|\tilde\phi(b)|e^{-b}+|\tilde\phi(a)|e^{-a}\).
Thus in either case, for some \(c>0\) and all sufficiently large \(\kappa\) along a sequence,
\begin{equation}\label{eq:schrodinger-profile-lower-bound}
\left|\int_a^b\tilde\phi(t)e^{\mathrm i(\kappa+\mathrm i)t}\,\mathrm dt\right|
\ge\frac c\kappa.
\end{equation}
Dividing \eqref{eq:schrodinger-source-product} by the lower bound in \eqref{eq:schrodinger-profile-lower-bound} makes the transverse Fourier coefficient tend to zero.
That coefficient is independent of \(\kappa\), so it vanishes for each fixed \(m\).
Fourier uniqueness gives \(p=\tilde p\), and the common profile then gives \(f=\tilde f\).

\subsubsection{Step 3: Recovery of \(V\)}

The two sources now have the same transverse factor \(p\).
Keeping the free test \(e^{\mathrm ikt}v\), with \(k,v\) from \eqref{eq:schrodinger-free-phase}, we perform the time integration in \eqref{eq:schrodinger-secondID} and use \(f=\tilde f\) to obtain
\begin{equation}\label{eq:schrodinger-potential-transform}
\int_\Omega\bigl[V\check\psi(k)-\tilde V\check{\tilde\psi}(k)\bigr]v\,\mathrm dx=0.
\end{equation}
We need an approximation of the fields that remains accurate after this identity is multiplied by \(k\).
Write \(\check\psi=\mathrm i f/k+\mathbf R\).
At a jump of the zero-extended profile, \(\mathbf R\) contains endpoint terms of the same order as \(\mathrm i f/k\).
To retain these terms, we solve the free equation in the longitudinal variable exactly.
Let \(h_k\) be the \(L^2\) solution of \((-\partial_3^2-k)h_k=\tilde\phi\) on the line.
With \(\operatorname{Im}\sqrt k>0\), the free resolvent kernel gives
\[
h_k(z)=\frac{\mathrm i}{2\sqrt k}\int_a^b e^{\mathrm i\sqrt k|z-t|}\tilde\phi(t)\,\mathrm dt .
\]
Here \(\sqrt k=\kappa+\mathrm i+O_m(\kappa^{-1})\), so \(\operatorname{Im}\sqrt k\to1\).
For \(z\in I\), split the integral at \(z\) to obtain
\[
\begin{aligned}
h_k(z)=\frac{\mathrm i}{2\sqrt k}\biggl(
&e^{\mathrm i\sqrt k z}\int_a^z e^{-\mathrm i\sqrt k t}\tilde\phi(t)\,\mathrm dt\\
&+e^{-\mathrm i\sqrt k z}\int_z^b e^{\mathrm i\sqrt k t}\tilde\phi(t)\,\mathrm dt\biggr).
\end{aligned}
\]
Integrating both terms by parts twice gives, uniformly on \(I\),
\begin{equation}\label{eq:schrodinger-profile-expansion}
k h_k(z)=-\tilde\phi(z)
+\tfrac12\tilde\phi(a)e^{\mathrm i\sqrt k(z-a)}
+\tfrac12\tilde\phi(b)e^{\mathrm i\sqrt k(b-z)}+O_m(\kappa^{-1}).
\end{equation}
The terms involving \(\tilde\phi'(z)\) cancel between the two integrals.
The remaining error is bounded by \(C_m\kappa^{-1}\) times \(|\tilde\phi'(a)|+|\tilde\phi'(b)|+\|\tilde\phi''\|_{L^1(I)}\).
Since \(|k|\) is comparable to \(\kappa^2\), \eqref{eq:schrodinger-profile-expansion} gives \(\sup_{[a,b]}|h_k|=O_m(\kappa^{-2})\).
For \(z>b\), the kernel formula gives \(h_k(z)=h_k(b)e^{\mathrm i\sqrt k(z-b)}\), with the corresponding decaying expression for \(z<a\).
As \(\operatorname{Im}\sqrt k\) stays bounded away from zero, the exterior \(L^2\) norms are controlled by the endpoint values.
Consequently,
\[
\|h_k\|_{L^2(\mathbb R)}=O_m(\kappa^{-2}).
\]

We use \(-\mathrm i p h_k\) to approximate the full transformed field.
The regularity of \(p\) in Theorem~\ref{ext:schrodinger-model} and the one-dimensional resolvent give \(p h_k\in H^2(\mathbb R^3)\).
Substituting it into the spatial equation yields
\[
(-\Delta+V-k)(p h_k)-f=(-\Delta_{x'}p+Vp)h_k .
\]
The \(L^2(\mathbb R^3)\) norm of the right-hand side is at most
\[
\bigl(\|\Delta_{x'}p\|_{L^2(\mathbb R^2)}
+\|V\|_\infty\|p\|_{L^2(\mathbb R^2)}\bigr)
\|h_k\|_{L^2(\mathbb R)}=O_m(\kappa^{-2}).
\]
Combining this residual equation with \eqref{eq:schrodinger-transform} gives the exact relation
\[
\check\psi(k)+\mathrm i p h_k
=\mathrm i(-\Delta+V-k)^{-1}\bigl[(-\Delta_{x'}p+Vp)h_k\bigr].
\]
The resolvent bound \eqref{eq:schrodinger-resolvent-bound} contributes another factor \((2\kappa)^{-1}\).
Applying the same argument to the second configuration therefore gives
\begin{equation}\label{eq:schrodinger-field-error}
\|\check\psi(k)+\mathrm i p h_k\|_2+
\|\check{\tilde\psi}(k)+\mathrm i p h_k\|_2=O_m(\kappa^{-3}).
\end{equation}

Insert the field approximations from \eqref{eq:schrodinger-field-error} into \eqref{eq:schrodinger-potential-transform}.
The Cauchy--Schwarz inequality and the uniform \(L^2\) bounds for \(Vv\) and \(\tilde Vv\) give
\[
\left|\int_\Omega(V-\tilde V)p h_k v\,\mathrm dx\right|\le C_m\kappa^{-3}.
\]
Since \(V-\tilde V=(B-\tilde B)\theta\) and \(|k|=O_m(\kappa^2)\), multiplication by \(k\) and separation of variables yield
\begin{equation}\label{eq:schrodinger-potential-product}
\left|\int_{\mathbb R^2}(B-\tilde B)p\,e^{\mathrm im\cdot x'}\,\mathrm dx'\right|
\left|k\int_a^b\theta(z)h_k(z)e^{\mathrm i(\kappa+\mathrm i)z}\,\mathrm dz\right|
\le C_m\kappa^{-1}.
\end{equation}
It remains to show that the longitudinal factor in \eqref{eq:schrodinger-potential-product} stays bounded away from zero.
Substituting the expansion \eqref{eq:schrodinger-profile-expansion} gives
\[
\begin{aligned}
&k\int_a^b\theta(z)h_k(z)e^{\mathrm i(\kappa+\mathrm i)z}\,\mathrm dz\\
&\quad=-\int_a^b\theta(z)\tilde\phi(z)e^{\mathrm i(\kappa+\mathrm i)z}\,\mathrm dz\\
&\qquad+\tfrac12\tilde\phi(a)e^{-\mathrm i\sqrt k a}
\int_a^b\theta(z)e^{\mathrm i(\sqrt k+\kappa+\mathrm i)z}\,\mathrm dz\\
&\qquad+\tfrac12\tilde\phi(b)e^{\mathrm i\sqrt k b}
\int_a^b\theta(z)e^{\mathrm i(\kappa+\mathrm i-\sqrt k)z}\,\mathrm dz
+O_m(\kappa^{-1}).
\end{aligned}
\]
The first integral tends to zero by the Riemann--Lebesgue lemma because \(\theta\tilde\phi e^{-z}\in L^1(I)\).
For the term from \(a\), the real frequency \(\operatorname{Re}\sqrt k+\kappa\) tends to infinity, while the exponential weight converges uniformly to \(e^{-2z}\) on \(I\).
The same lemma therefore makes that integral tend to zero, and its prefactor remains bounded.
For the term from \(b\), we have \(\kappa+\mathrm i-\sqrt k=O_m(\kappa^{-1})\).
Its exponential tends uniformly to one, so its integral tends to \(\int_a^b\theta(z)\,\mathrm dz\).
Thus
\[
 k\int_a^b\theta(z)h_k(z)e^{\mathrm i(\kappa+\mathrm i)z}\,\mathrm dz
 =\tfrac12\tilde\phi(b)e^{\mathrm i\sqrt k b}\int_a^b\theta(z)\,\mathrm dz+o(1).
\]
Since \(\operatorname{Im}\sqrt k\to1\), the modulus satisfies
\begin{equation}\label{eq:schrodinger-profile-nondegenerate}
\lim_{\kappa\to\infty}
\left|k\int_a^b\theta(z)h_k(z)e^{\mathrm i(\kappa+\mathrm i)z}\,\mathrm dz\right|
=\tfrac12|\tilde\phi(b)|e^{-b}\left|\int_a^b\theta(z)\,\mathrm dz\right|>0.
\end{equation}
This limit is positive by the assumptions \(\tilde\phi(b)\ne0\) and \(\int_I\theta\ne0\) in Theorem~\ref{ext:schrodinger-model}.
The lower bound supplied by \eqref{eq:schrodinger-profile-nondegenerate} allows division in \eqref{eq:schrodinger-potential-product}.
Letting \(\kappa\to\infty\) gives a zero transverse Fourier coefficient for each fixed \(m\).
Hence \((B-\tilde B)p=0\) by Fourier uniqueness.
This product is continuous, so it vanishes pointwise and \(B=\tilde B\) wherever \(p\ne0\).
Continuity of \(B,\tilde B\) extends their equality to \(\operatorname{supp}p\).
Finally, \(V-\tilde V=(B-\tilde B)\theta\) and \(\operatorname{supp}f\subset\operatorname{supp}p\times[a,b]\).
Thus
\[
f=\tilde f\quad\text{a.e. in }\Omega,\qquad
V=\tilde V\quad\text{a.e. on }\operatorname{supp}(f).\qedhere
\]
\end{proof}
\endgroup

\Needspace{10\baselineskip}
\appendix
\section{Proofs of the technical lemmas in Section~\ref{sec:generalhyperbolic}}\label{app:general_technical_proofs}

\subsection{Proof of Lemma~\ref{lem:odd_high_frequency_limit}}\label{app:odd_high_frequency_limit}

\begin{proof}
The local H\"older estimate for uniformly elliptic equations with bounded measurable coefficients gives a continuous representative of \(\hat u\), including across the interfaces; see \cite[Theorem~8.24]{gilbarg2001elliptic}.

Choose \(\delta>0\) so that the cube of half-side \(\delta\) centred at \((r,b_2,b_3)\) meets no other face of \(\Omega\).
{In the coordinates of Definition~\ref{def:general_scaled_coordinates}, equation \eqref{eq:scaled_odd_equation} holds in the cube \(\{|y_j|<\delta\,\mathrm{i}k,\ j=1,2,3\}\).}
The mean value theorem bounds the right-hand side by a constant multiple of \(|y_2|+|y_3|\), independently of \(k\).
As \(\mathrm{i}k\to+\infty\), the coefficients and the source converge locally in \(L^\infty\) to the coefficients and source in \eqref{eq:frozen_odd_equation}.

For large \(\mathrm{i}k\), estimate \eqref{eq:laplace_bound_general} and the definition \eqref{eq:odd_field_general} give only \(|U_{k,r}|\le C\,\mathrm{i}k\).
To obtain a bound independent of \(k\) on each fixed bounded set, we use a standard barrier argument based on the weak maximum principle \cite[Theorem~8.1]{gilbarg2001elliptic}.
For a sufficiently large constant \(C\), independent of \(k\), the comparison principle gives
\begin{equation}\label{eq:odd_growth_bound}
|U_{k,r}(y)|
\le C\sqrt{C^2+|y'|^2}
+C\,\mathrm{i}k\sum_{j=1}^3
\frac{\cosh(y_j/C)}{\cosh(\delta\,\mathrm{i}k/C)}
\end{equation}
throughout this cube for all sufficiently large \(\mathrm{i}k\).
To verify the comparison, apply the operator on the left-hand side of \eqref{eq:scaled_odd_equation} to the right-hand side of \eqref{eq:odd_growth_bound}.
The Laplacian of \(\sqrt{C^2+|y'|^2}\) is bounded above by \(2C^{-2}\sqrt{C^2+|y'|^2}\).
The positive zeroth-order term therefore makes the first term a supersolution for the absolute value of the source when \(C\) is large.
For each hyperbolic cosine term, the second derivatives contribute a factor \(C^{-2}\), while the only nonzero derivative of the principal coefficient {in the interior} is \(O((\mathrm{i}k)^{-1})\).
Thus these terms are supersolutions of the homogeneous equation for all sufficiently large \(\mathrm{i}k\).
Their normal derivatives, and those of the first term, vanish on \(y_2=0\) and \(y_3=0\), so the comparison is valid across both interfaces.
On the boundary of the cube, the sum in \eqref{eq:odd_growth_bound} is at least one, and \eqref{eq:laplace_bound_general} gives \(|U_{k,r}|\le C\,\mathrm{i}k\).
This proves \eqref{eq:odd_growth_bound} by applying the comparison principle to \(U_{k,r}\) and \(-U_{k,r}\).

The second term in \eqref{eq:odd_growth_bound} tends to zero on every fixed compact set.
On each fixed ball, this gives a uniform bound for \(U_{k,r}\).
The right-hand side of \eqref{eq:scaled_odd_equation} is also uniformly bounded, and the principal coefficients are uniformly elliptic.
The local energy estimate and the same H\"older estimate therefore have constants independent of \(k\), so every sequence with \(\mathrm{i}k\to+\infty\) has a subsequence converging weakly in \(H^1_{\mathrm{loc}}(\mathbb R^3)\) and locally uniformly.
Passing to the limit in the weak equation shows that every such limit satisfies \eqref{eq:frozen_odd_equation}, with derivatives taken in all three variables, and grows at most linearly in \(y'\).

We apply the weak maximum principle \cite[Theorem~8.1]{gilbarg2001elliptic} on expanding cubes, using the hyperbolic cosine barriers from \eqref{eq:odd_growth_bound}.
For the difference of two solutions of at most linear growth, the boundary contribution tends to zero on every fixed compact set, which gives uniqueness for the three-dimensional limiting equation.
Since the coefficients and the source do not depend on \(y_1\), translation in \(y_1\) gives another solution with the same growth.
Uniqueness therefore implies that the limit is independent of \(y_1\), so it has the form \(U_r(y')\) and solves \eqref{eq:frozen_odd_equation} in \(\mathbb R^2\).
The same uniqueness proves convergence of the whole family \(U_{k,r}\).

Finally, setting \(y_1=y_2=0\) in \eqref{eq:odd_field_general} and using local uniform convergence gives \eqref{eq:odd_boundary_limit_general}.
The argument applies in the same way to the second configuration.

The proof is complete.
\end{proof}

\subsection{Proof of Lemma~\ref{lem:odd_CGO_pairing_general}}\label{app:odd_CGO_pairing_general}

\begin{proof}
We first establish a local two-sided bound to ensure that the leading angular mode has a nonzero coefficient.
{By the endpoint condition in Admissibility~\ref{ass:general_structure}, \(f_0'(b_2)\ne0\), and it suffices to consider \(f_0'(b_2)>0\), replacing \(U_r\) and the right-hand side of \eqref{eq:frozen_odd_equation} by their negatives if necessary.}
Uniqueness in Lemma~\ref{lem:odd_high_frequency_limit} and the antisymmetry of the forcing in \eqref{eq:frozen_odd_equation} give \(U_r(y_2,y_3)=-U_r(y_3,y_2)\).
Thus \(U_r\) has zero diagonal trace, and its forcing is nonnegative and nonzero in \(y_3>y_2\).
The comparison on expanding domains used with \eqref{eq:odd_growth_bound} in the proof of Lemma~\ref{lem:odd_high_frequency_limit} gives \(U_r\ge0\) there, and the strong maximum principle \cite[Theorem~8.19]{gilbarg2001elliptic} gives \(U_r>0\), including the interface.

To determine the power in this bound, use polar coordinates \((\rho,\theta)\) measured from the positive diagonal, with \(0<\theta<\pi\) in \(y_3>y_2\) and \(0<\theta<\pi/4\) in its intersection with \(\mathcal Q\).
Seeking a separated solution \(\rho^{\alpha_r}\phi_r\) for the homogeneous principal part of \eqref{eq:frozen_odd_equation}, with zero diagonal trace and matched value and angular flux at \(\theta=\pi/4\), gives
\begin{equation}\label{eq:odd_angular_root_general}
\sigma_0(r)\cot(\pi\alpha_r/4)+\cot(3\pi\alpha_r/4)=0,
\end{equation}
whose root in \((2/3,4/3)\) we denote by \(\alpha_r\); define \(\tilde\alpha_r\) in the same way with \(\tilde\sigma_0(r)\).
The left-hand side of \eqref{eq:odd_angular_root_general} is strictly decreasing from a positive value to \(-\infty\) on this interval, so the root is unique.
The corresponding positive angular mode is
\[
\phi_r(\theta)=
\begin{cases}
\sin(\alpha_r\theta)/\sin(\pi\alpha_r/4),&0<\theta<\pi/4,\\
\sin(\alpha_r(\pi-\theta))/\sin(3\pi\alpha_r/4),
&\pi/4<\theta<\pi.
\end{cases}
\]
This mode has zero diagonal trace and continuous value and matched angular flux at \(\theta=\pi/4\); its positivity identifies it as the first angular mode.

To compare \(U_r\) with \(\rho^{\alpha_r}\phi_r\), we use the following identity, valid for \(p>0\):
\begin{equation}\label{eq:odd_barrier_identity_general}
\begin{aligned}
&\bigl[-\nabla\cdot(\sigma_r\nabla)+1+(c_0(r)^{-2}-1)\chi_{\mathcal Q}\bigr]
(\rho^p\phi_r)\\
&\qquad=\sigma_r(\alpha_r^2-p^2)\rho^{p-2}\phi_r
+\bigl[1+(c_0(r)^{-2}-1)\chi_{\mathcal Q}\bigr]\rho^p\phi_r.
\end{aligned}
\end{equation}
The forcing in \eqref{eq:frozen_odd_equation} is bounded by \(C\rho\phi_r\).
In a sufficiently small half-disc, \eqref{eq:odd_barrier_identity_general} gives a lower comparison function \((\rho^{\alpha_r}+K\rho^{\alpha_r+2})\phi_r\) for large \(K\), and an upper one \((M\rho^{\alpha_r}-K'\rho^3)\phi_r\) when \(K'\) dominates the forcing.
Positivity bounds \(U_r/\phi_r\) above and away from zero on a small outer semicircle except possibly at its endpoints on the diagonal.
Near these endpoints the coefficients are constant, so local regularity and the boundary point lemma \cite[Lemma~3.4]{gilbarg2001elliptic} give the same bounds there.
We may therefore choose a small positive multiple of the lower comparison function and a large \(M\) in the upper one to compare their boundary values with \(U_r\).
Both belong to \(H^1\), have matched fluxes and create no point source at the vertex, since their flux around \(\rho=\epsilon\) tends to zero.
Restoring the sign of \(f_0'(b_2)\), weak comparison gives
\begin{equation}\label{eq:odd_power_bounds_general}
c_*\rho^{\alpha_r}\phi_r
\le \operatorname{sgn}(f_0'(b_2))U_r
\le C\rho^{\alpha_r}\phi_r
\qquad(0<\rho<R_0,\ 0<\theta<\pi).
\end{equation}
The lower bound in \eqref{eq:odd_power_bounds_general} prevents the leading coefficient from vanishing, while the upper bound controls the rescaling below.
The tilded estimate follows in the same way.
We extend the angular modes oddly across the diagonal.

We next identify the field that enters the leading gradient pairing.
The tilded bound \eqref{eq:odd_power_bounds_general} suggests the normalization \(s^{\tilde\alpha_r}\tilde U_r(y'/s)\), which keeps the leading power at fixed size.
Since \(\mathcal Q\) is invariant under dilation, rescaling the tilded equation \eqref{eq:frozen_odd_equation} leaves \(\tilde\sigma_r\) fixed and gives
\[
\begin{aligned}
&-\nabla\cdot\bigl(\tilde\sigma_r\nabla[s^{\tilde\alpha_r}\tilde U_r(y'/s)]\bigr)
\quad+s^{-2}\bigl[1+(\tilde c_0(r)^{-2}-1)\chi_{\mathcal Q}\bigr]
[s^{\tilde\alpha_r}\tilde U_r(y'/s)]\\
&\qquad=s^{\tilde\alpha_r-3}\tilde c_0(r)^{-2}\tilde f_0'(b_2)(y_3-y_2)\chi_{\mathcal Q}.
\end{aligned}
\]
The tilded estimate \eqref{eq:odd_power_bounds_general} becomes, for \(0<\rho<R_0s\) and \(0<\theta<\pi\),
\[
c_*\rho^{\tilde\alpha_r}\tilde\phi_r
\le\operatorname{sgn}(\tilde f_0'(b_2))s^{\tilde\alpha_r}\tilde U_r(y'/s)
\le C\rho^{\tilde\alpha_r}\tilde\phi_r.
\]
Antisymmetry extends the upper bound for the absolute value to the reflected half-plane.
The displayed bound, together with \(s^{-2}\to0\) and \(s^{\tilde\alpha_r-3}\to0\), yields
\[
-\nabla\cdot\bigl(\tilde\sigma_r\nabla[s^{\tilde\alpha_r}\tilde U_r(\cdot/s)]\bigr)
\longrightarrow0\qquad\text{in }L^2_{\mathrm{loc}}(\mathbb R^2).
\]
The local energy and H\"older estimates \cite[Problem~8.2 and Theorem~8.24]{gilbarg2001elliptic} give, along a sequence \(s_j\to\infty\), locally uniform and weak \(H^1_{\mathrm{loc}}\) convergence to a solution of the homogeneous transmission equation.
Testing the difference equations with a cutoff times the difference upgrades this to strong \(H^1_{\mathrm{loc}}\) convergence, since \(\tilde\sigma_r\) is fixed and the lower-order terms vanish.

For every fixed \(\rho>0\), the displayed two-sided bound applies for all sufficiently large \(j\), so the limit retains it at every radius.
The limit is antisymmetric, so we expand it in angular transmission eigenfunctions with zero endpoint traces.
Each angular mode gives radial factors \(\rho^p\) and \(\rho^{-p}\), where \(p\ge\tilde\alpha_r\), with equality only for the first mode \(\tilde\phi_r\).
The inherited upper bound excludes all negative powers at zero and all powers above \(\tilde\alpha_r\) at infinity.
The lower bound in the tilded version of \eqref{eq:odd_power_bounds_general} makes the remaining real multiple of \(\rho^{\tilde\alpha_r}\tilde\phi_r\) nonzero.

To pass from local convergence to the pairing on the unbounded quadrant, we next control its tail uniformly in \(s\).
For \(|y'|<R_0s\), use the tilded upper bound in \eqref{eq:odd_power_bounds_general}; for \(|y'|\ge R_0s\), use the linear growth in Lemma~\ref{lem:odd_high_frequency_limit} and \(s\le |y'|/R_0\).
Since \(\tilde\alpha_r<4/3\), this gives
\[
\left|s^{\tilde\alpha_r}\tilde U_r(y'/s)\right|
\le C(1+|y'|)^{4/3}\qquad(s\ge s_0),
\]
for sufficiently large \(s_0\), with \(C\) independent of \(s\).
The forcing in the rescaled equation is also bounded by \(C|y'|\), uniformly in \(s\).
As in the cutoff argument in the proof of Proposition~\ref{prop:edge_test_identity}, the Caccioppoli estimate \cite[Problem~8.2]{gilbarg2001elliptic} gives polynomial bounds for the gradient energy on annuli, now uniformly in \(s\), after dropping the nonnegative mass term.
By \eqref{eq:first_odd_test_general}, \(p_s/s=1+\mathrm{i}\) and \(q_s/s\to1-\mathrm{i}\), so the scaled tests and their first derivatives decay exponentially in \(\mathcal Q\), uniformly in \(s\).
Cauchy--Schwarz on dyadic annuli therefore yields
\[
\sup_{s\ge s_0}
\left|\int_{\mathcal Q\cap\{|y'|>R\}}
\nabla[s^{\tilde\alpha_r}\tilde U_r(y'/s)]\cdot\nabla[v_s^{(1)}(y'/s)]\right|
\le C(1+R)^4e^{-cR}\longrightarrow0.
\]
This uniform tail estimate and the local strong convergence permit passage to the limit in
\[
s_j^{\tilde\alpha_r}
\int_{\mathcal Q}\nabla\tilde U_r\cdot\nabla v_{s_j}^{(1)}
=\int_{\mathcal Q}\nabla[s_j^{\tilde\alpha_r}\tilde U_r(y'/s_j)]\cdot
\nabla[v_{s_j}^{(1)}(y'/s_j)].
\]

It remains to verify that this limit does not vanish through cancellation in the complex pairing.
The scaled tests in \eqref{eq:first_odd_test_general} converge to the difference of exponentials below.
Since \(\rho^{\tilde\alpha_r}\tilde\phi_r\) is harmonic in \(\mathcal Q\), Green's formula reduces the pairing to integrals on the two coordinate half-axes, which the Gamma integral evaluates as
\[
\begin{aligned}
&\int_{\mathcal Q}\nabla(\rho^{\tilde\alpha_r}\tilde\phi_r)
\cdot\nabla\bigl[
 e^{-(1+\mathrm i)y_2-(1-\mathrm i)y_3}
-e^{-(1-\mathrm i)y_2-(1+\mathrm i)y_3}\bigr]\\
&\qquad=4\mathrm i\,2^{-\tilde\alpha_r/2}
\Gamma(\tilde\alpha_r+1)\cos(\pi\tilde\alpha_r/4)\ne0.
\end{aligned}
\]
The small-circle boundary term is \(O(\epsilon^{\tilde\alpha_r+1})\), and the large-circle term decays exponentially.
Since the limiting mode has a nonzero coefficient, there is a constant \(B_r\in\mathbb C\setminus\{0\}\) such that
\[
s_j^{\tilde\alpha_r}
\int_{\mathcal Q}\nabla\tilde U_r\cdot\nabla v_{s_j}^{(1)}
\longrightarrow B_r.
\]
Discarding finitely many terms gives \eqref{eq:odd_nonzero_gradient_general}.

Finally, to prove \eqref{eq:odd_lower_order_pairings_general}, we compare the three remaining pairings with this nonzero leading gradient term.
The local bound \eqref{eq:odd_power_bounds_general}, the linear growth in Lemma~\ref{lem:odd_high_frequency_limit}, and the exponential decay in \eqref{eq:first_odd_test_general} give
\[
\int_{\mathcal Q}U_rv_s^{(1)}=O(s^{-\alpha_r-2}),\qquad
\int_{\mathcal Q}\tilde U_rv_s^{(1)}=O(s^{-\tilde\alpha_r-2}),
\]
while direct integration of the same test gives
\[
\int_{\mathcal Q}(y_3-y_2)v_s^{(1)}
=\frac{2(p_s-q_s)}{p_s^2q_s^2}=O(s^{-3}).
\]
The exponent range established with \eqref{eq:odd_angular_root_general} gives \(\tilde\alpha_r-\alpha_r-2<0\) and \(\tilde\alpha_r-3<0\).
Hence all three pairings are \(o(s^{-\tilde\alpha_r})\).
Together with the nonzero normalized gradient limit \(B_r\), this proves \eqref{eq:odd_lower_order_pairings_general}.

The proof is complete.
\end{proof}

\subsection{Completion of the proof of Lemma~\ref{lem:one_face_expansion}}\label{app:one_face_expansion}

\begin{proof}
We justify the normal approximation used in Lemma~\ref{lem:one_face_expansion}.
Let \(v(0)\) be given by \eqref{eq:face_normal_moment}, with \(r=x_1\) and \(y=\mathrm{i}k(b_3-x_3)\).
Fix a small cube centred at an arbitrary point \((x_1,x_2,b_3)\), with \(x_1\in(a_1,b_1)\) and \(x_2\in(a_2,b_2)\), and meeting no other face of \(\Omega\).
{Use the normal profile computed in the proof of Lemma~\ref{lem:one_face_expansion}, with \(r=x_1\) and \(y=\mathrm{i}k(b_3-x_3)\), together with its exterior continuation \(v(0)\exp(y)\), to define \(u^{\mathrm{app}}(x,k)\) in the cube.}
The value and physical conormal flux match at the face, so the residual has no interface measure.

For the equation obtained by multiplying \eqref{eq:FourierTransform_recalled} by \(\mathrm{i}k\), the normal derivative, mass, and source terms cancel exactly.
The remaining tangential derivatives are uniformly bounded: the stated regularity bounds derivatives through order two, and quadratic factors in \(y\) arising from the exponential are absorbed by its decay.
{The profile formula bounds \(u^{\mathrm{app}}\) uniformly, and \eqref{eq:laplace_bound_general} gives the same bound for \((\mathrm{i}k)\hat u\).}

{Apply the comparison argument for \eqref{eq:odd_growth_bound} in Appendix~\ref{app:odd_high_frequency_limit} to both signs of \((\mathrm{i}k)\hat u-u^{\mathrm{app}}\).}
The positive mass term \(-k^2c^{-2}\) controls the bounded residual by a constant of order \(|k|^{-2}\).
The hyperbolic cosine barriers from that proof, scaled by \(|k|\) with a sufficiently small fixed rate, control the artificial boundary values and have zero normal derivative at the interface.
They decay exponentially on the concentric half-size cube, so we obtain there
\[
\bigl|(\mathrm{i}k)\hat u-u^{\mathrm{app}}\bigr|\le C|k|^{-2}.
\]
Substituting \eqref{eq:face_normal_moment} into the profile formula and using a finite covering proves \eqref{eq:one_face_expansion} on the claimed inward neighbourhood.

Finally, we approximate the exact test \(v_k\) from \eqref{eq:general_exact_tests} by \(v_0(x_1,x_2,b_3)v^{(2)}(y)\), with \(r=x_1\).
This product has the prescribed trace on \(x_3=b_3\), vanishes on the side faces, and has an exponentially small trace on the opposite face.
As above, the normal derivatives cancel the mass term; the remaining tangential residual retains a fixed exponential decay in \(y\).

With a sufficiently slow normal exponential rate, the same positive mass comparison bounds the difference from \(v_k\) by \(C|k|^{-2}\exp(-y/C)+C\exp(-|k|/C)\).
For complex data, apply this comparison to the real and imaginary parts.
The first term integrates to \(O(|k|^{-3})\), and the exponentially small second term is absorbed into this bound, proving \eqref{eq:normal_test_error_general}.

The proof is complete.
\end{proof}

\begingroup

\Needspace{10\baselineskip}
\section{Examples concerning the observation time}\label{app:observation-time-examples}

We prove Propositions~\ref{prop:observation-time-obstruction} and~\ref{prop:persistent-boundary-echoes} by finite propagation and an explicit calculation of reflected pulses, respectively.

\subsection{A finite propagation obstruction to a uniform observation time}
\label{subsec:observation-time-obstruction}

\begin{proof}[Proof of Proposition~\ref{prop:observation-time-obstruction}]
Write \(B'(0,r)\) for the ball in \(\mathbb R^{n-1}\), and choose nonnegative, nonzero functions
\[
 F\in C_c^\infty(B'(0,\tfrac14)),\qquad
 \phi\in C_c^\infty((-\tfrac14,\tfrac14)),
\]
with \(\phi=1\) on \([-\tfrac18,\tfrac18]\).
Set \(f_1(x)=F(x')\phi(x_n)\) and \(f_2=2f_1\).
These sources are fixed, supported in \(B'(0,\tfrac14)\times(-\tfrac14,\tfrac14)\), and independent of \(x_n\) on \(B'(0,\tfrac14)\times(-\tfrac18,\tfrac18)\).
Let \(Q_{\rm slow}=B'(0,1)\times(-1,1)\), and fix a smooth function \(0\leq\chi\leq1\) equal to one on \(\overline{Q_{\rm slow}}\) and compactly supported in \(B'(0,\tfrac32)\times(-\tfrac32,\tfrac32)\Subset\Omega\).
For a prescribed \(T>0\), define
\[
 a_T=\frac{3}{8(1+T)},\qquad
 c_{j,T}=1-(1-ja_T)\chi,\quad j=1,2.
\]
The speeds equal \(ja_T\) on \(Q_{\rm slow}\) and one outside \(\operatorname{supp}\chi\).
Since \(0<2a_T<1\), both perturbations have the fixed support \(\operatorname{supp}\chi\), and \eqref{eq:general_bounds} holds with \(c_-=a_T\) and \(c_+=1\).
On the source supports the speeds are constant, so
\[
 \frac{f_1(x)}{c_{1,T}(x)^2}
 =\frac{F(x')\phi(x_n)}{a_T^2},\qquad
 \frac{f_2(x)}{c_{2,T}(x)^2}
 =\frac{F(x')\phi(x_n)}{2a_T^2}.
\]
Thus the two quotients are distinct and have the same longitudinal profile.

To locate the signals, for a smooth speed \(c\) bounded above and away from zero define
\[
 d_c(E_1,E_2)=\inf_{\xi}\int_0^1\frac{|\dot\xi(s)|}{c(\xi(s))}\,\mathrm ds,
\]
where \(\xi\) ranges over absolutely continuous paths joining the nonempty sets \(E_1,E_2\).
For a nonzero real \(f\in C_c^\infty(\mathbb R^n)\), put \(r(x)=d_c(\operatorname{supp}f,\{x\})\).
Then \(r\) is Lipschitz and \(c|\nabla r|\leq1\) almost everywhere.
For the corresponding solution with zero initial velocity, the energy density
\[
 e=\tfrac12(c^{-2}|u_t|^2+|\nabla u|^2)
\]
satisfies \(\partial_t e=\nabla\cdot(u_t\nabla u)\) and \(|u_t\nabla u\cdot\nabla r|\leq e\).
Choose a smooth bounded nondecreasing function \(\eta\), with bounded derivative, zero on \((-\infty,0]\) and positive on \((0,\infty)\).
Integration by parts gives
\[
 \frac{\mathrm d}{\mathrm dt}\int_{\mathbb R^n}\eta(r-t)e\,\mathrm dx
 =-\int_{\mathbb R^n}\eta'(r-t)(e+u_t\nabla u\cdot\nabla r)\,\mathrm dx\leq0.
\]
The calculation is justified by the Lipschitz chain rule and spatial cutoffs, which can be removed using finite energy.
The integral vanishes initially and hence at every later time.
If \(r(x)>t\), it follows that \(u_t(x,s)=0\) for \(0\leq s\leq t\), while \(f(x)=0\).
Consequently,
\begin{equation}\label{eq:observation-finite-propagation}
 \operatorname{supp}u(\cdot,t)
 \subset\{x:d_c(\operatorname{supp}f,\{x\})\leq t\}.
\end{equation}

Every path from either source support to \(\Gamma\) must leave \(Q_{\rm slow}\), and its initial portion inside this region has Euclidean length at least \(3/4\).
Since \(c_{j,T}\leq2a_T\) there, its travel time is at least \(3/(8a_T)\), giving
\[
 d_{c_{j,T}}(\operatorname{supp}f_j,\Gamma)
 \geq\frac{3}{8a_T}=1+T,\qquad j=1,2.
\]
This bound exceeds \(T\), so continuity of the travel distance and \eqref{eq:observation-finite-propagation} show that both fields vanish in a neighbourhood of \(\Gamma\) for \(0\leq t\leq T\).
In particular, their boundary observations and normal derivatives vanish throughout this interval.

The support condition \eqref{eq:general_support} holds by construction.
In dimension three, the nonnegative profile \(\phi\) has strictly positive Laplace moments in \eqref{eq:nonde}, so the quotients satisfy Admissibility~\ref{ass:x3_invariance_cf}.
Their distinctness and Theorem~\ref{thm0} then imply that the two full-time observations differ.
\end{proof}

Throughout this construction, only the sound speeds depend on \(T\); the sources, profile and spatial supports are fixed.

The example also explains the terminal terms in the finite-time version of Lemma~\ref{lem:integral_identity}.
Since the boundary flux vanishes up to time \(T\), conservation of energy gives
\[
 \int_\Omega\!\left(c_{j,T}^{-2}|\partial_tu_j(x,T)|^2+|\nabla u_j(x,T)|^2\right)\mathrm dx
 =\int_\Omega|\nabla f_j|^2\,\mathrm dx>0.
\]
A zero boundary record therefore does not justify discarding the interior terminal state in that identity.

If one lower bound \(c_->0\) is fixed for the whole class, a straight segment up to its first intersection with the boundary gives
\[
 d_c(\operatorname{supp}f,\Gamma)
 \leq\frac{\operatorname{diam}(\Omega)}{c_-}
\]
for every nonzero source supported in \(\Omega\).
This bounds the minimum travel time from the source support to the boundary, while the duration of the boundary signal requires separate information, as Proposition~\ref{prop:persistent-boundary-echoes} shows.

\subsection{Persistent boundary echoes for a fixed positive speed}
\label{subsec:persistent-boundary-echoes}

\begin{proof}[Proof of Proposition~\ref{prop:persistent-boundary-echoes}]
Set
\[
 c(x)=\begin{cases}
 \frac12,&|x|<1,\\
 1,&|x|\geq1,
 \end{cases}
\]
and choose nonnegative, nonzero functions
\[
 F\in C_c^\infty(B'(0,1/4)),\qquad
 \phi\in C_c^\infty((-1/4,1/4)),
\]
with \(\phi\) supported on an interval, and set \(f(x)=F(x')\phi(x_3)\).
Then \(\operatorname{supp}f\Subset B(0,1/2)\), so \(f/c^2=4F\phi\) on all of \(\mathbb R^3\), including the zero extension in \eqref{eq:qphi1}.
Both separating factors belong to \(L^2\), and the nonnegative profile \(\phi\) has the strictly positive Laplace moments required in \eqref{eq:nonde}.
Together with \(1/2\leq c\leq1\) and \(\operatorname{supp}(1-c)=\overline{B(0,1)}\Subset\Omega\), this verifies the configuration hypotheses of Theorem~\ref{thm0}.

Define the spherical averages
\[
 \overline f(r)=\frac{1}{4\pi}\int_{\mathbb S^2}f(r\omega)\,\mathrm dS(\omega),\qquad
 \overline u(r,t)=\frac{1}{4\pi}\int_{\mathbb S^2}u(r\omega,t)\,\mathrm dS(\omega).
\]
The separating condition concerns the original source quotient; the average is used only to detect a nonzero component of the observation.
Since \(c\) is radial, averaging commutes with the weak wave equation.
Hence \(v(r,t)=r\overline u(r,t)\) satisfies
\[
 c(r)^{-2}\partial_t^2v-\partial_r^2v=0,\qquad
 v(0,t)=0,\qquad
 v(r,0)=r\overline f(r),\quad \partial_tv(r,0)=0.
\]
For a piecewise smooth solution, continuity of \(\overline u\) and \(\partial_r\overline u\) removes the interface terms in the weak equation at \(r=1\).
Since \(v=r\overline u\), these conditions are equivalent to continuity of \(v\) and \(\partial_rv\).
The factor \(c^{-2}\) multiplies only the time derivative and introduces no spatial flux weight.

Define
\[
 A(s)=-\frac{s}{4}\,\overline f(|s|/2),\qquad s\leq2.
\]
This profile is smooth because \(\overline f\) has a smooth even extension, and it vanishes for \(s\leq-1\) and for \(1\leq s\leq2\).
The interior expression \(A(t-2r)-A(t+2r)\) satisfies the condition at the origin.
Matching it to an outgoing exterior wave, for which \(\partial_rv=-\partial_tv\), requires
\[
 (\partial_rv+\partial_tv)(1-,t)
 =-A'(t-2)-3A'(t+2)=0.
\]
Integration from \(0\) to \(t\), using \(A(-2)=A(2)=0\), therefore leads to the extension rule
\begin{equation}\label{eq:echo-profile-recursion}
 A(s+4)=-\frac13A(s),\qquad s\geq-2.
\end{equation}
This determines \(A\) successively on the remaining intervals, and its vanishing near \(s=\pm2\) makes the extension smooth.
We claim that
\[
 v(r,t)=\begin{cases}
 A(t-2r)-A(t+2r),&0\leq r\leq1,\\[2pt]
 \displaystyle\frac43A(t-r-1),&r\geq1.
 \end{cases}
\]
Both expressions solve the equation with the corresponding constant speed.
By \eqref{eq:echo-profile-recursion}, their interface values satisfy
\[
 \begin{aligned}
 v(1-,t)&=A(t-2)-A(t+2)=\frac43A(t-2)=v(1+,t),\\
 \partial_rv(1-,t)&=-2A'(t-2)-2A'(t+2)
 =-\frac43A'(t-2)=\partial_rv(1+,t).
 \end{aligned}
\]
The reflected term \(-A(t+2)\) at the interface equals one third of the incident term \(A(t-2)\).
Reflection at the origin reverses the sign, and travel to the interface and back takes four time units, giving the factor \(-1/3\) in \eqref{eq:echo-profile-recursion}.

On \([-2,2]\), \(A\) is odd and \(A'\) is even, so for \(0\leq r\leq1\),
\[
 v(r,0)=A(-2r)-A(2r)=r\overline f(r),\qquad
 \partial_tv(r,0)=A'(-2r)-A'(2r)=0.
\]
The exterior initial data vanish because \(-r-1\leq-2\) when \(r\geq1\).
Moreover, \(v/r\) extends smoothly at the origin and has finite energy and compact spatial support on every bounded time interval.
The matching identities therefore give an energy solution with initial displacement \(\overline f\) and zero initial velocity.
{By uniqueness in the weak framework of Subsection~\ref{subsec Forward Hy}, this field is \(\overline u\).}

At the observation radius, the spherical mean of the boundary trace is
\[
 G(t):=\overline u(3,t)=\frac13v(3,t)=\frac49A(t-4),\qquad t\geq0.
\]
Equation~\eqref{eq:echo-profile-recursion} now gives
\begin{equation}\label{eq:echo-boundary-recursion}
 G(t+4)=-\frac13G(t),\qquad t\geq2.
\end{equation}
For \(0<r<1/2\), the first pulse satisfies
\[
 G(4-2r)=\frac49A(-2r)=\frac{2r}{9}\overline f(r).
\]
Since \(f\geq0\) is nonzero, \(\overline f(r)>0\) for some such \(r\), and hence \(G(t_*)>0\) for some \(t_*\in(3,4)\).
Iterating \eqref{eq:echo-boundary-recursion} yields
\[
 G(t_*+4k)=\left(-\frac13\right)^kG(t_*)\ne0,
 \qquad k=0,1,2,\ldots.
\]
These times tend to infinity, and continuity of \(G\) gives a nonzero boundary observation after every \(T>0\).
\end{proof}

The interior jump in \(c\) is excluded by the continuity required in Theorem~\ref{thm:unique_f_c}.
For this fixed medium, \eqref{eq:echo-boundary-recursion} determines all later values of the spherical mean \(G\) from \(G|_{[2,6]}\).
\par\endgroup

\section*{Acknowledgement}

{The work of L. Chen was partly supported by the National Natural Science Foundation of China (No.~12622108) and a grant from Beijing Institute of Technology (No.~2022CX01002).}
The work of Y. Jiang was supported by the RGC project JRFS2627-1S06. The work of H. Liu was supported by the Hong Kong RGC General Research Funds (projects 11311122,
12301420 and 11300821). The work of C. W. K. Lo is supported by the National Natural Science Foundation of China (No. 12501660) and the Guangdong Province ``Pearl River" Talent Recruitment Program (Top youth talent) (No. 2024QN11X268). The authors would like to thank Dr. Yavar Kian for many constructive comments and helpful suggestions.

\section*{Statements and Declarations}

\noindent\textbf{Competing interests.} The authors declare no competing interests.

\medskip
\noindent\textbf{Data availability.} No datasets were generated or analysed during this study.

\bibliographystyle{plain}
\bibliography{ref}


\end{document}